\documentclass[leqno,11pt]{article}
\usepackage[colorlinks,pagebackref,hypertexnames=false]{hyperref}

\usepackage{amsmath,amsthm,amssymb,mathrsfs} 
\usepackage[alphabetic,backrefs]{amsrefs}
\usepackage[T1]{fontenc}
\usepackage{ae,aecompl}
\usepackage{times}
\usepackage{microtype}

\usepackage[english]{babel}

\usepackage{bookmark}

\numberwithin{equation}{section}

\newcommand{\SU}{{\rm SU}}

\newcommand{\Ad}{\mathrm{Ad}}

\newcommand{\End}{{\mathrm{End}}}

\renewcommand{\epsilon}{\varepsilon}

\newcommand{\Hom}{{\mathrm{Hom}}}

\newcommand{\diag}{\mathrm{diag}}

\newcommand{\dvol}{\mathop\mathrm{dvol}\nolimits}
\newcommand{\id}{\mathrm{id}}
\newcommand{\im}{\mathop{\mathrm{im}}}

\def\<{\mathopen{}\left<}
\def\>{\right>\mathclose{}}
\def\({\mathopen{}\left(}
\def\){\right)\mathclose{}}

\newtheorem{theorem}{Theorem}

\newtheorem{corollary}{Corollary}

\newtheorem{example}{Example}

\newtheorem{lemma}{Lemma}

\newtheorem{proposition}{Proposition}
\newtheorem{remark}{Remark}

\numberwithin{equation}{section}

\author{Jason D. Lotay \\ University College London \and  Goncalo Oliveira \\ Duke University}

\title{$SU(2)^2$-invariant $G_2$-instantons} 

\date{}

\begin{document}
\maketitle


 \begin{abstract}
\noindent We initiate the systematic study of $G_2$-instantons with $SU(2)^2$-symmetry.  As well as developing foundational theory, we give existence, non-existence and classification results for these instantons.  We particularly focus on $\mathbb{R}^4\times S^3$ with its two explicitly known distinct holonomy $G_2$ metrics, which have different volume growths at infinity, exhibiting the different behaviour of instantons in these settings.  We also give an explicit example of sequences of $G_2$-instantons where ``bubbling'' and ``removable singularity'' phenomena occur in the limit.
\end{abstract}

\tableofcontents

\section{Introduction}

In this article we study $G_2$-instantons: these are examples of Yang--Mills connections on Riemannian manifolds whose holonomy group is contained in the exceptional Lie group $G_2$ (so-called $G_2$-manifolds).  
These connections are, in a sense, analogues of anti-self-dual connections in dimension 4, and are likewise hoped to be used to understand the geometry and topology of $G_2$-manifolds, via the construction of enumerative invariants.  Our focus is on $G_2$-instantons on 
$G_2$-manifolds where both the connections and ambient $G_2$ geometry enjoy $SU(2)^2$-symmetry. In particular, as a $G_2$-manifold is Ricci flat, for it to admit continuous symmetries it must be noncompact.  By restricting to this case, we are able to 
shed light on the still rather poorly understood theory of $G_2$-instantons, in an explicit setting.  In particular, we give new existence and non-existence results for $G_2$-instantons. Furthermore, we can see how general theory works in practice, examine how the ambient geometry affects the $G_2$-instantons and give local models for the behaviour of $G_2$-instantons on compact $G_2$-manifolds.

\subsection[\texorpdfstring{$G_2$-instantons}{G2-instantons}]{{\boldmath $G_2$}-instantons}

Let $(X^7,\varphi)$ be a $G_2$-manifold\footnote{For further background on $G_2$-manifolds, the reader may wish to consult Joyce's book \cite{Joyce2000}.}, which implies the 7-manifold $X^7$ is endowed with a $3$-form $\varphi$ which is closed and determines a Riemannian metric $g$ with respect to which $\varphi$ is also coclosed. We shall denote $\ast \varphi$ by $\psi$ for convenience.  Let $P \rightarrow X$ be a principal bundle with structure group $G$ which we suppose to be a  compact and semisimple Lie group. A connection $A$ on $P$ is said to be a $G_2$-instanton if
\begin{equation}\label{eq:G2Instanton}
F_A \wedge \psi = 0.
\end{equation}
Equivalently, $G_2$-instantons satisfy the following $G_2$-analogue of the ``anti-self-dual''  condition:
\begin{equation}\label{eq:ASD}
F_A\wedge\varphi=-*F_A.
\end{equation}
As far as the authors are aware, the first time $G_2$-instantons appeared in the literature was in \cite{Corrigan1983}. This reference investigates generalizations of the anti-self-dual gauge equations, in dimension greater than $4$, and $G_2$-instantons appear there as an example.\\
More recently, the study of $G_2$-instantons has gained a special interest,  primarily due to Donaldson--Thomas' suggestion  \cite{Donaldson1998} that it may be possible to use $G_2$-instantons to define invariants for $G_2$-manifolds, inspired by Donaldson's pioneering work on anti-self-dual connections on  4-manifolds. Later Donaldson--Segal \cite{Donaldson2009}, Haydys \cite{Haydys2011}, and Haydys--Walpuski \cite{Haydys2015} gave further insights regarding that possibility.\\
On a compact holonomy $G_2$-manifold $(X^7, \varphi)$ any harmonic $2$-form is ``anti-self-dual'' as in \eqref{eq:ASD}, hence  any complex line bundle $L$ on $X$ admits a $G_2$-instanton, namely that whose curvature is the harmonic representative of $c_1(L)$.  However, the construction of non-abelian $G_2$-instantons on compact $G_2$-manifolds is much more involved. In the compact case, the first such examples were constructed by Walpuski  \cite{Walpuski2011}, over Joyce's $G_2$-manifolds (see \cite{Joyce2000}).  S\'a Earp and Walpuski's work  \cites{SaEarp2015 ,Walpuski2015} gives an abstract construction of $G_2$-instantons, and currently one example, on the other known class of compact $G_2$-manifolds, namely ``twisted connected sums'' (see \cites{Kovalev2003,Haskins}).\\
The goal of this paper is to perform a general analysis of $G_2$-instantons on some noncompact $G_2$-manifolds.  In the noncompact setting, the first examples of $G_2$-instantons where found by Clarke in \cite{Clarke14}, and further examples were given by the second author in \cite{Oliveira2014}. 
In this article we primarily study $G_2$-instantons on $\mathbb{R}^4 \times S^3$, which has two known complete and explicit $G_2$-holonomy metrics, namely: the Bryant-Salamon (BS) metric \cite{Bryant1989} and the Brandhuber et al.~(BGGG) metric \cite{Brandhuber2001}. Both these metrics have $\{0\} \times S^3$ as an associative submanifold: such area-minimizing submanifolds in $G_2$-manifolds have both known and expected relationships with $G_2$-instantons, so studying these metrics allows us to  verify known theory and test  expectations.  Of particular note is that the BS and BGGG metrics have different volume growths at infinity, and are in a sense analagous to the flat and Taub-NUT hyperk\"ahler metrics on $\mathbb{R}^4$. Our results exhibit the similarities and differences in the existence theory for $G_2$-instantons for these metrics.

\subsection{Summary}

The aim of the article is to start the systematic study of $SU(2)^2$-invariant $G_2$-instantons.  We now summarize the organization of our paper and the main results.\\
Both the BS and BGGG metric have $SU(2)^2$ as a subgroup of their isometry group: in fact, $SU(2)^2$ acts with cohomogeneity-$1$.  All known complete $SU(2)^2$-invariant $G_2$-manifolds of cohomogeneity-$1$ actually 
have $SU(2)^2\times U(1)$-symmetry. These facts are summarized in $\S$\ref{sec:SU(2)^2Equations}, where  we also deduce the ODEs for $SU(2)^2$-invariant $G_2$-instantons.  Since most of the known $SU(2)^2$-invariant $G_2$-manifolds are asymptotically locally conical (ALC) we prove some general results on the structure at infinity and asymptotic behaviour of $G_2$-instantons on ALC $G_2$-manifolds in $\S$\ref{ss:ALC}, relating them to Calabi--Yau monopoles.   In $\S$\ref{sec:solutions}, we give some explicit elementary solutions to the equations, namely flat connections and abelian ones.  Already in this simple abelian setting we see a marked difference between the $G_2$-instantons for the BS and BGGG metric.\\
In $\S$\ref{sec:SU(2)^3invariant} we focus on the BS metric, which has isometry group $SU(2)^3$. This also acts with cohomogeneity-$1$ and has a unique singular orbit which is the associative $S^3$. We describe $SU(2)^3$-invariant $G_2$-instantons with gauge group $SU(2)$. A dichotomy arises from the two possible homogeneous bundles over the associative $S^3$ on which the instantons can extend: let $P_1$ and $P_{\id}$ denote these two bundles.\\  In the $P_1$ case, by combining our study in $\S$\ref{sec:SU(2)^3invariant} with our work  in $\S$\ref{sec:SU(2)^2xU(1)invariant} we obtain our first main result.  

\begin{theorem}
Let $A$ be an irreducible $SU(2)^2\times U(1)$-invariant $G_2$-instanton with gauge group $SU(2)$ on the BS metric. If $A$ smoothly extends over $P_1$, then it is one of Clarke's $G_2$-instantons in \cite{Clarke14}.
\end{theorem}

\noindent See Theorems \ref{prop:Clarke} and \ref{thm:BS.cor} for more precise statements and an explicit formula for the instantons, and see Corollary \ref{cor:U1} for a classification of the reducible instantons. Here we mention that Clarke's $G_2$-instantons form a family $\lbrace A^{x_1} \rbrace$, parametrized by $x_1\geq 0$, and the curvature of these connections decays at infinity. \\
In the $P_{\id}$ case, we find (in Theorem \ref{thm:Alim}) a new explicit $G_2$-instanton $A^{\lim}$.  We show in Theorem \ref{thm:Compactness} and Corollary \ref{cor:delta} that $A^{\lim}$ is, in a certain (precise) sense, the limit of Clarke's ones as $x_1 \rightarrow + \infty$. We state our second main result informally, which confirms expectations from \cites{Tian2000,Tao2004}.

\begin{theorem}\label{thm:CompactnessUnprecise}
Let $\lbrace A^{x_1} \rbrace$ be a sequence of Clarke's $G_2$-instantons with $x_1 \rightarrow + \infty$.
\begin{itemize}
\item[(a)] After a suitable rescaling, the family $\lbrace A^{x_1} \rbrace$ bubbles off an anti-self-dual connection transversely to the associative $ S^3 = \lbrace 0 \rbrace \times S^3$. 
\item[(b)] The connections $A^{x_1}$ converge uniformly with all derivatives to $A^{\lim}$ on every compact subset of $(\mathbb{R}^4 \setminus \{0\})\times S^3$. 
\item[(c)] The functions $\vert F_{A^{x_1}} \vert^2 - \vert F_{A_{\lim}} \vert^2$ are integrable and converge to $8 \pi^2 \delta_{\lbrace 0 \rbrace \times S^3}$, where $\delta_{\lbrace 0 \rbrace \times S^3}$ denotes the delta current associated with the associative $S^3$.
\end{itemize} 
\end{theorem}

\noindent Whilst (a) gives the familiar ``bubbling'' behaviour of sequences of instantons, with curvature concentrating on an associative $S^3$ by (c), we can interpret (b) as a ``removable singularity'' phenomenon since $A^{\lim}$ is a smooth connection on $\mathbb{R}^4\times S^3$.   In proving Theorem \ref{thm:CompactnessUnprecise}, we show that as $\lbrace A^{x_1} \rbrace$ bubbles along the associative $S^3$ one obtains a Fueter section, as in \cites{Donaldson2009, Haydys2011, Walpuski2017}.  Here this is just a constant map  from $S^3$ to the moduli space of anti-self dual connections on $\mathbb{R}^4$ (thought of as a fibre of the normal bundle), taking value at the basic instanton on $\mathbb{R}^4$.  Since $8\pi^2$ is the Yang--Mills energy of the basic instanton, we can also view (c) as the expected  ``conservation of energy''. \\ 
We also give a local existence result for $G_2$-instantons in a neighbourhood of the associative $S^3$ that extend over $P_{\id}$ in Proposition \ref{prop:LocalExistenceBS_1}. The outcome is that there is a local one-parameter family of such instantons. Of these only one, i.e.~$A^{\lim}$, is shown to extend over the whole of $\mathbb{R}^4 \times S^3$. The other ones may blow up at a finite distance to $\lbrace 0 \rbrace \times S^3$, as suggested by numeric simulations.  Some of the necessary analysis 
leading to our local existence results is given in Appendix \ref{app:R4xS3}.\\
In order to use similar techniques for $G_2$-intantons on the BGGG metric, we must reduce the symmetry group to $SU(2)^2 \times U(1)$. This acts with cohomogeneity-$1$ both on BGGG and BS and, as before, its only singular orbit is the associative $\lbrace 0 \rbrace \times S^3$. Hence, in $\S$\ref{sec:SU(2)^2xU(1)invariant} we describe $SU(2)^2 \times U(1)$-invariant $G_2$-instantons on cohomogeneity-$1$ metrics with that symmetry on $\mathbb{R}^4 \times S^3$. As a result, the same dichotomy appears in that the $G_2$-instantons can extend over the associative $S^3$ either on the homogeneous bundle $P_1$ or $P_{\id}$. We can thus compare the existence of $G_2$-instantons for the BS and BGGG metrics. While there is a $1$-parameter family of $G_2$-instantons (Clarke's ones) that smoothly extend over $P_1$ on the BS metric, for the BGGG metric we instead have the following.

\begin{theorem}\label{th:NoG2Inst}
The moduli space $\mathcal{M}^{BGGG}_{P_1}$ 
of irreducible $SU(2)^2 \times U(1)$-invariant $G_2$-instantons with gauge group $SU(2)$ on the BGGG metric, smoothly extending on $P_1$, contains a nonempty (and unbounded) open set $U \subset \mathbb{R}^2$. Moreover, the following holds.
\begin{itemize}
\item[(a)] The instantons in $U$ have quadratically decaying curvature. 
\item[(b)] The map $Hol_{\infty} : U \rightarrow U(1) \subset SU(2)$, which evaluates the holonomy of the $G_2$-instanton along the finite size circle at $+\infty$, is surjective.
\end{itemize}
\end{theorem}

\noindent The more precise version of this result appears as Theorem \ref{thm:BGGG.Inst} and Corollary \ref{cor:Holonomy}. It is typical in gauge theory to assume a bound on the curvature of the connection.  One might be tempted to impose an $L^2$-bound, but this is too restrictive in the $G_2$ setting: in particular, Clarke's examples do not satisfy this.  Therefore, we impose a weak natural curvature bound in deriving Theorem \ref{th:NoG2Inst}, namely that the curvature stays bounded. We also prove  that there is a 2-parameter family of locally defined instantons on $P_1$ for the BGGG metric which do not extend globally with bounded curvature: this is Theorem \ref{thm:BGGG.NoInst}.\\
Finally, we give local existence results for $G_2$-instantons with $SU(2)^2 \times U(1)$-symmetry in a neighbourhood of an associative $S^3$, on any $SU(2)^2 \times U(1)$-invariant $G_2$-metric. In Proposition \ref{prop:LocalG2Instantons}, we show the existence of a $2$-parameter family of locally defined $G_2$-instantons smoothly extending over $P_1$, whereas in Proposition \ref{prop:LocalG2Instantons2} we show the existence of a $1$-parameter family of $G_2$-instantons smoothly extending over $P_{\id}$. 
 This yields the possibility for further study of $G_2$-instantons even on the well-known Bryant--Salamon metric on $\mathbb{R}^4\times S^3$.

\subsection*{Acknowledgments}

We would like to thank Tom Beale, Robert Bryant, Lorenzo Foscolo, Derek Harland, Mark Haskins, Thomas Madsen,  David Sauzin and Mark Stern for discussions. We would particularly like to thank Lorenzo Foscolo for introducing us to Eschenburg--Wang's analysis for extending invariant tensors over singular orbits \cite{Eschenburg2000}, and the reference \cite{Malgrange1974} for singular initial value problems.  The first author was partially supported by EPSRC grant EP/K010980/1.

\section[\texorpdfstring{The {\boldmath $SU(2)^2$}-invariant equations}{The SU(2)xSU(2)-invariant equations}]{The {\boldmath $SU(2)^2$}-invariant equations}\label{sec:SU(2)^2Equations}

In this section we derive the ordinary differential equations (ODEs) which describe invariant $G_2$-instantons on $SU(2)^2$-invariant $G_2$-manifolds of cohomogeneity-$1$. We begin by giving the general framework 
of the evolution equations approach to $G_2$-manifolds and $G_2$-instantons in $\S$\ref{ss:evol}.  We then apply this theory in $\S$\ref{sec:G2metrics} to the case of the invariant $G_2$-manifolds we wish to study, leading to systems of ODEs describing the $G_2$-manifolds, and summarise the known complete examples 
which arise from this approach.  Since most of these invariant $G_2$-manifolds are asymptotically locally conical (ALC), in $\S$\ref{ss:ALC} we give a brief discussion of the asymptotic behaviour of ALC $G_2$-manifolds and their $G_2$-instantons.  We then give a short presentation of the theory of invariant fields on homogeneous bundles in $\S$\ref{ss:hombdles} so that we can 
obtain the general expression for an invariant connection on a principal orbit and its curvature.  Combining these considerations yields our desired ODEs in $\S$\ref{ss:ODEs}, which we then solve in elementary cases in $\S$\ref{sec:solutions}.

\subsection{Evolution equations}\label{ss:evol}

In the work to be developed it is relevant to analyze the case when $X^7=I_t
\times M^6$ and $I_t \subset \mathbb{R}$ is an interval with coordinate $t\in\mathbb{R}$. Let $(\omega(t), \Omega_2(t))$ be a $1$-parameter family of $SU(3)$-structures parametrized by $t \in I_t$ and write the $G_2$-structure
\begin{equation}\label{eq:G2str}
\varphi = dt \wedge \omega(t) + \Omega_1(t) , \quad \psi= \frac{\omega^2(t)}{2} - dt \wedge \Omega_2(t),
\end{equation}
where $\Omega_1(t)=J_t \Omega_2(t)$ and $J_t$ is the almost complex structure determined by $\Omega_2(t)$.  (Recall that an $SU(3)$-structure on an almost complex $6$-manifold $(M,J)$ can be given by a pair of a real $(1,1)$-form $\omega$ and a real $3$-form $\Omega_2$ such that
$$\omega \wedge \Omega_2 =0 , \quad \omega^3 = - \frac{8}{3} \Omega_1 \wedge \Omega_2 , $$
where $\Omega_1=-J\Omega_2$.) This $G_2$-structure $\varphi$ is torsion-free  (i.e.~$d\varphi=0$ and $d\psi=0$) if and only if the $1$-parameter family $(\omega(t), \Omega_2(t))$ is a solution of the so-called ``Hitchin flow''\footnote{The nomenclature ``Hitchin flow'' is somewhat misleading since the system \eqref{eq:Hitchinflow} is not parabolic in any usual sense and it does not satisfy the typical regularity properties of geometric flows \cite{Bryant2010}.}, i.e.~if we write $\dot{f}=df/dt$, then
\begin{equation}\label{eq:Hitchinflow}
\dot{\Omega}_1= d\omega , \quad  \omega \wedge \dot{\omega}
= -d\Omega_2,
\end{equation}
subject to the constraints $d \Omega_1=0=d\omega^2$ for all $t$, which means that $(\omega(t), \Omega_2(t))$ is a family of half-flat $SU(3)$-structures
 solving \eqref{eq:Hitchinflow}.  (In fact, it is enough to impose the half-flat condition on the $SU(3)$-structure at some initial time and the evolution \eqref{eq:Hitchinflow} will then preserve this condition.)  For more on half-flat $SU(3)$-structures, in a case relevant to us, the reader can see \cite{Madsen2013}.
The resulting $G_2$-structure induces the metric $g=dt^2 + g_t$, where $g_t$ is the metric on $\lbrace t \rbrace \times M$ induced by the $SU(3)$-structure $(\omega(t), \Omega_2(t))$.\\
In this situation our bundle $P$ must be pulled back from $M$ and, working in temporal gauge, $A=a(t)$ is a $1$-parameter family of connections on $P$, so  $F_A = dt \wedge \dot{a} + F_a(t)$. Hence $A$ is a $G_2$-instanton, i.e.~solves \eqref{eq:G2Instanton}, if and only if
\begin{equation}\label{eq:evolution}
\dot{a} \wedge \frac{\omega^2}{2} - F_a \wedge \Omega_2 =0, \quad F_a \wedge \frac{\omega^2}{2} = 0.
\end{equation}
Using $\ast_t$ to denote the Hodge-$\ast$ associated with the $SU(3)$ structure $(\omega(t), \Omega_2(t))$ we have 
\begin{equation}
J_t \dot{a} = -\ast_t \left( \dot{a} \wedge \frac{\omega^2}{2}  \right)\quad\text{and}\quad\Lambda_t F_a = \ast_t \left( F_a \wedge \frac{\omega^2}{2} \right),
\end{equation}
where $\Lambda_t$ denotes the metric dual of the operation of wedging with $\omega$. 
Then, applying $\ast_t$ to both sides of \eqref{eq:evolution} we have
\begin{eqnarray}\label{eq:MEvolution1}
J_t \dot{a} & = & - \ast_t \left( F_a \wedge \Omega_2 \right), \\ \label{eq:MEvolution2}
\Lambda_t F_a & = & 0.
\end{eqnarray}

\begin{lemma}\label{lem:Constraint}
Let $X=I_t\times M$ be equipped with a $G_2$-structure $\varphi$ as in \eqref{eq:G2str} satisfying $\omega \wedge d \omega =0$ and $\omega \wedge \dot{\omega}=- d \Omega_2$, which is equivalent to $d\psi=0$. Then, $G_2$-instantons $A$ for $\varphi$ are in one-to-one correspondence with $1$-parameter families of connections $\lbrace a(t) \rbrace_{t \in I_t}$ solving the evolution equation
\begin{equation}\label{eq:IEvolution}
J_t \dot{a} =  - \ast_t \left( F_a \wedge \Omega_2 \right),
\end{equation}
subject to the constraint $\Lambda_t F_a =0$. Moreover, this constraint is compatible with the evolution: more precisely, if it holds for some $t_0 \in I_t$, then it holds for all $t \in I_t$.
\end{lemma}
\begin{proof}
The evolution equation and the constraint follow immediately from equations \eqref{eq:MEvolution1} and \eqref{eq:MEvolution2}. 
To prove that the constraint is preserved by the evolution we compute 
\begin{align*}
\frac{d}{dt} \left( F_a \wedge  \omega^2 \right) & = d_a \dot{a} \wedge \omega^2 + F_a \wedge \frac{d}{dt}\omega^2 =  d_a (\dot{a} \wedge \omega^2) -2 F_a \wedge d \Omega_2\\
& = 2 d_a(F_a \wedge \Omega_2) - 2 F_a \wedge d \Omega_2 =  0,
\end{align*}
where we used \eqref{eq:Hitchinflow}, \eqref{eq:evolution},  
\eqref{eq:IEvolution} and the Bianchi identity $d_a F_a=0$.
\end{proof}

\begin{proposition}
In the setting of Lemma \ref{lem:Constraint}, suppose that the family of $SU(3)$-structures $(\omega(t), \Omega_2(t))$ depends real analytically on $t$, and let $a(0)$ be a real analytic connection on $P$ such that $\Lambda_0 F_a(0)=0$. Then there is $\epsilon>0$ and a $G_2$-instanton $A$ on $(-\epsilon , \epsilon) \times M^6$ with $A\vert_{\lbrace 0 \rbrace \times M^6}=a(0)$.
\end{proposition}
\begin{proof}
This is immediate from applying the Cauchy-Kovalevskaya theorem to \eqref{eq:IEvolution}.
\end{proof}

\begin{remark}
We can similarly derive evolution equations defining $G_2$-monopoles, i.e.~pairs $(A, \Phi)$ where $A$ is a connection on $P$ and $\Phi$ is a section of the adjoint bundle satisfying
$$\ast \nabla_A \Phi = F_A \wedge \psi.$$
In this setting we can write $A= a(t)$ in temporal gauge as before and $\Phi=\phi(t) \in \Omega^0(I_t , \Omega^0(M, \mathfrak{g}_P))$ as a $1$-parameter family of Higgs fields over $M$. Then, the family $(a(t),\phi(t))$ of connections and Higgs fields on $M$ gives rise to a $G_2$-monopole if and only if they satisfy:
\begin{eqnarray}
J_t \dot{a} & = & -d_a \phi - \ast_t \left( F_a \wedge \Omega_2 \right), \nonumber \\ 
\dot{\phi} & = & \Lambda_t F_a. \nonumber
\end{eqnarray}
We do not pursue the analysis of these equations here, particularly since at least in the asymptotically conical case (the Bryant--Salamon $G_2$-manifolds) one does not expect to find non-trivial finite mass monopoles which are not instantons in our setting (see Theorem 2 in \cite{Oliveira2014}).
\end{remark}

\subsection[\texorpdfstring{$SU(2)^2$-invariant $G_2$-manifolds of cohomogeneity-$1$}{SU(2)xSU(2)-invariant G2-manifolds of cohomogeneity-1}]{{\boldmath $SU(2)^2$}-invariant {\boldmath $G_2$}-manifolds of cohomogeneity-1}\label{sec:G2metrics}

In this section we shall give a self-contained exposition of all the known complete $SU(2)^2$-invariant $G_2$-holonomy metrics. We shall see that all these examples actually have $SU(2)^2 \times U(1)$-symmetry. We start with some preparation. Split the Lie algebra $\mathfrak{su}(2) \oplus \mathfrak{su}(2)$ as $\mathfrak{su}^+ \oplus \mathfrak{su}^-$, as follows.
 If $\lbrace T_i \rbrace_{i=1}^3$ is a basis for $\mathfrak{su}(2)$ such that $[T_i,T_j] = 2 \epsilon_{ijk} T_k$, 
then $T_i^+ = (T_i, T_i)$ and $T_i^-=(T_i, -T_i)$ for $i=1,2,3$ give a basis for $\mathfrak{su}^+$ and $\mathfrak{su}^-$ respectively.  (Thus $\mathfrak{su}^+$ and $\mathfrak{su}^-$ are 
diagonal and anti-diagonal copies of $\mathfrak{su}(2)$ in $\mathfrak{su}(2)\oplus\mathfrak{su}(2)$.) We shall let $\lbrace \eta_i^+ \rbrace_{i=1}^{3}$ and $\lbrace \eta_i^- \rbrace_{i=1}^{3}$ be dual bases to $\{T_i^+\}_{i=1}^3$ and $\{T_i^-\}_{i=1}^3$ respectively. The Maurer--Cartan relations in this case give
\begin{eqnarray}
d\eta_i^+ & = & -\epsilon_{ijk} \left( \eta_j^+ \wedge \eta_k^+ + \eta_j^- \wedge \eta_k^- \right), \label{eq:MC1}\\ 
d\eta_i^- & = & -2\epsilon_{ijk} \eta_j^- \wedge \eta_k^+.\label{eq:MC2}
\end{eqnarray}
The complement of the singular orbit can be written as $\mathbb{R}^+_t \times M$, where $M$ 
denotes a principal orbit, which is a finite quotient of $S^3\times S^3$. 
The $SU(2)\times SU(2)$-invariant $SU(3)$-structure on the principal orbit $\{t\}\times M$ is given by (\cite{Madsen2013})
\begin{eqnarray}
\omega & = & 4 \sum_{i=1}^3 A_iB_i \eta_i^- \wedge \eta_i^+, \label{eq:BBSU(3)structure1}\\ 
\Omega_1 & = & 8 B_1B_2B_3 \eta_{123}^- - 4 \sum_{i,j,k}  \epsilon_{ijk} A_i A_j B_k \eta_i^+ \wedge \eta_j^+ \wedge \eta_k^-, \\ \label{eq:BBSU(3)structure}
\Omega_2 & = & - 8 A_1A_2A_3 \eta_{123}^+ + 4 \sum_{i,j,k}  \epsilon_{ijk} B_i B_j A_k \eta_i^- \wedge \eta_j^- \wedge \eta_k^+,\label{eq:BBSU(3)structure2}
\end{eqnarray}
for real-valued functions $A_i$, $B_i$ of $t \in \mathbb{R}^+$, where $\eta_{123}^{\pm}$
 denotes $\eta_{1}^{\pm} \wedge \eta_{2}^{\pm} \wedge \eta_{3}^{\pm}$. For future reference, we remark that 
$$4 \sum_{i,j,k}  \epsilon_{ijk} B_i B_j A_k \eta_i^- \wedge \eta_j^- \wedge \eta_k^+= 8 B_1 B_2 A_3 \eta_1^- \wedge \eta_2^- \wedge \eta_3^+ + \text{cyclic permutations}.$$ The compatible metric determined by this $SU(3)$ structure on $\{t\}\times M$ is (\cite{Madsen2013})
\begin{equation}\label{eq:metric}
g_t = \sum_{i=1}^3 (2A_i)^2 \eta_{i}^+ \otimes \eta_i^+ + (2B_i)^2 \eta_i^- \otimes \eta_i^-,
\end{equation}
and the resulting metric on $\mathbb{R}_t \times M$, compatible with the $G_2$-structure $\varphi = dt \wedge \omega + \Omega_1$, is given by $g=dt^2 + g_t$. Recall also that this metric has holonomy in $G_2$ if and only if the $SU(3)$-structure above solves the Hitchin flow equations \eqref{eq:Hitchinflow}.\\
  These considerations allow us to derive the general ODEs describing $SU(2)^2$-invariant $G_2$-manifolds of cohomogeneity-$1$ as follows (c.f.~\cite{Madsen2013}):
\begin{eqnarray}
\dot{A}_i&=&\frac{1}{2}\left(\frac{A_i^2}{A_jA_k}-\frac{A_i^2}{B_jB_k}-\frac{A_j^2+A_k^2}{A_jA_k}+\frac{B_j^2+B_k^2}{B_jB_k}\right),\label{eq:cohom.one.G2.metric1}\\
\dot{B}_i&=&\frac{1}{2}\left(\frac{A_j^2+B_k^2}{A_jB_k}+\frac{A_k^2+B_j^2}{A_kB_j}-\frac{B_i^2}{A_jB_k}-\frac{B_i^2}{A_kB_j}\right),\label{eq:cohom.one.G2.metric2}
\end{eqnarray}  
where $(i,j,k)$ denotes a cyclic permutation of $(1,2,3)$. 
   We will be interested in this article in the setting where we have known complete examples. In fact, in every such example there is an extra $U(1)$-symmetry: this $U(1)$ acts diagonally on $S^3 \times S^3$ with infinitesimal generator $T_1^+$.  As a consequence, we have $A_2=A_3$ and $B_2=B_3$ and \eqref{eq:Hitchinflow} becomes (as in \cite{Bazaikin2013}):
\begin{eqnarray}
\dot{A}_1 &= & \frac{1}{2}\left(\frac{A_1^2}{A_2^2}-\frac{A_1^2}{B_2^2}\right), \label{eq:dotA1}\\
 \dot{A}_2 &= & \frac{1}{2}\left(\frac{B_1^2+B_2^2-A_2^2}{B_1B_2}-\frac{A_1}{A_2}\right),\label{eq:dotA2}\\
\dot{B}_1 &= & \frac{A_2^2+B_2^2-B_1^2}{A_2B_2}, \label{eq:dotB1}\\
\dot{B}_2 &= &\frac{1}{2}\left(\frac{A_2^2+B_1^2-B_2^2}{A_2B_1}+\frac{A_1}{B_2}\right).\label{eq:dotB2}
\end{eqnarray}
We now give the known examples of cohomogeneity-$1$ complete $G_2$-metrics with $SU(2)^2$-symmetry. 

\subsubsection{The Bryant--Salamon (BS) metric}\label{sss:BS}

The Bryant--Salamon metric on $\mathbb{R}^4\times S^3$ \cite{Bryant1989} is one of the first examples of a complete metric with $G_2$-holonomy. It  
is not only $SU(2)^2$-invariant, but actually $SU(2)^3$-invariant, having group diagram $I(SU(2)^3, SU(2), SU(2)^2)$; i.e.~the principal orbits are $SU(2)^3/SU(2)\cong S^3\times S^3$ and the (unique) singular orbit is  $SU(2)^3/SU(2)^2\cong S^3$.  (Here, the $SU(2)$ in $SU(2)^3$ is the subgroup $SU(2)_3 = 1 \times 1 \times SU(2)$, and $SU(2)^2 \subset SU(2)^3$ is the subgroup $ SU(2)_3 \times \Delta SU(2)$, where $\Delta SU(2) \subset SU(2)^2 \times 1$ is the diagonal.) In terms of the $SU(2)^2$-invariant point of view above, the metric can be explicitly written as follows.\\
In this case the extra symmetry means that $A_1=A_2=A_3$ and $B_1=B_2=B_3$ and the equations \eqref{eq:dotA1}-\eqref{eq:dotB2} reduce to:
\begin{equation}\label{eq:dotAdotB}
\dot{A}_1=\frac{1}{2}\left(1-\frac{A_1^2}{B_1^2}\right)\quad\text{and}\quad
\dot{B}_1=\frac{A_1}{B_1}.
\end{equation}
Setting $B_1=s$ and $A_1=sC(s)$ we see that \eqref{eq:dotAdotB} becomes
$\frac{d}{ds}(sC)=\frac{1-C^2}{2C}$ 
which we can easily solve as 
$C(s)=\sqrt{\frac{1-c^3s^{-3}}{3}}$, 
so that, for $c>0$ and $s\geq c$,
\begin{equation}\label{eq:AB.c.BS}
A_1(s)=\frac{s}{\sqrt{3}}\sqrt{1-c^3s^{-3}}\quad\text{and}\quad 
B_1(s)=s.
\end{equation}
In particular, choosing $c=1$ and using $t$, the arc length parameter along the geodesic parametrized by $s$, we define a coordinate $r \in [1; + \infty)$ implicitly by 
\begin{equation}\label{eq:r.BS}
t(r)=\int_1^r \frac{ds}{\sqrt{1-s^{-3}}},
\end{equation} and solve \eqref{eq:dotAdotB} as follows:
\begin{equation}\label{eq:AB.BS}
A_1=A_2=A_3=\frac{r}{3}\sqrt{1-r^{-3}} \quad \text{and} \quad B_1=B_2=B_3=\frac{r}{\sqrt{3}}.
\end{equation}
It is easy to verify that the geometry at infinity is asymptotically conical to the standard holonomy $G_2$-cone on $S^3\times S^3$.  In fact, we see from \eqref{eq:AB.c.BS} that one obtains a one-parameter family\footnote{There are, in fact, distinct $SU(2)^3$-invariant torsion-free $G_2$-structures on $\mathbb{R}^4\times S^3$ inducing the same asymptotially conical Bryant--Salamon metric, determined by their image in $H^3(S^3\times S^3)$.} of solutions to \eqref{eq:dotAdotB}, equivalent up to scaling, whose limit with $c=0$ is the conical solution.  Moreover, the torsion-free $G_2$-structure has a unique compact associative submanifold which is the singular orbit $\lbrace 0 \rbrace \times S^3 \cong SU(2)^2/SU(2)$.\\ There is a one-parameter family of $SU(2)^3$-invariant $G_2$-instantons for this Bryant--Salamon torsion-free $G_2$-structure constructed by Clarke \cite{Clarke14}, where the parameter can be interpreted as how concentrated the instanton is around the associative $S^3$. We shall prove, in Theorem \ref{prop:Clarke} and Proposition \ref{prop:LocalG2Instantons}, a uniqueness result for these $G_2$-instantons in the class of $SU(2)^2 \times U(1)$-invariant ones.

\begin{remark}
In \cite{Bryant1989} Bryant--Salamon constructed $G_2$-holonomy metrics on the total spaces of the bundles of anti-self-dual $2$-forms over $\mathbb{CP}^2$ and $\mathbb{S}^4$, i.e.~$\Lambda^2_- \mathbb{CP}^2$ and $\Lambda^2_-\mathbb{S}^4$. Such metrics are also of cohomogeneity-$1$ with respect to $SO(5)$ and $SU(3)$ respectively and asymptotically conical. Instantons on these $G_2$-manifolds are also known to exist and some explicit examples can be found in \cite{Oliveira2014}.
\end{remark}

\noindent It follows from Proposition 3 in \cite{Oliveira2014} (or easily from \eqref{eq:MEvolution1}-\eqref{eq:MEvolution2}) that on an asymptotically conical $G_2$-manifold, a $G_2$-instanton whose curvature is decaying pointwise at infinity will have as a limit (if it exists) a pseudo-Hermitian--Yang--Mills connection $a_{\infty}$ (or nearly K\"ahler instanton): i.e.~if $\varphi_{\infty}=t^2dt\wedge\omega_{\infty}+t^3\Omega_{1,\infty}$ and $\psi_{\infty}=t^4\omega_{\infty}^2/2-t^3dt\wedge\Omega_{2,\infty}$ is the conical $G_2$-structure on the asymptotic cone then $F_{a_{\infty}}\wedge\omega_{\infty}^2=0$ and $F_{a_{\infty}}\wedge\Omega_{2,\infty}=0$.  

\subsubsection{The Brandhuber et al.~(BGGG) metric}\label{sss:BGGG}

On $\mathbb{R}^4 \times S^3$ there is another complete $G_2$-holonomy metric constructed by Brandhuber and collaborators in \cite{Brandhuber2001}, which is a member of a family of complete $SU(2)^2\times U(1)$-invariant, cohomogeneity-1,  $G_2$-holonomy metrics on $\mathbb{R}^4\times S^3$ found in \cite{Bogo2013}\footnote{We thank Lorenzo Foscolo and Mark Haskins for bringing the metrics in \cite{Bogo2013} to our attention.}.  To derive this example one can choose $c>0$, set $B_1=s$ and
$$A_1=c\frac{ds}{dt}=c\frac{A_2^2+B_2^2-s^2}{A_2B_2}$$
from \eqref{eq:dotB1}.  Letting $C_{\pm}=A_2^2\pm B_2^2$ the equations 
\eqref{eq:dotA2} and \eqref{eq:dotB2} yield
$$\frac{d}{ds}C_+=\frac{s^2C_+-C_-^2}{s(C_+-s^2)}\quad\text{and}\quad 
\frac{d}{ds}C_-=\frac{C_-}{s}-2c.$$
The second equation is easily integrated and so we are able to find solutions 
$$C_+(s)=\frac{3s^2-c^2}{2}\quad\text{and}\quad C_-(s)=-cs.$$
We thus obtain a one-parameter family of solutions to \eqref{eq:dotA1}-\eqref{eq:dotB2}:
\begin{gather}\label{eq:BGGG1} A_1(s)=2c\sqrt{\frac{s^2-c^2}{9s^2-c^2}},\quad A_2(s)=\frac{1}{2}\sqrt{(3s+c)(s-c)},\\
\label{eq:BGGG2}
B_1(s)=s,\quad B_2(s)=\frac{1}{2}\sqrt{(3s-c)(s+c)},
\end{gather}
defined for $s\geq c>0$. 
These solutions give holonomy $G_2$ metrics on $\mathbb{R}^4\times S^3$ by Lemma \ref{lem:ExtendingSmoothlyMetric} in Appendix \ref{app:R4xS3}.  
We can further scale the metric from $g$ to $\lambda^2 g$ and the resulting fields scale as $A_i^{\lambda}(s)=\lambda A_i (s/\lambda)$, $B_i^{\lambda}(s)=\lambda B_i(s/\lambda)$. These give the following family of solution to the ODEs \eqref{eq:dotA1}-\eqref{eq:dotB2} above:
\begin{gather}\nonumber A_1^{\lambda}(s)= 2c \lambda \sqrt{\frac{s^2-c^2\lambda^2}{9s^2-c^2\lambda^2}},\quad  A_2^{\lambda}(s)=\frac{1}{2}\sqrt{(3s+c\lambda)(s-c \lambda)},\\
\nonumber
B_1^{\lambda}(s)= s , \quad B_2^{\lambda}(s)=\frac{1}{2}\sqrt{(3s-c\lambda)(s+c\lambda)}.
\end{gather}
We see that under the scaling we have $c\mapsto c\lambda$, so we can always scale so that $c=1$.  In particular, one can set $\lambda=3/2$, $c=1$ and as in \cite{Brandhuber2001} define the coordinate $r\in [9/4,+\infty)$ implicitly by 
\begin{equation}\label{eq:r.BGGG}
t(r)=\int_{9/4}^{r} \frac{\sqrt{(s-3/4)(s+3/4)}}{\sqrt{(s-9/4)(s+9/4)}} ds
\end{equation} and find that
\begin{gather*}
A_1 =\frac{\sqrt{(r-9/4)(r+9/4)}}{\sqrt{(r-3/4)(r+3/4)}}, \quad A_2=A_3=\sqrt{\frac{(r-9/4)(r+3/4)}{3}},\\
B_1=\frac{2r}{3}, \quad B_2=B_3=\sqrt{\frac{(r-3/4)(r+9/4)}{3}}
\end{gather*}
solve \eqref{eq:dotA1}-\eqref{eq:dotB2}.  
We see in this case that the principal orbits are again $S^3\times S^3$ and the singular orbit $\lbrace 0 \rbrace \times S^3$ is associative. \\
In this setting, the geometry at infinity presents a new feature (that also exists in the BB manifolds below): there is a circle that remains of finite length at infinity. 
More precisely, the metric is asymptotic to a metric on a circle bundle over a $6$-dimensional cone with the fibres of the fibration having constant finite length. The length of this circle is  the limit of $A_1$ at infinity: for the family depending on the parameters $\lambda, c$ this is $2c\lambda /3$. One also sees that the volume of the associative $S^3$ is $B_1^{\lambda}(c\lambda^2)B_2^{\lambda}(c\lambda^2)^2 \sim (c \lambda)^3$, and so, using this family, it is impossible to vary the size of the circle while keeping the volume of the singular orbit fixed.\\  In \cite{Bogo2013}, Bogoyavlenskaya constructed a $1$-parameter family (up to scaling) of $SU(2)^2\times U(1)$-invariant, cohomogeneity-1, $G_2$-holonomy metrics on $\mathbb{R}^4\times S^3$, obtained by continuously deforming the BGGG metric. With these metrics, one can independently vary the size of the circle at infinity and the associative $S^3$, and thus, in particular, obtain the BS metric as a limit of the family.\\
The BGGG metric is the only one from \cite{Bogo2013} which is explicitly known. Choosing the scaling so that the circle at infinity has size $1$, for large $t$ we compute that $t(r) \sim r$, so
$$A_1 = 1 +O(t^{-2}), \ A_2 = \frac{t}{\sqrt{3}} +O(t^{-1}) , \ B_1 = \frac{2t}{3} + O(t^0), \ B_2 = \frac{t}{\sqrt{3}} + O(t^{-1}),$$
and thus the metric is asymptotic to
$$h=dt^2 + 4(\eta_1^+ )^2 + \frac{4t^2}{3} \left( (\eta_2^+ )^2 + ( \eta_3^+)^2 \right)+ 
\frac{16t^2}{9}(\eta_1^- )^2 +  \frac{4t^2}{3} \left( (\eta_2^- )^2 + (\eta_3^-)^2 \right).$$
This limit of the family of metrics given by \eqref{eq:BGGG1}-\eqref{eq:BGGG2} as $c\to 0$ is an $S^1$-bundle over a Calabi--Yau cone on the standard 
homogeneous Sasaki--Einstein metric on $S^2\times S^3$.  (We shall describe this Sasaki--Einstein structure explicitly in Example \ref{ex:S2S3}.)  This conical Calabi--Yau metric is also known as the conifold or 3-dimensional ordinary double point.

\subsubsection{The Bazaikin--Bogoyavlenskaya (BB) metrics}\label{sss:BB}

The Bazaikin--Bogoyavlenskaya $G_2$-manifolds $X$ \cite{Bazaikin2013} (BB manifolds for short) have group diagram $I(SU(2)^2; \mathbb{Z}_4 ; U(1))$, i.e.~the principal orbits are of the form $S^3 \times S^3 / \mathbb{Z}_4$ and the (unique) singular orbit is $ SU(2)^2 / U(1) \cong S^2 \times S^3$. In fact, $X$ is diffeomorphic to $L^4 \times S^3$, where $L \rightarrow S^2$ is the complex line bundle canonically associated with the Hopf bundle and $L^4$ denotes its fourth tensor power.\\
In \cite{Bazaikin2013} some complete torsion-free $G_2$-structures with an extra $U(1)$-symmetry, i.e.~with $A_2=A_3$ and $B_2=B_3$, are constructed. These structures give rise to a $1$-parameter family of holonomy $G_2$-metrics on $X=L^4 \times S^3$, which have $(A_1(0),A_2(0),B_1(0), B_2(0))= (  \mu, \lambda,0 , \lambda)$ for some values of $\lambda, \mu \in \mathbb{R}$ with $\lambda^2 + \mu^2 =1$.  In particular, the volume of the singular orbit $S^2 \times S^3$ is proportional to $\lambda^4 \mu =(1-\mu^2)^2 \mu$ and that of any $3$-sphere $\ast \times S^3$ is proportional to $\lambda^2 \mu = (1-\mu^2)\mu$: these 3-spheres are \emph{not} associative. At least some of these metrics are asymptotic to an $S^1$-bundle over the conifold. We further remark that if one thinks of the BS and BGGG metrics as somehow analogous to the flat and Taub-NUT metrics respectively, then the BB metrics are related to the Atiyah--Hitchin metric which is defined on the line bundle $L^4 \rightarrow S^2$.  One can make this correspondence between the holonomy $G_2$ and hyperk\"ahler metrics more precise by considering an ``adiabatic limit'' where the size of the 4-dimensional fibres tends to $0$.\\
We have some preliminary results on $G_2$-instantons on these $G_2$-manifolds and intend to investigate them further in future work.

\subsection[\texorpdfstring{Asymptotics of ALC $G_2$-manifolds and their instantons}{Asymptotics of ALC G2-manifolds and their instantons}]{Asymptotics of ALC {\boldmath $G_2$}-manifolds and their instantons}\label{ss:ALC}

We have seen in $\S$\ref{sss:BGGG}-\ref{sss:BB} examples of noncompact $G_2$-manifolds which are asymptotic at infinity to a circle bundle over a cone: such manifolds are called asymptotically locally conical (ALC).  Since we shall be studying $G_2$-instantons on the 
 ALC $G_2$-manifold $\mathbb{R}^4\times S^3$ with the BGGG metric in some detail (and the other metrics discussed in $\S$\ref{sss:BGGG}-\ref{sss:BB} in future work), 
 we present some general results on ALC $G_2$-manifolds here. 
 Specifically, we describe the induced structure on the asymptotic 
 circle bundle over a cone, and characterise the limits of 
 $G_2$-instantons with pointwise decaying curvature at infinity.

\subsubsection[\texorpdfstring{The $G_2$-structure}{The G2-structure}]{The {\boldmath $G_2$}-structure}\label{ss:ALC.structure}

A noncompact $G_2$-manifold $(X,\varphi)$ is said to be ALC if there is:
\begin{itemize}
\item a $U(1)$-bundle $\pi: \Sigma^6 \rightarrow M^5$ and a $U(1)$-invariant $G_2$-structure $\varphi_{\infty}$ on $(1,+ \infty) \times \Sigma$, whose associated metric is
$$g_{\varphi_{\infty}}=dr^2 + m^{2} \eta_{\infty}^2 + r^2 \pi^* g_5,$$
where $m \in \mathbb{R}^+$, $\eta_{\infty}$ is a connection on $\Sigma$ and $g_5$ a metric on $M$;
\item a compact set $K \subset X$ and (up to a double cover)\footnote{The possible need for the double cover is because $X$ may only be asymptotic to an $S^1$-bundle, but we can get a principal bundle by taking a double cover.} 
a diffeomorphism $p: (1,+ \infty)_r \times \Sigma \rightarrow X \backslash K$,
\end{itemize}
such that if $\nabla$ denotes the Levi-Civita connection of $g_{\varphi_{\infty}}$ then
\begin{equation}\label{eq:ALC.decay}
\vert \nabla^j ( \varphi_{\infty} - p^* \varphi \vert_{X \backslash K} 
)\vert_{g_{\varphi_{\infty}}} = O(r^{\nu-j}) \quad\text{ as }r\to+\infty,
\end{equation}
for some $\nu<0$ and $j=0,1$.  (We could make the stronger assumption that \eqref{eq:ALC.decay} holds for all $j\in\mathbb{N}$, but this is unnecessary for our purposes, and should in any case surely follow  from suitable weighted elliptic estimates given Proposition \ref{prop:SU(2)StructureLimit}.)\\ 
 Our next result describes the structure on $(1,+\infty)\times\Sigma$ 
 induced from the torsion-free $G_2$-structure $\varphi$ on $X$ and limits the range of rates $\nu$ to consider.

\begin{proposition}\label{prop:SU(2)StructureLimit}
Let $(X,\varphi)$ be an ALC $G_2$-manifold and use the notation above.
\begin{itemize}
\item[(a)]  If $\nu<0$, 
the metric $g_5$ is induced by a Sasaki--Einstein $SU(2)$-structure on $M$ given by $(\alpha, \omega_1,\omega_2,\omega_3)$ satisfying 
\end{itemize}

\vspace{-24pt}

\begin{equation}\label{eq:su2structure}
d\alpha=-2\omega_1, \quad d\omega_2 = 3\alpha \wedge \omega_3, \quad d \omega_3 = -3 \alpha \wedge \omega_2.
\end{equation}
\begin{itemize}
\item[] Hence, the cone metric $dr^2+r^2g_5$ on $(1,+\infty)_r\times M$ is Calabi--Yau.
\item[(b)] If $\nu<-1$, then $d\eta_{\infty}=0$, and thus the connection is flat.
\end{itemize}
\end{proposition}
\begin{proof}
(a) \ We first see that, if $g$ is the metric induced by $\varphi$ on $X$, then \eqref{eq:ALC.decay} implies, in the norm defined by $g_{\varphi_\infty}$:
$$|\nabla^j(g_{\varphi_\infty}-p^*g|_{X\setminus K})|=O(r^{\nu-j})\quad\text{ as }r\to+\infty,$$
for $j=0,1$.
Therefore, if we let $\psi_{\infty}$ and $\psi$ be the Hodge duals of $\varphi_{\infty}$ and $\varphi$ respectively, then we also have that
\begin{equation}\label{eq:ALC.decay.psi}
|\nabla^j(\psi_{\infty}-p^*\psi |_{X\setminus K})|=O(r^{\nu-j})\quad\text{ as }r\to+\infty.
\end{equation}
Let $V$ be the infinitesimal generator of the $U(1)$-action on $\Sigma$. The hypothesis that $\varphi_{\infty}$ is $U(1)$-invariant implies that $L_{V} \varphi_{\infty}=0$.  Given $d\varphi=0$ and \eqref{eq:ALC.decay}, we see that 
$d\varphi_{\infty}=O(r^{\nu-1})$ (i.e.~it equals a form whose pointwise norm with respect to $g_{\varphi_{\infty}}$ decays with order  $O(r^{\nu-1})$). 
We deduce from Cartan's formula that
\begin{equation}\label{eq:ALC.Vphi}
d(\iota_V \varphi_{\infty})= L_V \varphi_{\infty} - \iota_V d \varphi_{\infty} = O(r^{\nu-1}),
\end{equation}
and so we can write $\iota_V \varphi_{\infty}= \omega + O(r^{\nu})$, where $\omega$ is the pullback of a homogeneous $2$-form of order $O(1)$ on $(1,+\infty)\times M$.  Since $\psi_{\infty}$ is $U(1)$-invariant and satisfies \eqref{eq:ALC.decay.psi}, we can similarly deduce that 
$$\iota_V \varphi_{\infty}, \quad \varphi_{\infty}- \eta_{\infty} \wedge \iota_V \varphi, \quad \iota_V \psi_{\infty},\quad \psi_{\infty} - \eta_{\infty} \wedge \iota_V \psi_{\infty},$$
are all asymptotically $V$-basic, i.e.~are asymptotic to the pullback of homogeneous forms on $(1,+\infty) \times M$ of order $O(1)$. We may therefore write
\begin{eqnarray}\label{eq:ALC.phi}
\varphi_{\infty} & = & m \eta_{\infty} \wedge \omega + \Omega_1 + O(r^{\nu}), \\
\psi_{\infty} & = & \frac{1}{2} \omega^2 -m \eta_{\infty} \wedge \Omega_2  + O(r^{\nu}) ,\label{eq:ALC.psi}
\end{eqnarray}
where (omitting the pullbacks) $(\omega, \Omega_2)$ is an $SU(3)$-structure $(1,+\infty) \times M$ (so $\Omega_1=J\Omega_2$ where $J$ is the almost complex structure determined by $\Omega_2$).\\
From \eqref{eq:ALC.Vphi} and \eqref{eq:ALC.phi} we see that
$d\omega=O(r^{\nu-1})$ 
where $\nu-1<-1$, but $d\omega$ is homogeneous of order $O(r^{-1})$ so we must have $d\omega=0$.  A similar argument using $d(\iota_V\psi_{\infty})=O(r^{\nu-1})$ and \eqref{eq:ALC.psi} gives that $d\Omega_2=0$.  
As this $SU(3)$-structure must be compatible with the metric $dr^2 + r^2 g_5$ we conclude that $M^5$ is equipped with an $SU(2)$-structure $(\alpha, \omega_1, \omega_2, \omega_3)$ such that
\begin{eqnarray}\label{eq:su2.omega}
\omega & = & -rdr \wedge \alpha + r^2 \omega_1, \\ \label{eq:su2.Omega1}
\Omega_1 & = & r^3 \alpha \wedge \omega_2 - r^2 dr \wedge \omega_3, \\ \label{eq:su2.Omega2}
\Omega_2 & = & r^2 dr \wedge \omega_2 + r^3 \alpha \wedge \omega_3.
\end{eqnarray}
Recall that $d\varphi_{\infty}=O(r^{\nu-1})$ and we have shown that $d\omega=0$.  Thus, taking the derivative of \eqref{eq:ALC.phi} gives:
$$m d \eta_{\infty} \wedge \omega + d \Omega_1=O(r^{\nu-1}).$$
We know that $d\eta_{\infty}$ is $V$-basic and homogeneous of order $O(r^{-2})$ and $\omega$ is of order $O(1)$, so we deduce that 
$d\Omega_1=O(r^{\text{max}\{\nu-1,-2\}})$.  However, $\nu-1<-1$ and 
$d\Omega_1$ is homogeneous of order $O(r^{-1})$ so we must have $d\Omega_1=0$.\\
Summarising, we have that
$d\omega=0$,  $d\Omega_1=0$ and $d\Omega_2=0$, 
which is to say that the cone $(1,+\infty)\times M$ is Calabi--Yau.  
Moreover, the conditions on $(\alpha,\omega_1,\omega_2,\omega_3)$  
 now follow immediately from \eqref{eq:su2.omega}-\eqref{eq:su2.Omega2}.\\
(b) \ From \eqref{eq:ALC.phi} and $d\omega=d\Omega_1=0$ we know that 
$d\eta_{\infty}\wedge\omega=O(r^{\nu-1})$. Since $d\eta_{\infty}$ is homogeneous of order $O(r^{-2})$ and $\nu-1<-2$ we deduce that $d\eta_{\infty}\wedge\omega=0$.  In terms of the $SU(2)$-structure, by \eqref{eq:su2.omega}, we have that 
$$d\eta_{\infty}\wedge\alpha=0\quad\text{and}\quad d\eta_{\infty}\wedge\omega_1=0.$$
The first equation gives $d\eta_{\infty}=\alpha\wedge\beta$ for some $V$-semibasic 1-form $\beta$, but then the second equation forces $\beta=0$.  
\end{proof}

\noindent Now we know from Proposition \ref{prop:SU(2)StructureLimit} that the asymptotic cone for an ALC $G_2$-manifold is Calabi--Yau, we can impose a further condition on the connection $\eta_{\infty}$ for the definition of an ALC $G_2$-manifold: namely, that $\eta_{\infty}$ is Hermitian--Yang--Mills, i.e.~$d\eta_{\infty}\wedge\omega^2=0$ and $d\eta_{\infty}\wedge\Omega_2=0$.\\
We now give the example of the standard Sasaki-Einstein structure on $S^2 \times S^3$ in terms of the framework above. This is (up to scaling) the limiting structure appearing for the metrics in $\S$\ref{sss:BGGG}-\ref{sss:BB}, and so is the most important for our study.  

\begin{example}\label{ex:S2S3}
Let $S^2 \times S^3= SU(2)^2/ \Delta U(1)$ and let $\lbrace \eta_i^+, \eta_i^- \rbrace_{i=1}^3$ be as in $\S$\ref{sec:G2metrics}.
We can equip $S^3 \times S^3 \rightarrow S^2 \times S^3$ with a connection such that $\eta_2^+,\eta_3^+,\eta_1^-,\eta_2^-,\eta_3^-$ is a horizontal coframing. We define:
\begin{gather*}
\eta_{\infty}=2\eta_1^+,\quad \alpha  =  -\frac{4}{3} \eta_1^-, \quad \omega_1  =  \frac{4}{3} \left( \eta_2^+ \wedge \eta_3^- + \eta_2^- \wedge \eta_3^+ \right),  \\
 \quad
 \omega_2  =  \frac{4}{3} \left( \eta_{2}^+\wedge\eta_{3}^+ - \eta_{2}^-\wedge\eta_{3}^- \right),\quad
\omega_3  =  \frac{4}{3} \left( \eta_2^+ \wedge \eta_2^- + \eta_3^+ \wedge \eta_3^- \right).
\end{gather*}
The forms $\alpha,\omega_1,\omega_2,\omega_3$ are basic for the $\Delta U(1)$-action and equip $S^2\times S^3$ with an $SU(2)$-structure.  We can check that \eqref{eq:su2structure} holds 
 and so this is the standard homogeneous Sasaki-Einstein structure on $S^2 \times S^3$. 
We also see that $\eta_{\infty}$ is a connection form on $S^3\times S^3$ such that 
$$d \eta_{\infty} = -4 \left( \eta_{2}^+\wedge\eta_3^+ + \eta_{2}^-\wedge\eta_{3}^- \right)$$
is basic anti-self-dual: i.e.~$d \eta_{\infty} \wedge \omega_i=0$  for $i=1,2,3$.  This implies that $\eta_{\infty}$ is Hermitian--Yang--Mills  
by \eqref{eq:su2.omega} and \eqref{eq:su2.Omega2}.
\end{example}

\noindent From Example \ref{ex:S2S3} and $\S$\ref{sss:BGGG}  we see that the BGGG $\mathbb{R}^4\times S^3$ is ALC with rate $\nu=-1$ (and $m=1$), and the asymptotic structure is that given by Example \ref{ex:S2S3}.  This is in particular shows that 
Proposition \ref{prop:SU(2)StructureLimit}(b) is sharp.

\subsubsection[\texorpdfstring{$G_2$-instantons}{G2-instantons}]{{\boldmath $G_2$}-instantons}

We now study the asymptotic behaviour of $G_2$-instantons on ALC $G_2$-manifolds, and begin by examining the $G_2$-instanton condition on the asymptotic $U(1)$-bundle over a Calabi--Yau cone.  We shall use the notation of the previous subsection.\\
Let $\pi: (1,+\infty)_r \times \Sigma \rightarrow (1,+ \infty)_r \times M$ be a $U(1)$-bundle over a Calabi--Yau cone, equipped with the $G_2$-structure
$$\varphi_{\infty} = m \eta_{\infty} \wedge \omega + \Omega_1,$$
as above. Let $P$ be the pullback to $(1,+\infty)\times \Sigma$ of a bundle over $M$. If $A_{\infty}$ is a connection on $P$ (not necessarily pulled back from a connection over $\Sigma$), then without loss of generality we can write it in temporal gauge.  Denote by $a'$ be the pullback of a connection on $P \rightarrow (1,+\infty) \times M$ to $\pi^* P \rightarrow (1,+\infty) \times \Sigma$.  Then, we can write $A_{\infty}= a' + (A_{\infty}-a')$, where $A_{\infty}-a' \in \Omega^1((1,+\infty)\times \Sigma, \mathfrak{g}_P)$.  Hence, $A_{\infty}-a'=a''+ m\Phi \otimes \eta_{\infty}$, for some unique $\Phi\in \Omega^0((1,+\infty)\times \Sigma, \mathfrak{g}_P)$ and $\pi$-semibasic $a'' \in \Omega^1((1,+\infty)\times \Sigma, \mathfrak{g}_P)$, i.e.~$\iota_V a''=0$. We can thus define the connection $a=a'+a''$ and write 
\begin{equation}\label{eq:ALC.limit.A}
A_{\infty}=a+m\Phi \otimes \eta_{\infty}.
\end{equation} The curvature of this connection is
\begin{eqnarray}\nonumber
F_{A_{\infty}} & = & F_a+d_a (m\Phi \otimes \eta_{\infty}) \\ \nonumber
& = & F_a + m\Phi \otimes d \eta_{\infty} +m d_a \Phi \wedge \eta_{\infty} \\ \nonumber
& = & F_a^B +m \Phi \otimes d \eta_{\infty} + (m d_a \Phi - \iota_V F_a ) \wedge \eta_{\infty},
\end{eqnarray}
where $F_a^B= F_a - \eta_{\infty} \wedge \iota_V F_a$ denotes the semibasic component of $F_a$. Now, using \eqref{eq:ALC.psi} and $d \eta_{\infty} \wedge \omega^2=0$ and $d \eta_{\infty} \wedge \Omega_2=0$ (the Hermitian--Yang--Mills condition), the $G_2$-instanton equation $F_{A_{\infty}} \wedge \psi_{\infty}=0$ turns into
\begin{eqnarray}\nonumber
0 & = & \left( F_a^B + m\Phi \otimes d \eta_{\infty} + (md_a \Phi - \iota_V F_a ) \wedge \eta_{\infty} \right) \wedge \left(-m \eta_{\infty} \wedge \Omega_2 + \frac{1}{2} \omega^2\right) \\ \nonumber
& = & \left( m (F_a^B + m\Phi \otimes d \eta_{\infty} ) \wedge \Omega_2 + \frac{1}{2}(m d_a \Phi - \iota_V F_a ) \wedge \omega^2 \right) \wedge \eta_{\infty} \\ \nonumber
& & + \frac{1}{2}(F_a^B + m\Phi \otimes d \eta_{\infty} ) \wedge \omega^2 \\ \nonumber
& = & \left( m F_a^B \wedge \Omega_2 + \frac{1}{2}(m d_a \Phi - \iota_V F_a ) \wedge \omega^2 \right) \wedge \eta_{\infty}  + \frac{1}{2}F_a^B \wedge \omega^2.
\end{eqnarray}
We deduce that
\begin{equation}\label{eq:PreMonopole}
 m F_a^B \wedge \Omega_2 + \frac{1}{2}(md_a \Phi - \iota_V F_a ) \wedge \omega^2 =0\quad\text{and}\quad  F_a^B \wedge \omega^2 =0 .
\end{equation}
In our case we will be investigating $G_2$-instantons that are invariant under the $U(1)$-action generated by $V$; that is, we take a lift of the $U(1)$-action to the total space and the connection $A$ is invariant under the lifted action.\\  Take a local trivialization of $P$ and local coordinates $(x_i, \theta)$ such that $V= \partial_{\theta}$ and the $\partial_{x_i}$ are $\eta_{\infty}$-horizontal.  With these choices, the connection is given by $$A_{\infty}= a_i \otimes dx_i + m\Phi \otimes d \theta,$$ so $a=a_i\otimes dx_i$ and thus 
\begin{equation}\label{eq:iVFa}
\iota_VF_a=\frac{\partial a_i}{\partial\theta} dx_i.
\end{equation}
The $U(1)$-invariance of $A_{\infty}$ forces $\frac{\partial a_{i}}{\partial \theta}=0$.  
Notice that if we pick a different trivialization, then as the bundle is pulled back from $M$ the transition function $g$ to this new trivialization is independent of $\theta$, i.e.~$\frac{\partial g}{\partial \theta}=0$.  The $a_i$ will become $\tilde{a}_i=g^{-1} a_i g + g^{-1} \frac{\partial g}{\partial x_i}$, so we see that $\frac{\partial \tilde{a}_i}{\partial \theta} = 0$ as well.  Thus, the $\theta$-independence of the $a_i$ is independent of the choice of trivialisation.\\ This means that $\iota_V F_a =0$ by \eqref{eq:iVFa} and so $F_a^B=F_a$.  Hence, \eqref{eq:PreMonopole}  becomes
\begin{equation}\label{eq:PreMonopole2}
 F_a \wedge \Omega_2 = - \frac{1}{2 }d_a \Phi  \wedge \omega^2,
 \quad  F_a \wedge \omega^2 =0 ,
\end{equation}
which are the equations for a Calabi--Yau monopole $(a, \Phi)$ on $(1,+\infty) \times M$ equipped with the conical torsion-free $SU(3)$-structure $(\omega, \Omega_2)$.\\
These observations lead to the following.

\begin{proposition}\label{prop:ALC.instant}
Let $A$ be a $G_2$-instanton on an ALC $G_2$-manifold $(X,\varphi)$ and use the notation from the start of $\S$\ref{ss:ALC.structure}. Suppose there exists a $U(1)$-invariant connection $A_{\infty}=a+m\Phi\otimes\eta_{\infty}$ as in \eqref{eq:ALC.limit.A} such that $p^*F_A|_{X\setminus K}$ is asymptotic at infinity to $F_{A_{\infty}}$.
  Then $(a,\Phi)$ is a Calabi--Yau monopole on the Calabi--Yau cone $(1,+\infty)\times M$.
\end{proposition}

\begin{proof}
Since $A$ is a $G_2$-instanton we have $F_A\wedge\psi=0$ and we have that $\psi$ is asymptotic to $\psi_{\infty}$ by \eqref{eq:ALC.decay.psi}.  By hypothesis, the leading order term in $p^*(F_A\wedge\psi)$ is $F_{A_{\infty}}\wedge\psi_{\infty}$, so this must vanish.  The discussion before the statement means that \eqref{eq:PreMonopole2} holds, so the result follows.
\end{proof}

\noindent Proposition \ref{prop:ALC.instant} shows that the natural limits of $G_2$-instantons on ALC $G_2$-manifolds (if they exist) are Calabi--Yau monopoles on Calabi--Yau cones. 

\begin{remark}
These observations further motivate the study of Calabi--Yau monopoles on  cones, particularly the conifold.  See \cite{Oli16} for some examples and results on Calabi--Yau monopoles in the asymptotically conical and conical settings. 
\end{remark}

\begin{remark}
We already noted that $G_2$-instantons on ALC manifolds are the $G_2$-analogue of anti-self-dual connections on ALF gravitational instantons. However, in the $G_2$ case assuming an $L^2$ bound on the curvature is too restrictive.  Instead one may suppose either some $L^p$ bound for $p >2$, or simply that the pointwise norm of the curvature decays. In this paper we shall use the latter assumption when we need a curvature bound. However, it would be interesting to use the techniques in chapters 3 and 4 of \cite{Cherkis2016} in order to better understand the asymptotic form of $G_2$-instantons with some $L^p$ bound on their curvature.
\end{remark}

\subsection{Homogeneous bundles and invariant fields}\label{ss:hombdles}

We will now classify invariant connections on bundles over the $SU(2)^2$-principal orbits in the $G_2$-manifolds $X$ of $\S$\ref{sec:G2metrics} so $X\cong\mathbb{R}^4\times S^3$ or $L^4\times S^3$.\\ 
We start with a review of the general setup on a homogeneous manifold $K/H$. First, $K$-homogeneous $G$-bundles over $K/H$ (which will be our principal orbits) are determined by their isotropy homomorphism. These are group homomorphisms $\lambda: H \rightarrow G$, associated with which we construct the bundle $P_{\lambda}=K \times_{(H,\lambda)} G$. The reductive splitting $\mathfrak{k} = \mathfrak{h} \oplus \mathfrak{m}$ equips $K \rightarrow K/H$ with a connection whose horizontal space is $\mathfrak{m}$. This is the so-called canonical invariant connection and its connection form $A^c_{\lambda} \in \Omega^1(K, \mathfrak{g})$ is the left-invariant translation of $d \lambda \oplus 0 : \mathfrak{h} \oplus \mathfrak{m} \rightarrow \mathfrak{g}$. Other invariant connections are classified by Wang's theorem \cite{Wang1958} and are in correspondence with morphisms of $H$-representations $\Lambda : (\mathfrak{m} , \Ad ) \rightarrow (\mathfrak{g} , \Ad \circ \lambda)$.\\
In the cases we shall consider, $SU(2)^2$ acts with cohomogeneity-$1$ and the principal orbits are of the form $M=S^3 \times S^3 / H$, where $H$ will only be nontrivial in the BB case where it is $\mathbb{Z}^4$. Isomorphism classes of homogeneous $G$-bundles on these principal orbits are in correspondence with (conjugacy classes) of isotropy homomorphisms, i.e.~group homomorphisms $\lambda: H \rightarrow G$. Therefore $\lambda$ will be the trivial homomorphism, except in the BB case where the possible $\lambda$'s are in one-to-one correspondence with cyclic subgroups of $G$ of order $1$, $2$ or $4$. Given such $\lambda$ determines the $SU(2)^2$-homogeneous $G$-bundle
$$P_\lambda=  SU(2)^2  \times_{(H,\lambda)} G.$$
The canonical invariant connection $a_c$ 
is the trivial one (given the choice of $H$), hence its connection $1$-form as an element of $\Omega^1(SU(2)^2, \mathfrak{g})$ vanishes. It follows from Wang's theorem \cite{Wang1958}, that any other invariant connection differs from $a_c$ by a morphism  of $H$-representations 
$$\Lambda : ( \mathfrak{su}^+ \oplus \mathfrak{su}^- , \Ad)  \rightarrow ( \mathfrak{g} , \Ad \circ \lambda ).$$
When $H$ is trivial, both these representations are trivial, and so $\Lambda$ is any linear map. Given such a $\Lambda$ we extend it by left-invariance to $SU(2)^2$. This gives rise to the $1$-form with values in $\mathfrak{g}$:
\begin{equation}\label{eq:InvariantConnection}
a=\sum_{i=1}^3 a_i^+ \otimes \eta_i^+ + a_i^- \otimes \eta_i^-,
\end{equation}
where $a_i^{\pm} \in \mathfrak{g}$ are constant on each principal orbit. Hence, on the open dense set $\mathbb{R}^+_t \times M \subset X$, the most general $SU(2) \times SU(2)$-invariant connection on any $P_{\lambda}$ can be written as in \eqref{eq:InvariantConnection} with the $a_i^{\pm}$ depending on $t \in \mathbb{R}^+$ and taking values in $\mathfrak{g}$.

\begin{remark}
We can always use an $SU(2)^2$-invariant gauge transformation $g : \mathbb{R}^+ \rightarrow G$ to put any invariant connection $A$ in temporal gauge. This amounts to solving the ODE $\dot{g}g^{-1} + gA(\partial_t)g^{-1} =0$, which has a unique solution $g$ converging to $1$ as $t \rightarrow + \infty$.
\end{remark}

\begin{lemma}\label{lem:Curvature}
The curvature of the connection $a(t)$ above on $\lbrace t \rbrace \times M$ is given by
\begin{eqnarray}\nonumber
F_a \!\! & = & \!\! \sum_{i=1}^3 [a_i^+, a_i^-] \otimes \eta_i^+ \wedge \eta_i^-   \\ \nonumber
& + &  \!\! \sum_{i=1}^3 \left( (-2 a_i^+ + [a_j^+, a_k^+]) \otimes \eta_{j}^+\wedge\eta_{k}^+ + (-2 a_i^+ + [a_j^-, a_k^-]) \otimes \eta_{j}^-\wedge\eta_{k}^- \right) \\ \nonumber
& + &  \!\! \sum_{i=1}^3 \left( (-2 a_i^- + [a_j^-, a_k^+]) \otimes \eta_{j}^- \wedge \eta_{k}^+ + (-2 a_i^- + [a_j^+, a_k^-]) \otimes \eta_{j}^+ \wedge \eta_{k}^- \right),
\end{eqnarray}
where in the summation above 
$(j,k)$ is such that $(i,j,k)$ is a cyclic permutation of $(1,2,3)$.
\end{lemma}
\begin{proof}
We can compute the curvature via $F_a = da + \frac{1}{2}[a \wedge a]$ and the Maurer--Cartan relations \eqref{eq:MC1}-\eqref{eq:MC2} for the coframing $\eta_i^{\pm}$. The details are a lengthy but straightforward computation.
\end{proof}

\subsection[\texorpdfstring{The $SU(2)^2$-invariant ODEs}{The SU(2)xSU(2)-invariant ODEs}]{The {\boldmath $SU(2)^2$}-invariant ODEs}\label{ss:ODEs}

We may now write down the ODEs arising from equations \eqref{eq:MEvolution1} and \eqref{eq:MEvolution2} which describe our invariant $G_2$-instantons. 

\begin{lemma}\label{lem:ODE}
Let $(i,j,k)$ denote cyclic permutations of $(1,2,3)$. Using the notation from \eqref{eq:metric} 
and \eqref{eq:InvariantConnection}, the evolution equations \eqref{eq:MEvolution1}-\eqref{eq:MEvolution2} for $SU(2)^2$-invariant $G_2$-instantons $a$ 
on $\mathbb{R}^+_t \times M$ are 
\begin{eqnarray}\nonumber
\frac{B_i}{A_i} \dot{a}_i^+ + \left(\frac{B_i}{B_jB_k} - \frac{B_i}{A_jA_k} \right) a_i^+ & = & \frac{B_i }{2 B_j B_k } [a_j^- , a_k^-]-\frac{B_i}{2A_j A_k} [a_j^+ , a_k^+], 
\\ \nonumber
\frac{A_i}{B_i} \dot{a}_i^- + \left(\frac{A_i}{B_jA_k} + \frac{A_i}{A_jB_k} \right) a_i^- & = &  \frac{A_i }{2 B_j A_k } [a_j^- , a_k^+]+\frac{A_i}{2A_j B_k} [a_j^+ , a_k^-] 
,
\end{eqnarray}
together with the constraint
$$\sum_{i=1}^3 \frac{1}{A_iB_i} [a_i^+ , a_i^-]=0.$$
\end{lemma}
\begin{proof}
The proof amounts to inserting the formula for the curvature $F_a$ from Lemma \ref{lem:Curvature} 
into \eqref{eq:MEvolution1}-\eqref{eq:MEvolution2}. For this we need to use the $SU(3)$-structure on the principal orbits given in \eqref{eq:BBSU(3)structure1}-\eqref{eq:BBSU(3)structure2}. For convenience we shall write $\eta^{\pm}_{a\ldots b}=\eta^{\pm}_a\wedge\ldots\wedge\eta^{\pm}_{b}$.\\  
We start by computing
\begin{eqnarray}\nonumber
F_a \wedge \Omega_2 \!\!\! & = & \!\!\!-8 B_1 \left( A_2 B_3 ([a_2^-, a_3^+] - 2a_1^-) +A_3 B_2 ([a_2^+,a_3^-]-2a_1^-) \right) \eta_{123}^-  \wedge \eta_{23}^+ \\ \nonumber
&  & \!\!\! -8 A_1 \left( A_2 A_3 ([a_2^-, a_3^-] - 2a_1^+) - B_2 B_3 ([a_2^+,a_3^+]-2a_1^+) \right) \eta_{123}^+  \wedge \eta_{23}^- \\ \nonumber
& & \!\!\! + \text{ cyclic permutations}.
\end{eqnarray}
Moreover, since $\vert \eta_i^-\vert_t = \frac{1}{2B_i}$ and $\vert \eta_i^+ \vert_t = \frac{1}{2A_i}$, we conclude that
$$\ast_t (8\eta_{123}^- \wedge \eta_{23}^+)  = -\frac{1}{2}\frac{A_1}{ A_2A_3B_1B_2B_3}\eta_1^+, \quad \ast_t (8\eta_{123}^+ \wedge \eta_{23}^-)  = \frac{1}{2}\frac{B_1}{ B_2B_3A_1A_2A_3}\eta_1^-$$
and cyclic permutations. Combining this with the previous computation we obtain \begin{eqnarray}\nonumber
\ast_t (F_a \wedge \Omega_2) & = & \frac{1}{2} \left( \frac{A_1}{A_3B_2 } ([a_2^-, a_3^+] - 2a_1^-) + \frac{A_1}{A_2B_3 } ([a_2^+,a_3^-]-2a_1^-) \right) \eta_{1}^+ \\ \nonumber
& - &  \frac{1}{2} \left( \frac{B_1}{B_2 B_3} ([a_2^-, a_3^-] - 2a_1^+) - \frac{B_1}{A_2 A_3} ([a_2^+,a_3^+]-2a_1^+) \right) \eta_{1}^- \\ \nonumber
& + &  \text{cyclic permutations}.
\end{eqnarray}
The complex structure $J_t$ is such that 
$$J_t \eta_i^+=-*\left(\eta_i^+\wedge\frac{\omega^2}{2}\right) = *(16A_jB_jA_kB_k\eta_{123}^+\eta_{jk}^-)= \frac{B_i}{A_i} \eta_i^-$$ and so it is straightforward to compute 
$$J_t \dot{a} = \sum_{i=1}^3 \frac{B_i}{A_i} \dot{a}_i^+ \otimes \eta_i^- - \frac{A_i}{B_i} \dot{a}_i^- \otimes \eta_i^+.$$
Inserting our formulae in \eqref{eq:IEvolution} 
gives the ODEs in the statement. We finally compute
$$F_a \wedge \frac{\omega^2}{2}= -16 A_1A_2 B_1B_2 [a_3^+,a_3^-] \eta_{123}^+ \wedge \eta_{123}^-+ \text{cyclic permutations},$$
yielding the constraint in the statement.
\end{proof}

\subsection{Elementary solutions}\label{sec:solutions}

In this subsection we consider elementary cases of $SU(2)^2$-invariant 
$G_2$-instanton equations on any of the $SU(2)^2$-invariant 
$G_2$-manifolds of cohomogeneity-$1$ described in $\S$\ref{sec:G2metrics}.   We will let $X$ denote such a $G_2$-manifold.\\  
We verify that flat connections satisfying our $G_2$-instanton equations in $\S$\ref{ss:flat} and we classify and describe all abelian $G_2$-instantons explicitly in $\S$\ref{ss:abelian}.

\subsubsection{Flat connections}\label{ss:flat}

Any flat connection on $X$ is obviously a $G_2$-instanton and so must be a solution to our equations (for any gauge group $G$). As the fundamental group $\pi_1(X)$ 
is trivial, any flat connection in this setting is gauge equivalent to the trivial connection. However, on a homogeneous bundle there may be invariant flat connections that are not gauge equivalent to the trivial connection through invariant gauge transformations.\\
Let $A=a(t)$ be an invariant connection on $X$ given as in \eqref{eq:InvariantConnection} for $a_i^{\pm}:\mathbb{R}^+\to\mathfrak{g}$ for $i=1,2,3$.  From the formula in Lemma \ref{lem:Curvature} for the curvature of $a(t)$, one sees that $A$ is flat if and only if $a_i^{\pm}$ are $t$-independent 
and, 
for all cyclic permutations $(i,j,k)$ of $(1,2,3)$,  we have
\begin{equation*}
[a_i^+, a_i^-] =0,\; a_i^+ = \frac{1}{2} [a_j^+, a_k^+] = \frac{1}{2} [a_j^-, a_k^-] \,\;\text{and}\;\, 
a_i^- = \frac{1}{2} [a_j^-, a_k^+] = \frac{1}{2} [a_j^+, a_k^-].
\end{equation*}
 It is elementary to verify that a constant (i.e.~$t$-independent) choice  of $a_i^{\pm}$ satisfying these conditions then solves the ODE system for $SU(2)^2$-invariant $G_2$-instantons in Lemma \ref{lem:ODE}.

\subsubsection{Abelian instantons}\label{ss:abelian}

On circle bundles, equivalently complex line bundles, the Lie algebra structure of the gauge group is trivial and the $G_2$-instanton equations in Lemma \ref{lem:ODE} become linear. Consequently, it is then easy to integrate them, which we shall now proceed to do.\\ By Lemma \ref{lem:ODE}, to find a $G_2$-instanton in this setting we must integrate
\begin{equation}\label{eq:abelian}
\dot{a}_i^+  =  - \left(\frac{A_i}{B_jB_k} - \frac{A_i}{A_jA_k} \right) a_i^+, \quad
\dot{a}_i^-  = - \left(\frac{B_i}{B_jA_k} + \frac{B_i}{A_jB_k} \right) a_i^-.
\end{equation}
Given $t_0 \in \mathbb{R}^+$, the equations \eqref{eq:abelian} can be integrated to
\begin{eqnarray}\label{eq:LinearInstantonPlus}
a_i^+ (t) & = & a_{i}^+(t_0) \exp \left( - \int_{t_0}^t \left(\frac{A_i}{B_jB_k} - \frac{A_i}{A_jA_k} \right) ds \right) , \\ \label{eq:LinearInstantonMinus}
a_i^- (t) & = & a_{i}^-(t_0) \exp \left( - \int_{t_0}^t \left(\frac{B_i}{B_jA_k} + \frac{B_i}{A_jB_k} \right) ds \right).
\end{eqnarray}
We see from the free constants in \eqref{eq:LinearInstantonPlus}-\eqref{eq:LinearInstantonMinus} that there is a family of $G_2$-instantons parametrized by $6$ real parameters defined on the complement of the singular orbit in $X$.  Using the results in Appendix \ref{app:R4xS3} 
we can characterise the subspaces of this family of  $G_2$-instantons  which smoothly extend over the singular orbit in the BS, BGGG and Bogoyavlenskaya metrics.  

\begin{proposition}\label{prop:U1}
Let $A$ be an $SU(2)^2$-invariant $G_2$-instanton on a $U(1)$-bundle, or equivalently a complex line bundle, over $\mathbb{R}^4\times S^3$ with an $SU(2)^2\times U(1)$-invariant $G_2$-holonomy metric. 
Then $A$ can be written as
$$A= \sum_{i=1}^3 a_{i}^+(t_0) \exp \left( - \int_{t_0}^t \left(\frac{A_i}{B_jB_k} - \frac{A_i}{A_jA_k} \right) ds \right) \eta_i^+ ,$$
for some $t_0\in\mathbb{R}^+$ and $a_i^+(t_0) \in \mathbb{R}$  for $i=1,2,3$, where $(i,j,k)$ is a cyclic permutation of $(1,2,3)$.
\end{proposition}
\begin{proof} 
The principal orbits on $\mathbb{R}^4\times S^3$ are $S^3 \times S^3$ and the singular one is $S^3=SU(2)^2/ \Delta SU(2)$. The extensions of a circle bundle $P$ on $\mathbb{R}^+\times S^3\times S^3$ to $S^3$ are parametrized by isotropy homomorphisms $\lambda:\Delta SU(2)\to U(1)$.  The only such homomorphism $\lambda$ is the trivial one, so the unique extension of $P$ to the singular orbit is as the trivial bundle.
\\ 
The canonical invariant connection on the trivial homogeneous bundle vanishes as an element of $\Omega^1(SU(2)^2 , \mathbb{R})$. Any other invariant connection on this bundle 
is then given as an element of $\Omega^1(SU(2)^2 , \mathbb{R})$ by the pullback of a bi-invariant $1$-form on $S^3=SU(2)^2/ \Delta SU(2)$.  However, the only such 1-form is the zero form, so the connection $A$ extends over the singular orbit if and only if  Lemma \ref{lem:Extend1formOnBS} in  Appendix \ref{app:R4xS3} applies to the $1$-form $a=\sum_{i=1}^3 a_i^+ \eta_i^+ + a_i^- \eta_i^-$.\\  
We deduce that, for $t$ near $0$, the $a_{i}^{\pm}(t)$ are even and $a_i^{\pm}(0)=0$ for  $i=1,2,3$. We know by Appendix \ref{app:R4xS3} (and explicitly by Examples \ref{ex:BS}-\ref{ex:BGGG} for the BS and BGGG metrics)   that for $t$ near $0$ we have 
$A_i(t)= \frac{t}{2} + t^3 C_i(t)$ and $B_i(t)=b_0+t^2 D_i(t)$, for some real analytic $C_i$, $D_i$ and some constant $b_0\neq 0$. Then, choosing $0<t_0 \ll 1$ and using the expressions \eqref{eq:LinearInstantonPlus}-\eqref{eq:LinearInstantonMinus}, we compute that for $t < t_0 \ll 1$
$$a_i^+(t) = a_i^+(t_0)t_0^{-2} t^2 + \ldots  \quad \text{and} \quad a_i^-(t) = a_i^-(t_0) t_0^4 t^{-4} + O(1).$$
Applying Lemma \ref{lem:Extend1formOnBS} to $a$, we deduce that $a_i^-(t_0)$ must vanish for $i=1,2,3$, while the $a_i^+(t_0)$ can be freely chosen.
\end{proof}

\noindent In the BS or BGGG case, we can evaluate the integrals in Proposition \ref{prop:U1} to give the following.

\begin{corollary}\label{cor:U1}
Let $A$ be an $SU(2)^2$-invariant $G_2$-instanton with gauge group $U(1)$ over the BS or BGGG $G_2$-manifold $\mathbb{R}^4\times S^3$ described in $\S$\ref{sec:G2metrics}.
\begin{itemize}
\item[(a)] In the BS case, $A$ can be written as
\begin{equation*}
A 
=\frac{r^3-1}{r}\sum_{i=1}^3x_i\eta_i^+
\end{equation*}
for some $x_1,x_2,x_3\in\mathbb{R}$, where $r\in[1,+\infty)$ is determined 
by \eqref{eq:r.BS}. 
\item[(b)] In the BGGG case, $A$ can be written as  
\begin{equation*}
A=\frac{(r-9/4)(r+9/4)}{(r-3/4)(r+3/4)}x_1\eta_1^++
\frac{(r-9/4)e^r}{\sqrt{r}(r+9/4)^2}(x_2\eta_2^++x_3\eta_3^+)
\end{equation*}
for some $x_1,x_2,x_3\in\mathbb{R}$, where $r\in[9/4,+\infty)$ is given by \eqref{eq:r.BGGG}.   When $x_2=x_3=0$, $A$ is a multiple of the harmonic 1-form dual to the Killing field generating the $U(1)$-action.
\end{itemize}
\end{corollary}

\noindent We already observe a marked difference in 
the behaviour of $G_2$-instantons on the BS and BGGG $\mathbb{R}^4\times S^3$ in this simple abelian setting.  In particular, the instantons in the BS case all have bounded curvature, whereas those in the BGGG case have bounded curvature only when $x_2=x_3=0$, in which case the curvature also decays to $0$ as $r\to\infty$.

\begin{remark}
Of course, for any abelian gauge group all Lie brackets vanish and the ODE system decouples into several independent linear ODEs for instantons on circle bundles. Hence, the construction of any abelian $G_2$-instanton in this setting reduces to the $U(1)$ case given in Proposition \ref{prop:U1}.
\end{remark}

\begin{remark}
We can also explicitly describe the abelian instantons on the Bazaikin--Bogoyavlenskaya manifolds.  Since the analysis for smooth extension over the singular orbit is more involved in this setting and we also wish to apply these results to more general $G_2$-instantons, we will report on this in future work.
\end{remark}

\section[\texorpdfstring{{\boldmath $SU(2)^3$}-invariant {\boldmath $G_2$}-instantons}{SU(2)xSU(2)xSU(2)-invariant G2-instantons}]{{\boldmath $SU(2)^3$}-invariant {\boldmath $G_2$}-instantons}\label{sec:SU(2)^3invariant}

In the Bryant--Salamon case the torsion-free $G_2$-structure is described in section \ref{sss:BS}. Recall that the structure enjoys an extra $SU(2)$-symmetry, so that $A_1=A_2=A_3$ and $B_1=B_2=B_3$ where
\begin{equation}\label{eq:ODEsMetricBryantSalamon}
\dot{A}_1=\frac{1}{2} \left( 1- \frac{A_1^2}{B_1^2} \right), \quad \dot{B}_1= \frac{A_1}{B_1}.
\end{equation}
We shall use this notation throughout this section.\\
The only possible homogeneous $SU(2)$-bundle $P$ on the principal orbits $S^3\times S^3$ is $P=SU(2)^2 \times SU(2)$, i.e.~the trivial $SU(2)$-bundle. We consider connection $1$-forms with the extra $SU(2)$-symmetry existent in the underlying geometry.\\
We begin in $\S$\ref{ss:BSODEs} by simplifying the ODEs and constraint system in Lemma \ref{lem:ODE} to this more symmetric situation, and then derive the conditions necessary to extend the solution to this system across the singular orbit in $\S$\ref{ss:BSinits}.  We give classification results for the solutions to these equations in $\S$\ref{ss:BSsols}.  We also examine the asymptotic behaviour of the solutions in terms of a connection on $S^3\times S^3$, and give a compactness result for the space of solutions.  The latter result is related to the familiar ``bubbling'' and ``removable singularities'' phenomena.  

\subsection[\texorpdfstring{The $SU(2)^3$-invariant ODEs}{The SU(2)xSU(2)xSU(2)-invariant ODEs}]{The {\boldmath $SU(2)^3$}-invariant ODEs}\label{ss:BSODEs}

We simplify the invariant $G_2$-instanton equations from Lemma \ref{lem:ODE} in this setting.

\begin{proposition}\label{prop:BSODEs}
Let $A$ be an $SU(2)^3$-invariant $G_2$-instanton  with gauge group $SU(2)$ on $\mathbb{R}^+ \times SU(2)^2 \cong \mathbb{R}^+ \times SU(2)^3/ \Delta SU(2)$. There is a standard basis $\{T_i\}$ of $\mathfrak{su}(2)$, 
i.e.~with $[T_i , T_j]= 2 \epsilon_{ijk} T_k$, such that (up to an invariant gauge transformation) we can write
\begin{equation}\label{eq:ClarkeA}
A=A_1x\left( \sum_{i=1}^3 T_i\otimes \eta_i^+ 
\right) + B_1y\left( \sum_{i=1}^3 T_i\otimes\eta_i^- 
\right),
\end{equation}
with $x$, $y:\mathbb{R}^+\to\mathbb{R}$ satisfying 
\begin{align}\label{eq:Clarke1}
\dot{x} & \; = \; \frac{\dot{A}_1}{A_1} x + y^2-x^2  && \hspace{-24pt}   = \; \frac{1}{2A_1}\left(1-\frac{A_1^2}{B_1^2}\right)x+y^2-x^2, && \\ \label{eq:Clarke2}
\dot{y} & \; = \;  \frac{ 2\dot{A}_1 -3}{A_1} y  +2xy  && \hspace{-24pt} =   \;
-\frac{1}{A_1}\left(2+\frac{A_1^2}{B_1^2}\right)y+2xy.&&
\end{align}
\end{proposition}
\begin{proof}
We start by realizing $SU(2)^2$ as $SU(2)^3/\Delta SU(2)$. Isomorphism classes of $SU(2)^3$-equivariant bundles over $SU(2)^2$ are then in correspondence with conjugation classes of homomorphisms $\mu : SU(2) \rightarrow SU(2)$.  There are only two such conjugation classes, namely those represented by the identity and the trivial homomorphism.  
\\ 
We begin with the case where $\mu$ is the identity.  First, we fix a reductive decomposition of $\mathfrak{su}(2)^3$, i.e.~a complement $\mathfrak{m}$ of the isotropy algebra $\Delta \mathfrak{su}(2) \subset \mathfrak{su}(2)^3$ such that $[\Delta \mathfrak{su}(2) , \mathfrak{m}] \subset \mathfrak{m}$. We set $\Delta^+,\Delta^- \subset \mathfrak{su}^2$ to be diagonal and anti-diagonal respectively, and let
$$\mathfrak{m} = (0 \oplus \Delta^+ ) \oplus (0 \oplus \Delta^-).$$
By Wang's theorem \cite{Wang1958}, any $SU(2)^3$-invariant connection can be written as $$A= d \mu + \Lambda^+ + \Lambda^- \in \Omega^1(SU(2)^3 , \mathfrak{su}(2)),$$ where $\Lambda^{\pm}: ( \Delta^{\pm}, \Ad ) \rightarrow (\mathfrak{su}(2), \Ad \circ \mu)$ are morphisms of $SU(2)$-representations. We now  pull $A$ back to $SU(2)^2$ via the map $ \psi: SU(2)^2 \rightarrow SU(2)^3$ given by $\psi(g_1,g_2)= (g_2 g_1^{-1} , g_1 , g_2)$. Then $\psi^* d \mu =0$ and $\psi^* \Lambda^{\pm} = f^{\pm}_{ij} T_j \otimes \eta_i^{\mp}$ (the inversion to $\mp$ on the $\eta_i^{\mp}$ is correct!), for some functions $f_{ij}^{\pm}$ and some fixed standard basis $\{T_i\}$ of $\mathfrak{su}(2)$.  Extending naturally to $\mathbb{R}^+ \times SU(2)^2$ we obtain $a_i^{\pm}=f^{\mp}_{ij} T_j$.\\ 
For a fixed $t\in\mathbb{R}^+$, we can apply a gauge transformation so that 
 $\mu=\id$.  Hence, we can write 
 $a_i^+(t)=A_1(t) x(t)$ and $a_i^-(t)=B_1(t) y(t)$ for $i=1,2,3$, since the adjoint representation  of $SU(2)$ is irreducible, where we have introduced the non-zero factors of $A_1$ and $B_1$ for convenience.  Since the gauge transformation depends on $t$, we deduce that we can write
 $$A=A_1x \gamma(\sum_i T_i\otimes \eta_i^+)\gamma^{-1} +B_1y \gamma(\sum_i T_i\otimes\eta_i^-)\gamma^{-1}$$
 for some functions $\gamma:\mathbb{R}^+\to SU(2)$ and $x,y:\mathbb{R}^+\to\mathbb{R}$.  
\\
We now turn to the ODEs and constraint from Lemma \ref{lem:ODE} arising from the $G_2$-instanton condition.  We see that the constraint is immediately satisfied 
and the symmetry in the ODEs forces 
$$A_1x[\gamma^{-1}\dot{\gamma},T_i]=0\quad\text{and}\quad B_1y[\gamma^{-1}\dot{\gamma},T_i]=0$$
for $i=1,2,3$, which means that $\dot{\gamma}=0$ whenever $A$ is non-zero. 
Therefore, we may always write $A$ as in \eqref{eq:ClarkeA}.  
Using \eqref{eq:ODEsMetricBryantSalamon}, we then conclude that the ODEs from Lemma \ref{lem:ODE} imply that $x$ and $y$ satisfy \eqref{eq:Clarke1}-\eqref{eq:Clarke2} as claimed.
\\
We turn now to the case when $\mu: SU(2) \rightarrow SU(2)$ is the trivial homomorphism. Here, the canonical invariant connection $d \mu$ vanishes as a $1$-form on $SU(2)^3$ with values in $\mathfrak{su}(2)$. By Wang's theorem, any other invariant connection is then given by a morphism of $\Delta SU(2)$-representations $\Lambda : ( \mathfrak{m}, \Ad) \rightarrow (\mathfrak{su}(2) , \Ad \circ \mu)$. The left-hand side splits into two copies of the adjoint representation of $SU(2)$ while the right-hand side decomposes into three trivial representations. Schur's lemma then implies that $\Lambda$ must vanish and so the trivial connection is the unique invariant one on this homogeneous bundle.  This corresponds to taking $x=y=0$ in the statement.
\end{proof}

\subsection{Initial conditions}\label{ss:BSinits}

Now we determine the initial conditions in order for an $SU(2)^3$-invariant $G_2$-instanton $A$, given by a solution to the ODEs in Proposition \ref{prop:BSODEs}, to extend smoothly over the singular orbit $S^3=SU(2)^2/ \Delta SU(2)$. For that we need to first extend the bundle over the singular orbit. Up to an isomorphism of homogeneous bundles, there are two possibilities: these are
\begin{equation}\label{eq:Plambda}
P_{\lambda}=SU(2)^2 \times_{(\Delta SU(2), \lambda)} SU(2),
\end{equation}
with the homomorphism $\lambda:SU(2)\rightarrow SU(2)$ being either the trivial one (which we denote by $1$) or the identity $\id$. Depending on the choice of $\lambda$, the conditions for the connection $A$ to extend are different, as we show in the following lemma.

\begin{lemma}\label{lem:SmoothlyExtendBS}
The connection $A$ in \eqref{eq:ClarkeA} extends smoothly over the singular orbit $S^3$ if $x(t)$ is odd, $y(t)$ is even, and their Taylor expansions around $t=0$ are 
\begin{itemize}
\item either $x(t)=x_1t+ x_3 t^3+\ldots , \ y(t)= 
y_2 t^2 + \ldots$, in which case $A$ extends smoothly 
as a connection on $P_{1}$; 

\item or $x(t)=\frac{2}{t}+x_1t+\ldots , \ y(t)=y_0 + y_2 t^2 +\ldots$, in which case $A$ extends smoothly as a connection on $P_{\id}$.
\end{itemize}
\end{lemma}
\begin{proof}
We only analyze the case $\lambda=\id$ in detail, as both situations are similar.
\\
 When $\lambda=\id$, the canonical invariant connection associated with the reductive splitting $\mathfrak{su}(2)^2= \mathfrak{su}^+(2) \oplus \mathfrak{su}^-(2)$ is 
\begin{equation}\label{eq:A.can}
A^{\text{can}}= \sum_{i=1}^3 T_i \otimes \eta_i^+ \in \Omega^1(SU(2)\times SU(2), \mathfrak{su}(2)).
\end{equation} Therefore, for  $A$ to extend over the singular orbit as a connection on $P_{\id}$ we need to apply Lemma \ref{lem:ExtendConnectionOnBS} in Appendix \ref{app:R4xS3} to the $1$-form
$$A-A^{\text{can}}= \left( A_1x - 1 \right)\left( \sum_{i=1}^3 T_i\otimes\eta_i^+ 
\right) + B_1y \left( \sum_{i=1}^3 T_i\otimes \eta_i^- 
\right).$$ 
We conclude that $A$ extends over the singular orbit $S^3$ if
\begin{itemize}
\item $A_1(t)x(t)$, $B_1(t)y(t)$ are both even,
\item $\lim_{t \rightarrow 0} A_1(t)x(t)=1$ and $\lim_{t\to 0}B_1(t)y(t)$ is finite.
\end{itemize}
By Lemma \ref{lem:ExtendingSmoothlyMetric} in Appendix \ref{app:R4xS3} (or by inspection since the BS 
metric is explicit), we see that $A_1(t)$ is odd and $B_1(t)$ is even, so $x(t)$ and $y(t)$ must be odd and even respectively.  Moreover, $\dot{A}_1(0)=\frac{1}{2}$ and $B_1(0)\neq 0$, 
 as we see in Example \ref{ex:BS} in Appendix \ref{app:R4xS3}, 
so the expansions of $x$, $y$ around zero must be as claimed in the lemma.\\
To carry over the analysis in the case where $\lambda=1$ we apply Lemma \ref{lem:ExtendConnectionOnBS} directly to the $1$-form $A$, giving $A_1x$, $B_1y$ are even with $\lim_{t\to 0}A_1x=\lim_{t\to0}B_1y=0$. 
\end{proof}

\subsection{Solutions and their properties}\label{ss:BSsols}

We now describe solutions of the $SU(2)^3$-invariant $G_2$-instanton equations, which splits into two cases: when the bundle $P=P_1$ and when $P=P_{\id}$, in the notation of the previous subsection.  In the first case we recover the $G_2$-instantons constructed in \cite{Clarke14}, and in the second case we find a new example of a $G_2$-instanton.  We then analyse the asymptotic behaviour of the instantons, and finally show that the $\mathbb{R}_{\geq 0}$-family of solutions on $P_1$ admits a natural compactification.

\subsubsection[\texorpdfstring{Solutions smoothly extending on $P_1$}{Solutions smoothly extend on P1}]{Solutions smoothly extending on {\boldmath $P_1$}}

Clarke \cite{Clarke14} constructed a 1-parameter family of $G_2$-instantons on the Bryant--Salamon $\mathbb{R}^4\times S^3$.  These instantons live on the bundle $P_1$ given by \eqref{eq:Plambda}, i.e.~when the homomorphism $\lambda$ is trivial.  Moreover, they have $y=0$ in the notation of Proposition \ref{prop:BSODEs}, and so the ODEs there reduce to a single ODE for $x$ which can be explicitly integrated.  We shall reconstruct these $G_2$-instantons in the proof of the next result, 
 which classifies and explicitly describes the $G_2$-instantons that smoothly extend over the singular orbit on the bundle $P_1$.

\begin{theorem}\label{prop:Clarke}\label{thm:Clarke}
Let $A$ be an $SU(2)^3$-invariant $G_2$-instanton with gauge group $SU(2)$ on the Bryant--Salamon $G_2$-manifold $\mathbb{R}^4\times S^3$, which smoothly extends over the singular orbit on $P_{1}$. Then, $A$ is one of Clarke's examples \cite{Clarke14}, in which case there is $x_1 \in \mathbb{R}$ such that, in 
the notation of Proposition \ref{prop:BSODEs},
\begin{equation}\label{eq:ClarkeG2InstantonFormula}
x(t)=\frac{2x_1 A_1(t)}{1+ x_1 (B_1^2(t)-\frac{1}{3})  
}\quad\text{and}\quad y(t)=0.
\end{equation}
 Given such an $x_1 \in \mathbb{R}$ we shall denote the resulting instanton by $A^{x_1}$.  Observe that $A^{x_1}$ is defined globally on $\mathbb{R}^4\times S^3$ if and only if $x_1\geq 0$ and that $A^0$ is the trivial flat connection.
\end{theorem}
\begin{proof}
It will be enough to show that any instanton as in the statement defined on a neighbourhood of the singular orbit must coincide with one of Clarke's examples there. For that, let $(x(t),y(t))$ be a solution to the ODEs \eqref{eq:Clarke1}-\eqref{eq:Clarke2}.  We  show that if the resulting instanton $A$ extends over the singular orbit then $y(t)=0$ for all $t$.\\ Recall from Lemma \ref{lem:SmoothlyExtendBS} that for $A$ to smoothly extend over the singular orbit on $P_1$ 
we must have
$$x(t) = x_1 t + t^3 u(t)\quad\text{and}\quad y(t) =  t^2 v(t)$$
for $t$ near $0$, where $u$, $v$ are real analytic even functions of $t$.
The system \eqref{eq:Clarke1}-\eqref{eq:Clarke2} for $x$, $y$ becomes the following system for $u$, $v$:
\begin{eqnarray} 
\label{eq:dotu1} \dot{u} & = & - \frac{2u + x_1^2 + x_1/2 }{t} + f_1(t,u,v), \\  
\label{eq:dotv1} \dot{v} & = & - \frac{ 6v}{t} + f_2(t,u,v),
\end{eqnarray}
where $f_1,f_2: [0,+\infty) \times \mathbb{R}^2 \rightarrow \mathbb{R}$ are some other real analytic functions. The existence and uniqueness theorem for equations with regular singular points (see \cite{Malgrange1974} chapters 6 and 7, and Theorem 4.7 in  \cite{Foscolo2015} for a clearer statement) applies here provided that  
$$u(0)=-\frac{x_1}{4}-\frac{x_1^2}{2}\quad\text{and}\quad v(0)=0.$$ In that case, for each $x_1 \in \mathbb{R}$ we obtain a unique solution $(x(t),y(t))$ in $[0, \epsilon) $, for some $\epsilon>0$.\\  
We are left with showing that all such solutions have $y=0$. That is indeed the case as we can simply set $y=0$ and integrate the equation for $x$:
\begin{equation}\label{eq:dotxBS}
\dot{x}=\frac{\dot{A}_1}{A_1} x-x^2.
\end{equation}
Writing this equation as 
\begin{equation}\label{eq:dotxBS2}
\frac{d}{dt}\left( \frac{x}{A_1} \right) = -A_1 \left( \frac{x}{A_1} \right)^2
\end{equation} makes it separable.  Since $B_1\dot{B}_1=A_1$ by \eqref{eq:ODEsMetricBryantSalamon} and $B_1^2(0)=\frac{1}{3}$, \eqref{eq:dotxBS2} can be readily integrated to show that $x$ is given as in \eqref{eq:ClarkeG2InstantonFormula}. 
By uniqueness the solutions guaranteed by the local existence theorem must be these ones and so have $y=0$. These are   the $G_2$-instantons found 
in \cite{Clarke14}.
\end{proof}

\noindent Using the implicit coordinate $r\in[1,+\infty)$ in \eqref{eq:r.BS} and the formula \eqref{eq:AB.BS} 
we can explicitly write the $G_2$-instanton $A^{x_1}=A_1xT_i\otimes\eta_i^+$ with
$$A_1(r)=\frac{r}{3}\sqrt{1-r^{-3}}\quad\text{and}\quad
x(r)=\frac{2x_1r\sqrt{1-r^{-3}}}{3+x_1(r^2-1)}.$$
We see that the curvature of $A^{x_1}$ is
\begin{align*}
F_{A^{x_1}} & =  T_i\otimes \Big(\frac{d}{dr}(A_1x)dr\wedge \eta_i^++
A_1x(A_1x-1)\epsilon_{ijk}\eta_j^+\wedge\eta_k^+\\
&\qquad\qquad  - A_1 x \epsilon_{ijk}\eta_j^-\wedge\eta_k^- \Big).
\end{align*}
This can then be used to compute that
\begin{equation}\label{eq:Norm_of_curvature}
|F_{A^{x_1}}|^2= \frac{9}{2B_1^2} \Big\vert \frac{d}{dr}(A_1 x) \Big\vert^2 + \frac{3x^2}{2} \frac{(A_1 x-1)^2}{A_1^2} + \frac{3A_1^2 x^2}{2 B_1^2},
\end{equation} 
which shows that $|F_{A^{x_1}}|$ decays at infinity at $O(r^{-2})$.  Observe in particular that the curvature of $A^{x_1}$ does not lie in $L^2$.

\subsubsection[\texorpdfstring{Solutions smoothly extending on $P_{\id}$}{Solutions smoothly extending on Pid}]{Solutions smoothly extending on  {\boldmath $P_{\id}$}}

We now turn to solutions defined on the bundle $P_{\id}$ given by \eqref{eq:Plambda} with the homomorphism $\lambda=\id$.  We first give a local existence result for instantons on $P_{\id}$.

\begin{proposition}\label{prop:LocalExistenceBS_1}
Let $S^3$ be the singular orbit in the Bryant-Salamon $G_2$-manifold $\mathbb{R}^4\times S^3$. There is a one-parameter family of $SU(2)^3$-invariant $G_2$-instantons, with gauge group $SU(2)$, defined in a neighbourhood of $S^3$ and smoothly extending over $S^3$ on $P_{\id}$. The instantons are parametrized by $y_0 \in \mathbb{R}$ and satisfy, in the notation of Proposition \ref{prop:BSODEs},
$$x(t) = \frac{2}{t} + \frac{y_0^2-1}{4} t + O(t^3) , \quad y(t) = y_0 + \frac{y_0}{2} \left( \frac{y_0^2}{2} - 3  \right) t^2 + O(t^4) .$$
\end{proposition}
\begin{proof}
We consider the initial value problem for $(x(t),y(t))$ to be a solution to the ODEs \eqref{eq:Clarke1}-\eqref{eq:Clarke2} on $P_{\id}$. 
 By Lemma \ref{lem:SmoothlyExtendBS}, 
the conditions for smooth extension over the singular orbit are that
$$x(t) = \frac{2}{t} + t u(t) , \quad y(t) = y_0 + t^2 v(t),$$
for some real analytic functions $u$, $v: [0, + \infty ) \rightarrow \mathbb{R}$. 
Substituting these expressions and 
the expansion of $A_1$ from Example \ref{ex:BS} into \eqref{eq:Clarke1}-\eqref{eq:Clarke2} yields 
\begin{eqnarray}\label{eq:dotu2} 
\dot{u} & = & \frac{y_0^2-4u-1}{t} + f_1(t,u,v) \\\label{eq:dotv2}  
\dot{v} & = & - \frac{2v+5y_0/2-2y_0 u}{t} + f_2(t,u,v),
\end{eqnarray}
where $f_1,f_2: [0,+\infty) \times \mathbb{R}^2 \rightarrow \mathbb{R}$ are two real analytic functions up to $t=0$.  
Now we use the existence and uniqueness theorem for equations with regular singular points (chapters 6 and 7 in \cite{Malgrange1974}, or Theorem 4.7 \cite{Foscolo2015}). At this stage, this requires that $(u(0),v(0))$ are such that the $O(t^{-1})$ terms in \eqref{eq:dotu2}-\eqref{eq:dotv2} vanish and that the linear map $(u,v) \mapsto (-4 u , 2 y_0 u -2v)$ has no eigenvalues in the positive integers. The second condition holds (the eigenvalues are $-2,-4$) and the first condition requires that
$$u(0)=\frac{y_0^2-1}{4} , \quad v(0) = \frac{y_0}{2} \left( \frac{y_0^2}{2} - 3  \right).$$
The theorem for equations with regular singular points applies and shows that, under these conditions, for each $y_0 \in \mathbb{R}$ there is a unique solution $(u(t),v(t))$ to \eqref{eq:dotu2}-\eqref{eq:dotv2}, which gives the result. 
\end{proof}

\begin{theorem}\label{thm:Alim}
The $G_2$-instanton $A^{\lim}$ arising from the case when $y_0=0$ in Proposition \ref{prop:LocalExistenceBS_1} has
$$x(t)= \frac{A_1(t)}{\frac{1}{2}(B_1^2(t)-\frac{1}{3}) 
}\quad\text{and}\quad y(t)=0.$$ 
Moreover, $A^{\lim}$ extends as a $G_2$-instanton to the Bryant--Salamon   $G_2$-manifold 
 $\mathbb{R}^4\times S^3 $.
\end{theorem}
\begin{proof}
Back to the functions $x$, $y$ in Proposition \ref{prop:LocalExistenceBS_1}, we have that $y=0$ and $x$ is the unique solution to 
$$\dot{x}=\frac{\dot{A}_1}{A_1} x-x^2, \quad \lim_{t \rightarrow 0} A_1(t) x(t)=1.$$
Writing the ODE in the form \eqref{eq:dotxBS2} makes it separable, and using the initial condition we obtain the solution claimed.  Since $x(t)$ is defined for all $t$, the resulting instanton is globally defined.
\end{proof}

\noindent Again using the coordinate $r\in[1,\infty)$ in \eqref{eq:r.BS} and the formula \eqref{eq:AB.BS} 
we can write $A^{\lim}$ explicitly with 
$$x(r)=\frac{2r\sqrt{1-r^{-3}}}{r^2-1}.$$
From \ref{eq:Norm_of_curvature} we see that the curvature of $A^{\lim}$ decays at infinity at order $O(r^{-2})$, just as for $A^{x_1}$.

\begin{remark}
The reader may wonder about potential $G_2$-instantons $A$ arising from the
 local solutions with $y_0\neq 0$ in Proposition
 \ref{prop:LocalExistenceBS_1}.   Numerical investigation appears to
  indicate that such local solutions do not extend globally, if we impose 
 the condition that the curvature of $A$ decays at infinity.  We hope to study this situation further.
\end{remark}

\subsubsection{Asymptotics of the solutions}

We now consider the asymptotic behaviour of the $G_2$-instantons $A^{x_1}$ and $A^{\lim}$ constructed in Theorems \ref{thm:Clarke} and \ref{thm:Alim}.\\
Using the formula \eqref{eq:ClarkeG2InstantonFormula} for Clarke's $G_2$-instanton $A^{x_1}$ we see that for $x_1>0$ and large $t$, the connection form $a(t)$ on the time $t$ slice of $\mathbb{R}^3\times S^3$, which is diffeomorphic to $S^3\times S^3$, is given by
\begin{eqnarray}\nonumber
a(t) & = & \frac{2x_1 A_1^2(t)}{1+ x_1 (B_1^2(t)-\frac{1}{3}) 
}  \sum_{i=1}^3 T_i\otimes\eta_i^+   \sim \frac{2x_1 \frac{t^2}{9}}{1+ 2x_1 \frac{t^2}{6} }   \sum_{i=1}^3 T_i\otimes\eta_i^+  \\  \nonumber
& \sim & \frac{2}{3} \frac{1}{1+\frac{3}{x_1 t^2}}\sum_{i=1}^3 T_i\otimes\eta_i^+ , 
\end{eqnarray}
where we used the asymptotic behaviour of $A_1$, namely that $A_1(t) \sim \frac{t}{3} + O(t^{-2})$ and $B_1(t)\sim \frac{t}{\sqrt{3}}$ for $t$ large. Therefore
\begin{equation}\label{eq:ainfty}
a_{\infty} := \lim_{t \rightarrow +\infty}a(t)=\frac{2}{3}\sum_{i=1}^3 T_i\otimes\eta_i^+ 
\end{equation}
is the canonical $SU(2) \subset SU(3)$ connection for the homogeneous nearly K\"ahler structure on $S^3 \times S^3$.  Recall that such connections are pseudo-Hermitian--Yang--Mills (or nearly K\"ahler instantons). We can also compute the rate at which this happens and conclude that there is $c>0$ such that $\vert a - a_{\infty} \vert \leq \frac{c}{\vert x_1 \vert t^3}$ along the end.\\ Similarly, we compute that for $t \gg 1$
\begin{eqnarray*}
A^{\lim} & = & \frac{A_1^2(t)}{\frac{1}{2}(B_1^2(t)-\frac{1}{3}) 
} \sum_{i=1}^3 T_i\otimes \eta_i^+  = \frac{(t/3 + O(t^{-2}))^2}{t^2/6 + O(t^{-1})} \sum_{i=1}^3 T_i\otimes \eta_i^+ \\
&=& \frac{2}{3} ( 1+ O(t^{-3}) ) \sum_{i=1}^3T_i\otimes \eta_i^+ .
\end{eqnarray*}
Thus, $\vert A^{\lim} - a_{\infty} \vert = O(t^{-4})$, as $\vert \eta_i^+ \vert = O(t^{-1})$. We summarize these conclusions. 

\begin{proposition}\label{prop:asym}
Let $A$ be an $SU(2)^3$-invariant $G_2$-instanton given by Theorem \ref{thm:Clarke} or \ref{thm:Alim} which is defined globally on the Bryant--Salamon $G_2$-manifold $\mathbb{R}^4\times S^3$.  Then $A$ is asymptotic to the canonical pseudo-Hermitian--Yang--Mills connection $a_{\infty}$ in \eqref{eq:ainfty} for the homogeneous nearly K\"ahler structure on $S^3 \times S^3$. In particular:
\begin{itemize}
\item if 
$A=A^{x_1}$ for some $x_1 \in \mathbb{R}^+$, then for $t \gg 1$
$$\vert A^{x_1} - a_{\infty} \vert \leq \frac{c}{ x_1  t^3},$$
where $c>0$ is some constant independent of $x_1$;
\item if $A=A^{\lim}$, then for $t \gg 1$, $\vert A^{\lim} - a_{\infty} \vert = O(t^{-4})$. 
\end{itemize}
\end{proposition}

\begin{remark}
As previously mentioned, any $G_2$-instanton on an asymptotically conical 
$G_2$-manifold which has a well-defined limit at infinity and has pointwise decaying curvature will be asymptotic to a pseudo-HYM connection on the link of the asymptotic cone (\cite{Oliveira2014}).  Proposition 
\ref{prop:asym} refines this result in this setting.
\end{remark}

\subsubsection{Compactness properties of the moduli of solutions}

Next we show that as $x_1 \rightarrow + \infty$ Clarke's $G_2$-instantons $A^{x_1}$ ``bubble off'' an anti-self-dual (ASD) connection along the normal bundle to the associative $S^3= \{0\}\times S^3 \subset \mathbb{R}^4\times S^3$. We shall also show that in the same limit Clarke's $G_2$-instantons converge outside the associative $S^3$ to $A^{\lim}$. The fact that $A^{\lim}$ smoothly extends over $S^3$ can then be interpreted as a removable singularity phenomenon.\\ To state the result we now introduce some notation for the re-scaling we wish to perform: for $p \in S^3$ and $\delta >0$ we define the map $s^p_{\delta}$ from the unit ball $B_1\subseteq \mathbb{R}^4$ by
$$s^p_{\delta} : B_1\subseteq \mathbb{R}^4 \rightarrow B_{\delta}\times\{p\}\subseteq \mathbb{R}^4\times S^3, \ \ x \mapsto (\delta x,p).$$
Recall that if we view $\mathbb{R}^4\setminus\{0\}=\mathbb{R}^+_t\times S^3$ then the basic ASD instanton on $\mathbb{R}^4$ with scale $\lambda>0$ can be written as
\begin{equation}\label{eq:basic.ASD}
A^{\text{ASD}}_{\lambda}= \frac{\lambda t^2}{1+\lambda t^2} \sum_{i=1}^3T_i\otimes\eta_i^+ .
\end{equation}

\begin{theorem}\label{thm:Compactness}
Let $\lbrace A^{x_1} \rbrace$ be a sequence of Clarke's $G_2$-instantons from Theorem \ref{thm:Clarke} with $x_1 \rightarrow + \infty$. 
\begin{itemize}
\item[(a)] Given any $\lambda>0$, there is a sequence of positive real numbers $\delta=\delta(x_1,\lambda) \rightarrow 0$ as $x_{1} \rightarrow + \infty$ such that: for all $p \in S^3$, $(s^p_{\delta})^* A^{x_1}$ converges uniformly with all derivatives to the basic ASD instanton $A^{\text{\emph{ASD}}}_{\lambda}$ on $B_1\subseteq\mathbb{R}^4$ as in \eqref{eq:basic.ASD}. 
\item[(b)] The connections $A^{x_1}$ converge  uniformly with all derivatives to $A^{\lim}$ given in Theorem \ref{thm:Alim} on every compact subset of $(\mathbb{R}^4 \setminus \{0\})\times S^3  $ as $x_1\to +\infty$. 
\end{itemize}
\end{theorem}
\begin{proof}
We prove the two parts independently.\\
(a) \/ We view the basic instanton $A^{\text{ASD}}_{\lambda}$ in \eqref{eq:basic.ASD} as defined on $\mathbb{R}^4\times \{p\}$. Using the formula for $A^{x_1}$ in Theorem \ref{thm:Clarke} and the expansions of $A_1$ and $B_1$ near $0$ in Example \ref{ex:BS} from Appendix \ref{app:R4xS3}, we compute, for $t< 1$, 
\begin{eqnarray}\nonumber
(s^p_{\delta})^* A^{x_1} & = & A_1(\delta t ) x(\delta t ) T_i\otimes \eta_i^+  = \frac{2x_1 A_1^2(\delta t)}{1+x_1 \left( B_1^2(\delta t) - \frac{1}{3} \right)} T_i\otimes \eta_i^+  \\ \nonumber
& = &  \frac{x_1\delta^2 t^2/2 + O(x_1 \delta^4 t^4)}{1+x_1\delta^2 t^2/2 + O(x_1 \delta^4 t^4)}  T_i \otimes\eta_i^+ .
\end{eqnarray}
Hence, setting $\delta = \delta(x_1,\lambda)= \sqrt{2 \lambda / x_1}$ we have that for every $k \in \mathbb{N}_0$, there is $c_k>0$, independent of $\lambda$ and $x_1$, such that
\begin{eqnarray*}
\Vert (s^p_{\delta})^* A^{x_1} - A^{\text{ASD}}_{\lambda} \Vert_{C^k(B_1)} \leq c_k \frac{\lambda^2}{x_1}.
\end{eqnarray*}
Therefore, given $\epsilon >0$, we have for any $x_1\geq c_k \lambda^2/ \epsilon$ 
 that $$\Vert (s^p_{\delta})^* A^{x_1} - A^{\text{ASD}}_{\lambda} \Vert_{C^k(B_1)} \leq \epsilon,$$
 demonstrating the claimed convergence. \\
(b) \/ We take the explicit formulas for $A^{x_1}$ and $A^{\lim}$ in 
Theorems \ref{thm:Clarke}-\ref{thm:Alim} and compute
\begin{eqnarray}\nonumber
\vert A^{x_1} - A^{\lim} \vert & = & \frac{A_1^2(t)}{\frac{1}{2}(B_1^2(t)-\frac{1}{3}) 
} \ \Big\vert  \frac{x_1 (B_1^2(t)-\frac{1}{3})}{1+x_1 (B_1^2(t)-\frac{1}{3}) 
} -1 \Big\vert \Big\vert \sum_{i=1}^3T_i\otimes \eta_i^+ \Big\vert \\ \nonumber
& \leq & \frac{c A_1(t)}{\frac{1}{2}(B_1^2(t)-\frac{1}{3}) 
} \frac{1}{1+x_1 (B_1^2(t)-\frac{1}{3}) 
},
\end{eqnarray}
for some  constant $c>0$.  Recall that in the coordinate $r\in[1,+\infty)$ from \eqref{eq:r.BS} we have $B_1(r)=r/\sqrt{3}$ by \eqref{eq:AB.BS}.  
Hence, $B_1^2-\frac{1}{3}$ is bounded and bounded away from zero on every compact $K\subseteq (\mathbb{R}^4 \setminus \{0\})\times S^3 $.  Thus, for every such $K$  there is some (possibly other) constant $c>0$ such that 
\begin{equation}\label{eq:cvge2}
\vert A^{x_1} - A^{\lim} \vert  \leq  \frac{c}{1+x_1},
\end{equation}
and we have similar estimates for the derivatives of $A^{x_1}-A^{\lim}$. By letting $x_1 \rightarrow + \infty$ the right-hand side of \eqref{eq:cvge2} tends to zero as required.
\end{proof}

\begin{remark}
As already mentioned,  the fact that $A^{\lim}$ smoothly extends over $S^3$ is an example of a removable singularity phenomenon. It follows from Tian and Tao's work \cites{Tian2000,Tao2004} that such phenomena occur more generally provided that the $G_2$-instanton is invariant under a group action all of whose orbits have dimension greater than or equal to $3$ (codimension less than or equal to $4$).
\end{remark}

\noindent Even though the $G_2$-instantons $A^{x_1}$ and $A_{\lim}$ do not have finite energy, and so the results of \cite{Tian2000} do not immediately apply, we now show that we do have the expected energy concentration along the associative $S^3$.  Below, we let $\delta_{\lbrace 0 \rbrace \times S^3}$ denote the delta current associated with $\lbrace 0 \rbrace \times S^3$.

\begin{corollary}\label{cor:delta}
The function $\vert F_{A^{x_1}} \vert^2 - \vert F_{A_{\lim}} \vert^2$ is integrable for all $x_1>0$. Moreover, as $x_1 \rightarrow + \infty$ it converges to $8 \pi^2 \delta_{\lbrace 0 \rbrace \times S^3}$ as a current, i.e.~for all compactly supported functions $f$ we have
$$\lim_{x_1 \rightarrow + \infty}\int_{\mathbb{R}^4 \times S^3} f  (\vert F_{A^{x_1}} \vert^2 - \vert F_{A_{\lim}} \vert^2) \ \dvol_g = 8 \pi^2 \int_{\lbrace 0 \rbrace \times S^3} f \  \dvol_{g \vert_{\lbrace 0 \rbrace \times S^3}}.$$
\end{corollary}
\begin{proof}
First, a computation using \eqref{eq:Norm_of_curvature} shows that
\begin{eqnarray}\label{eq:Curvature_Difference}
\vert F_{A^{x_1}} \vert^2 - \vert F_{A_{\lim}} \vert^2 = \sum_{n=0}^3 \sum_{k=0}^{10} \frac{ 6q_{n,k} (r-1)^k x_1^n }{(r+1)^4 r^6 (r^2 x_1 -x_1 +3)^4} ,
\end{eqnarray}
for some (explicit) $q_{n,k} \in \mathbb{R}$. The claimed integrability of $\vert F_{A^{x_1}} \vert^2 - \vert F_{A_{\lim}} \vert^2$ now follows. Moreover, for future reference we mention here that
\begin{equation}\label{eq:q.values}
q_{3,0}=0=q_{3,1}, \ \text{and} \ q_{2,0}=2529.
\end{equation}
Next we prove the claimed convergence of $\vert F_{A^{x_1}} \vert^2 - \vert F_{A_{\lim}} \vert^2$ by showing that
\begin{equation}\label{eq:Limit}
\lim_{x_1 \rightarrow + \infty} \int_K (\vert F_{A^{x_1}} \vert^2 - \vert F_{A_{\lim}} \vert^2) \dvol_g 
\end{equation} 
vanishes if $K \subset (\mathbb{R}^4 \backslash \lbrace 0 \rbrace ) \times S^3$ is compact and equals $8 \pi^2  \ Vol( \lbrace 0 \rbrace \times S^3)$ if $\lbrace 0 \rbrace \times S^3 \subset K$. Notice that $\vert F_{A^{x_1}} \vert^2 - \vert F_{A_{\lim}} \vert^2$ is $SU(2)^3$-invariant and so it is enough to consider its integral over $SU(2)^3$-invariant subsets $K$ of $\mathbb{R}^4 \times S^3$. First we consider the case when $K$ is a compact subset of $(\mathbb{R}^4 \backslash \lbrace 0 \rbrace ) \times S^3$. Then, Theorem \ref{thm:Compactness}(b) guarantees that $\vert F_{A^{x_1}} \vert^2 - \vert F_{A_{\lim}} \vert^2$ converges uniformly to $0$ in $K$ and so \eqref{eq:Limit} is zero by the dominated convergence theorem.\\
To examine the case where $\lbrace 0 \rbrace \times S^3 \subset K$ we first show that,  as currents, we have
\begin{equation}\label{eq:Currents_Same_Limit}
\lim_{x_1 \rightarrow + \infty} (\vert F_{A^{x_1}} \vert^2 - \vert F_{A_{\lim}} \vert^2 ) = \lim_{x_1 \rightarrow + \infty}  \frac{ 6 \cdot 2592  x_1^2 }{(r+1)^4 r^6 (r^2 x_1 -x_1 +3)^4}. 
\end{equation}
Recall from $\S$\ref{sss:BS} that we can identify $(\mathbb{R}^4\setminus\{0\})\times S^3\cong (1,\infty)\times S^3\times S^3$, with $r$ the coordinate on $(1,\infty)$.  For $f \in C^{\infty}_c(\mathbb{R}^4 \times S^3 , \mathbb{R})$ we can then compute that 
\begin{align}\nonumber
 \Big\vert \int_{\mathbb{R}^4 \times S^3}  &f \frac{(r-1)^k x_1^n}{(r+1)^4 r^6 (r^2 x_1 -x_1 +3)^4} \dvol_g  \Big\vert  \\ 
&  \leq  \Vert f \Vert_{L^{\infty}} \int_{1}^{+\infty}   \frac{(r-1)^k x_1^n}{(r+1)^4 r^6 (r^2 x_1 -x_1 +3)^4} \frac{r^6(1-r^{-3}) Vol(\mathbb{S}^3_1)^2}{3^4 \sqrt{3}} dr ,\label{eq:big.int}
\end{align}
where $Vol(\mathbb{S}^3_1)=2\pi^2$ denotes the volume of the unit $3$-sphere in $\mathbb{R}^4$. Let $I_{n,k}(x_1)$ denote the integral on the right-hand side of \eqref{eq:big.int} and let $\epsilon>0$. To examine $I_{n,k}(x_1)$, we separate it into two integrals: one over $[1,1+\epsilon]$ and the other over $[1+\epsilon, + \infty)$. The second of these integrals can be easily seen to be finite and of order $x_1^{n-4}$ (independently of $k$), hence it vanishes as $x_1 \rightarrow + \infty$ since $n\leq 3$. The first integral over $[1,1+\epsilon]$, and thus $I_{n,k}(x_1)$ by the preceding argument, can be bounded as follows for some constant $c$: 
\begin{eqnarray}\label{eq:big.int2}
I_{n,k}(x_1) \leq c x_1^{n-4} \int_1^{1+\epsilon} \frac{(r-1)^{k+1}}{(r- \sqrt{1-3/x_1})^4} dr + O(x_1^{n-4}).
\end{eqnarray}
The integral on the right-hand side of \eqref{eq:big.int2} can now be computed to be of order $O(x_1^{2-k})$ for $k=0,1$, $O(\log(x_1))$ for $k=2$ and $O(1)$ for $k>2$. Letting $x_1 \rightarrow + \infty$ we see that \eqref{eq:big.int2} vanishes unless $k=0$ and $n=2,3$, or $k=1$ and $n=3$. From \eqref{eq:q.values} we then see that \eqref{eq:Currents_Same_Limit} holds as desired.\\
Now suppose $\lbrace 0 \rbrace \times S^3 \subset K$. Rewriting  \eqref{eq:Limit} we have from the first part of the proof that
\begin{align*}
\lim_{x_1 \rightarrow + \infty} \int_K (\vert F_{A^{x_1}} \vert^2 - \vert F_{A_{\lim}} &\vert^2) \dvol_g\\
& = \lim_{x_1 \rightarrow + \infty} \int_{K \cap B_{\epsilon}(0 \times S^3)} (\vert F_{A^{x_1}} \vert^2 - \vert F_{A_{\lim}} \vert^2) \dvol_g.
\end{align*}
 Hence, using \eqref{eq:Currents_Same_Limit} to compute the integral gives
\begin{align}\nonumber
\lim_{x_1 \rightarrow + \infty}&\int_{K} \frac{ 6 \cdot 2592  x_1^2 }{(r+1)^4 r^6 (r^2 x_1 -x_1 +3)^4} \dvol_g \\ \nonumber
& = \lim_{x_1 \rightarrow + \infty} \int_1^{1+\epsilon} \frac{ 6 \cdot 2592  x_1^2 (1-r^{-3}) }{(r+1)^4 r^6 (r^2 x_1 -x_1 +3)^4}  \frac{2^3 Vol(\mathbb{S}^3_1)^2}{3^4 \sqrt{3}} dr \\ \nonumber
& = \frac{4}{3\sqrt{3}} Vol(\mathbb{S}_1^3)^2 = \frac{8 \pi^2}{3 \sqrt{3}} Vol(\mathbb{S}_1^3) .
\end{align} 
Now recall that the metric at the zero section is $\left( \frac{2}{\sqrt{3}} \right)^2 (\eta_1^- \otimes \eta_1^- + \eta_2^- \otimes \eta_2^- + \eta_3^- \otimes \eta_3^-)$, hence its volume form is $1 / 3\sqrt{3}$ times that of $\mathbb{S}^3_1$. The result then follows.
\end{proof}

\begin{remark}
The sequence of instantons with curvature concentrating along the associative $S^3$ determines a Fueter section, as in \cites{Donaldson2009,Haydys2011,Walpuski2017}, from $S^3$ to the bundle of moduli spaces of anti-self-dual connections associated to the normal bundle.   The section thus determined is constant, taking value at the basic instanton on $\mathbb{R}^4$.  The Yang--Mills energy of the basic instanton is $8\pi^2$, so Corollary \ref{cor:delta} confirms the expected ``conservation of energy'' formula (c.f.~\cites{Tian2000}).
\end{remark}

\section[\texorpdfstring{{\boldmath $SU(2)^2\times U(1)$}-invariant {\boldmath $G_2$}-instantons}{SU(2)xSU(2)xU(1)-invariant G2-instantons}]{{\boldmath $SU(2)^2\times U(1)$}-invariant {\boldmath $G_2$}-instantons}\label{sec:SU(2)^2xU(1)invariant}

\noindent The main goal of this section is to investigate $SU(2)^2$-invariant $G_2$-instantons on the Brandhuber et al.~(BGGG) $G_2$-manifold $\mathbb{R}^4\times S^3$ from $\S$\ref{sss:BGGG}.  We will restrict ourselves to instantons that enjoy an extra $U(1)$-symmetry present in the underlying geometry. 
As already mentioned, all of the known complete $SU(2)^2$-invariant $G_2$-manifolds of cohomogeneity-$1$ 
 enjoy an extra $U(1)$-symmetry and so the analysis of $SU(2)^2\times U(1)$-invariant $G_2$-instantons 
  provides a natural stepping stone to a complete understanding of $SU(2)^2$-invariant $G_2$-instantons.\\
We begin in $\S$\ref{ss:BGGGODEs} by deriving the ODEs determining $G_2$-instantons in this setting by simplifying the general ODEs and constraint in Lemma \ref{lem:ODE}.  We then  determine the necessary 
conditions ensuring that solutions to these ODEs smoothly extend across 
the singular orbit in the Bryant--Salamon (BS), BGGG and Bogoyavlenskaya $G_2$-manifolds in 
$\S$\ref{ss:BGGGinits}.  In the final section $\S$\ref{ss:solutions.BGGG}, we explicitly describe the $G_2$-instantons which exist near the singular orbit. This leads to a stronger classification result in the BS case, and existence and non-existence results for global $G_2$-instantons in the BGGG case.

\subsection[\texorpdfstring{The $SU(2)^2\times U(1)$-invariant ODEs}{The SU(2)xSU(2)xU(1)-invariant ODEs}]{The {\boldmath $SU(2)^2\times U(1)$}-invariant ODEs}\label{ss:BGGGODEs}

 We shall now rewrite the ODEs from Lemma \ref{lem:ODE} in this $SU(2)^2\times U(1)$-invariant setting. 
 As for the $SU(2)^3$-invariant case it will be convenient to rescale the fields $a_i^{\pm}$, for $i=1,2,3$, defining the connection 1-form as in \eqref{eq:InvariantConnection}. We thus define
$$c^+_i= \frac{a^+_i}{A_i}, \quad c^-_i = \frac{a^-_i}{B_i}$$
so that the connection 1-form is 
$$A=\sum_{i=1}^3 A_ic^+_i\otimes \eta_i^++B_ic^-_i\otimes\eta_i^-.$$
In these terms we can use \eqref{eq:dotA1}-\eqref{eq:dotB2} to obtain the general $SU(2)^2\times U(1)$-invariant $G_2$-instanton equations for $A$ from Lemma \ref{lem:ODE} as follows:
\begin{align}
\dot{c}_1^+ + 
\frac{1}{2}\left(\frac{A_1}{B_2^2}-\frac{A_1}{A_2^2}\right)c_1^+ & = 
 \frac{1}{2} [c_2^-, c_3^-]-\frac{1}{2} [c_2^+, c_3^+],\label{eq:dotc1+} \\ 
\dot{c}_2^+ + 
\frac{1}{2}\left(\frac{A_2^2+B_1^2+B_2^2}{A_2B_1B_2}-\frac{A_1^2+2A_2^2}{A_1A_2^2}\right) c_2^+ & = 
\frac{1}{2} [c_3^- , c_1^-]-\frac{1}{2} [c_3^+ , c_1^+],\label{eq:dotc2+} \\ 
\dot{c}_3^+ + 
\frac{1}{2}\left(\frac{A_2^2+B_1^2+B_2^2}{A_2B_1B_2}-\frac{A_1^2+2A_2^2}{A_1A_2^2}\right)c_3^+ & = 
\frac{1}{2} [c_1^- , c_2^-]-\frac{1}{2} [c_1^+ , c_2^+] ,\label{eq:dotc3+}\\
\dot{c}_1^- + 
\left(\frac{A_2^2+B_1^2+B_2^2}{A_2B_1B_2}\right) c_1^- & =
\frac{1}{2} [c_2^-, c_3^+]+\frac{1}{2} [c_2^+, c_3^-],\label{eq:dotc1-} \\ 
\dot{c}_2^- + 
\frac{1}{2}\left(\frac{A_2^2+B_1^2+B_2^2}{A_2B_1B_2}+\frac{A_1^2+2B_2^2}{A_1B_2^2}\right)c_2^- & =  
\frac{1}{2} [c_3^- , c_1^+] +\frac{1}{2} [c_3^+ , c_1^-],\label{eq:dotc2-} \\ 
\dot{c}_3^- + 
\frac{1}{2}\left(\frac{A_2^2+B_1^2+B_2^2}{A_2B_1B_2}+\frac{A_1^2+2B_2^2}{A_1B_2^2}\right)c_3^- & = 
\frac{1}{2} [c_1^- , c_2^+] +\frac{1}{2} [c_1^+ , c_2^-] ,\label{eq:dotc3-}
\end{align}
together with the constraint
\begin{equation}\label{eq:c.constraint}
\sum_{i=1}^3 [c_i^+ , c_i^-]=0.
\end{equation}
We now wish to simplify these equations further using an additional $U(1)$-symmetry in the ambient geometry. This extra symmetry in the 
known complete $SU(2)^2$-invariant cohomogeneity-$1$ $G_2$-manifolds from $\S$\ref{sec:G2metrics} 
can be encoded, for example, by regarding $S^3 \times S^3$ as $SU(2)^2 \times U(1) / \Delta U(1)$, with $\Delta U(1)$ acting via $$e^{i \theta} \cdot (g_1,g_2,e^{i \alpha})= (g_1 \diag (e^{i \theta}, e^{-i \theta})
, g_2 \diag (e^{i \theta}, e^{-i \theta})   
,  e^{i (\alpha+ \theta)}).$$ 
With this in hand, we can derive the simplified ODEs in this setting.

\begin{proposition}\label{prop:ODEsSU(2)xU(2)}
Let $A$ be an $SU(2)^2 \times U(1)$-invariant  $G_2$-instanton on 
$\mathbb{R}^+\times SU(2)^2\cong \mathbb{R}^+\times (SU(2)^2\times U(1)/\Delta U(1))$ with gauge group $SU(2)$. 
There is a standard basis  $\lbrace T_i \rbrace_{i=1}^3$ of $\mathfrak{su}(2)$, i.e.~with $[T_i,T_j]=2\epsilon_{ijk} T_k$, such that (up to an invariant gauge transformation) we can write
\begin{eqnarray}\label{eq:InvariantConnection.BGGG}
A & = &  A_1 f^{+} T_1\otimes \eta_1^{+}  + A_2g^{+} (T_2\otimes\eta_2^{+}  + T_3\otimes\eta_3^{+} ) \\ \nonumber
& &
 + B_1 f^{-} T_1\otimes \eta_1^{-} + B_2 g^{-} (T_2\otimes\eta_2^{-}  + T_3\otimes\eta_3^{-}),
\end{eqnarray}
with $f^{\pm}$, $g^{\pm}: \mathbb{R}^+ \rightarrow \mathbb{R}$ satisfying 
\begin{align}\label{eq:dotf+}
\dot{f}^+ + 
\frac{1}{2}\left(\frac{A_1}{B_2^2}-\frac{A_1}{A_2^2}\right)f^+ & =  (g^-)^2 - (g^+)^2,\\ 
\label{eq:dotg+}
\dot{g}^+ +  
\frac{1}{2}\left(\frac{A_2^2+B_1^2+B_2^2}{A_2B_1B_2}-\frac{A_1^2+2A_2^2}{A_1A_2^2}\right)g^+ & =  f^-g^- - f^+ g^+ ,\\
\label{eq:dotf-}
\dot{f}^- +  
\left(\frac{A_2^2+B_1^2+B_2^2}{A_2B_1B_2}\right) f^- & =  2g^- g^+, \\ 
\label{eq:dotg-}
\dot{g}^- + 
\frac{1}{2}\left(\frac{A_2^2+B_1^2+B_2^2}{A_2B_1B_2}+\frac{A_1^2+2B_2^2}{A_1B_2^2}\right)g^- & =  g^- f^+ + g^+ f^- .
\end{align}
\end{proposition}
\begin{proof}
We must consider $SU(2)^2 \times U(1)$-homogeneous $SU(2)$-bundles over $S^3 \times S^3 \cong SU(2)^2 \times U(1) / \Delta U(1)$. Such bundles are parametrized by isotropy homomorphisms $\lambda: \Delta U(1) \rightarrow SU(2)$, which take the form $\lambda_k(e^{i \theta})= \diag(e^{ik \theta}, e^{-i k \theta})$. We take the complement of the isotropy algebra $\Delta \mathfrak{u}(1)$ to be $\mathfrak{m}= \mathfrak{su}^+(2) \oplus \mathfrak{su}^-(2) \oplus 0$. The canonical invariant connection on the bundle $$P_k = ( SU(2)^2 \times U(1) ) \times_{(\Delta U(1), \lambda_k)} SU(2)$$ is given by $d \lambda_k = T_1\otimes k d \theta$, where the $\lbrace T_i \rbrace_{i=1}^3$ form a standard basis for $\mathfrak{su}(2)$ and $\theta$ is the periodic coordinate on $U(1)$. Wang's theorem \cite{Wang1958} states that any other invariant connection $a$ on $P_k$ can be written as $d \lambda_k + \Lambda_k$, where $\Lambda_k$ is the left-invariant extension to $SU(2)^2 \times U(1)$ of a morphism of $\Delta U(1)$-representations $\Lambda_k : (\mathfrak{m} , \Ad) \rightarrow (\mathfrak{su}(2) , \Ad \circ \lambda_k)$. Splitting into irreducibles, we have
$$\mathfrak{m} = (\mathbb{R} \oplus \mathbb{C}_2) \oplus ( \mathbb{R} \oplus \mathbb{C}_2) \oplus 0,$$
while $(\mathfrak{su}(2) , \Ad \circ \lambda_k)$ splits as $\mathbb{R}\oplus\mathbb{C}_{2k}$. 
 Therefore, other invariant connections exist only when $k=1$, in which case we can apply a gauge transformation so that
\begin{eqnarray}\nonumber
a & = & T_1\otimes d \theta    
+   A_1 f^{+} T_1\otimes \eta_1^{+} + A_2 g^{+} (T_2\otimes\eta_2^{+}  +  T_3\otimes \eta_3^{+} ) \\ \nonumber
&  & \qquad\quad\,\,\, +\,   B_1 f^{-} T_1\otimes \eta_1^{-} + B_2 g^{-} ( T_2\otimes\eta_2^{-}  + T_3\otimes\eta_3^{-} ),
\end{eqnarray}
where $f^{\pm}$, $g^{\pm}$ are constants. We now pull this back to $SU(2)^2$ via the inclusion map $SU(2)^2 \rightarrow SU(2)^2 \times U(1)$ and extend it to 
$\mathbb{R}^+\times SU(2)^2$ to  obtain
\begin{eqnarray}\nonumber
A & = &  \gamma\big(A_1 f^{+} T_1\otimes \eta_1^{+} + A_2 g^{+} (T_2\otimes\eta_2^{+}  +  T_3\otimes \eta_3^{+} )\big)\gamma^{-1} \\ \nonumber
& & + \gamma\big( B_1 f^{-} T_1\otimes \eta_1^{-} + B_2 g^{-} (T_2\otimes\eta_2^{-}  +  T_3\otimes\eta_3^{-} ) \big)\gamma^{-1} ,
\end{eqnarray}
for functions $\gamma:\mathbb{R}^+\to SU(2)$ and $f^{\pm},g^{\pm}:\mathbb{R}^+\to\mathbb{R}$.
\\
We now turn our attention to such connections $A$ which can solve the $G_2$-instanton equations \eqref{eq:dotc1+}-\eqref{eq:c.constraint}. 
We first see that the constraint \eqref{eq:c.constraint} is satisfied and 
 we claim that the evolution equations imply the ODEs \eqref{eq:dotf+}-\eqref{eq:dotg-} and that $\dot{\gamma}=0$.  
  Observe that  \eqref{eq:dotc1+} becomes
$$\dot{f}^+ T_1 + f^+ [\gamma^{-1}\dot{\gamma},T_1] + \frac{1}{2}\left(\frac{A_1}{B_2^2}-\frac{A_1}{A_2^2}\right)  f^+ T_1 =  ((g^{-})^2 - (g^{+})^2 ) T_1 .  $$
We conclude that $f^+[\gamma^{-1}\dot{\gamma},T_1]=0$  
and obtain \eqref{eq:dotf+}. 
Entirely analogous computations yield 
$$g^+[\gamma^{-1}\dot{\gamma},T_i]=0,\quad f^-[\gamma^{-1}\dot{\gamma},T_1]=0,\quad g^-[\gamma^{-1}\dot{\gamma},T_i]=0$$
for $i=2,3$, as well as \eqref{eq:dotg+}-\eqref{eq:dotg-}.  Hence, if $A\neq 0$ we 
obtain $\dot{\gamma}=0$ and so  
$A$ takes the form \eqref{eq:InvariantConnection.BGGG} as required.
\end{proof}

\begin{remark}
In the setup of the Proposition \ref{prop:ODEsSU(2)xU(2)}, we have
$$c_1^{\pm}= f^{\pm} T_1 , \ c_2^{\pm}=g^{\pm}  T_2, \ c_3^{\pm}=g^{\pm} T_3,$$
where $f^{\pm}, g^{\pm} : \mathbb{R}^+ \rightarrow \mathbb{R}$ satisfy the ODEs \eqref{eq:dotf+}-\eqref{eq:dotg-}.
\end{remark}

\subsection{Initial conditions}\label{ss:BGGGinits}
  
To investigate $SU(2)^2 \times U(1)$-invariant $G_2$-instantons $A$ on the BGGG $G_2$-manifold $\mathbb{R}^4\times S^3$, as well as the BS and Bogoyolavenskaya cases, we study the conditions for $A$ to extend smoothly over the singular orbit $SU(2)^2 / \Delta SU(2) \cong  \{0\}\times S^3 $. \\
As in $\S$\ref{ss:BSinits}, we have two bundles $P_{\lambda}$ as in \eqref{eq:Plambda} to consider,  
where $\lambda: \Delta SU(2) \rightarrow SU(2)$ is either trivial $\lambda=1$ or the identity $\lambda=\id$. Recall from Proposition \ref{prop:ODEsSU(2)xU(2)} that $A$ takes the form in \eqref{eq:InvariantConnection.BGGG}, determined by functions $f^{\pm},g^{\pm}:\mathbb{R}_{\geq 0}\to\mathbb{R}$. The next result gives the conditions on $f^{\pm}, g^{\pm}$ so that such $A$ extends over a singular orbit at $t=0$. To state it, we observe that Lemma \ref{lem:ExtendingSmoothlyMetric} in Appendix \ref{app:R4xS3} shows that for any $SU(2)^2 \times U(1)$-invariant $G_2$-metric which smoothly extends over a singular orbit at $t=0$ must be of the form \eqref{eq:metric} for functions $A_1$, $A_2=A_3$, $B_1$, $B_2=B_3$ which admit Taylor expansions of the form 
\begin{equation}\label{eq:metric.expansion}
A_i(t)=\frac{t}{2} + t^3 C_i(t)\quad\text{and}\quad B_i(t) = b + t^2 D_i(t),
\end{equation} for some $b \in \mathbb{R}\setminus\{0\}$ and real analytic even functions $C_1$, $C_2=C_3$, $D_1$, $D_2=D_3$ with $D_1(0)=D_2(0)$.  
The explicit values of $C_1(0)$, $C_2(0)$ and $D_1(0)$ for the BS and BGGG cases can be found in Examples \ref{ex:BS}-\ref{ex:BGGG} in Appendix \ref{app:R4xS3}.

\begin{lemma}\label{lem:SmoothlyExtendBGGG}
The connection $A$ in \eqref{eq:InvariantConnection.BGGG} extends smoothly over the singular orbit $S^3$ if and only if  $f^+$ and $g^+$ are odd, $f^-$ and $g^-$ are even, and,  using the notation in \eqref{eq:metric.expansion}, their Taylor expansions around  $t=0$ are:
\begin{itemize}\item either
\begin{align*}
f^{-}&= 
f_2^- t^2 + O(t^4), &  g^-&= 
g_2^- t^2 + O(t^4),\\ \nonumber
f^{+}&= f_1^+ t + O(t^3), &  g^+&= g_1^+ t + O(t^3),
\end{align*}
in which case $A$ extends smoothly as a connection on $P_{1}$; 
\item or
\begin{align*}
f^{-}&= b_0^- + b_2^- t^2 + O(t^4) , &  g^-&= b_0^- + b_2^- t^2 + O(t^4), \\ \nonumber
f^{+}&= \frac{2}{t} + (b_2^+-4C_1(0) ) t + O(t^3) , &   g^+&= \frac{2}{t} + (b_2^+-4C_2(0)) t + O(t^3),
\end{align*}
in which case $A$ extends smoothly as a connection on $P_{\id}$.
\end{itemize}
\end{lemma}
\begin{proof}  
We start with $P_{\id}$, i.e.~where the homomorphism $\lambda=\id$, as it is slightly more involved. The canonical invariant connection on $P_{\id} \rightarrow S^3$ is $A^{\text{can}}$ in \eqref{eq:A.can}. 
 We must then apply Lemma \ref{lem:ExtendConnectionOnBS} from Appendix \ref{app:R4xS3} to the $\mathfrak{su}(2)$-valued $1$-form
\begin{eqnarray}\nonumber
A-A^{\text{can}} & = & (A_1 f^{+}-1)T_1\otimes \eta_1^{+}  + (A_2 g^+ -1) (T_2\otimes\eta_2^{+} + T_3\otimes\eta_3^{+} ) \\ \nonumber
& & + B_1 f^{-} T_1\otimes\eta_1^{-} + B_2 g^- (T_2\otimes\eta_2^{-}  + T_3\otimes\eta_3^{-}).
\end{eqnarray}
We deduce that  $A_1 f^{+}-1$, $A_2 g^+ -1$, $B_1 f^{-}$ and $B_2 g^-$ are all even. Moreover, the first two of these must admit Taylor expansions of the form $\frac{b_2^+}{2} t^2 + O(t^4)$, for some $b_2^+ \in \mathbb{R}$, while the last two have expansions of the form $\frac{b_0^-}{2} + \frac{b_2^-}{2} t^2 + O(t^4)$, for some $b_0^{-}$, $b_2^{-} \in \mathbb{R}$. 
Using \eqref{eq:metric.expansion}, we deduce that
$$f^{+}= \frac{2}{t} + (b_2^+ -4C_1(0) ) t + O(t^3), \quad g^+= \frac{2}{t} + (b_2^+ -4C_2(0)) t + O(t^3),$$
and, since $D_1(0)=D_2(0)$ and $b\neq 0$, we have $f^-=h^-+O(t^4)$, $g^-=h^-+O(t^4)$ where
$$h^-=\frac{b_0^-}{2b} + \left( \frac{b_2^-}{2b} -\frac{b_0^-}{2b^2} D_1(0)   \right)t^2.$$
The statement for $P_{\id}$ then follows.\\
We now turn to $P_1$.  Here we instead apply Lemma \ref{lem:ExtendConnectionOnBS} to the $1$-form $A$ itself and conclude that $A_1 f^{+}$, $A_2 g^+$, $B_1 f^{-}$, $B_2 g^-$ must all be even and vanish at $t=0$.  
Hence, by \eqref{eq:metric.expansion}, we see that $f^+,g^+$ are odd and $f^-,g^-$ are even such that $f^-(0)=g^-(0)=0$.
\end{proof}

\begin{remark}
We have also determined the conditions for smooth extension to the singular orbit on the Bazaikin--Bogoyavlenskaya (BB)
 $G_2$-manifolds from $\S$\ref{sss:BB}.  This is more involved and we shall report on it in future work.
\end{remark}

\subsection{Solutions}\label{ss:solutions.BGGG}

We now investigate  existence of solutions of the $SU(2)^2\times U(1)$-invariant $G_2$-instanton equations with gauge group $SU(2)$ on the BS, BGGG and Bogoyavlenskaya $G_2$-manifolds $\mathbb{R}^4\times S^3$.  There are two cases: when the bundle is $P_1$ or $P_{\id}$, in the notation of the previous subsection.  
In both cases we explicitly classify the invariant $G_2$-instantons 
defined near the singular orbit which extend smoothly and, as a consequence, extend our uniqueness result for $G_2$-instantons on the BS $\mathbb{R}^4\times S^3$ to the case of $SU(2)^2 \times U(1)$-symmetry, and  obtain both \emph{existence} and \emph{non-existence} results 
for $G_2$-instantons  with decaying curvature on the BGGG $\mathbb{R}^4\times S^3$.

\subsubsection[\texorpdfstring{Solutions smoothly extending on $P_{1}$}{Solutions smoothly extending on P1}]{Solutions smoothly extending on {\boldmath $P_{1}$}}

We shall now investigate the existence of solutions that smoothly extend over the singular orbit $S^3 = SU(2)^2 / \Delta SU(2)$ on the bundle $P_1$. The main results are Proposition \ref{prop:LocalG2Instantons} and Theorems \ref{thm:BS.cor}--\ref{thm:BGGG.Inst}.  Proposition \ref{prop:LocalG2Instantons} shows the existence of a $2$-parameter family of $G_2$-instantons in a neighborhood of the singular orbit, so there is at most a 2-parameter family of $SU(2)^2\times U(1)$-invariant $G_2$-instantons on $P_1$ on the BS, BGGG and Bogoyavlenskaya $G_2$-manifolds.
Theorem \ref{thm:BS.cor} shows that in the BS case, just a 1-parameter family of these local instantons extends, and these are either given by Clarke's $SU(2)^3$-invariant examples from Theorem \ref{prop:Clarke} or are abelian.     Theorems \ref{thm:BGGG.NoInst} and \ref{thm:BGGG.Inst} 
show that, unlike the BS case, there is a 2-parameter family of local $G_2$-instantons which  extend to the whole BGGG $\mathbb{R}^4\times S^3$ so that their curvature is bounded, as well as a 2-parameter family which do not extend so as to have bounded curvature.

\begin{proposition}\label{prop:LocalG2Instantons}
Let $X \subset \mathbb{R}^4\times S^3 $ contain the singular orbit $\{0\}\times S^3 $ of the $SU(2)^2 \times U(1)$ action and be equipped with an $SU(2)^2 \times U(1)$-invariant holonomy $G_2$-metric.  There is a $2$-parameter family of $SU(2)^2 \times U(1)$-invariant $G_2$-instantons $A$ with gauge group $SU(2)$ in a neighbourhood of the singular orbit in $X$  smoothly extending over $P_{1}$.\\
Moreover, in the notation of Proposition \ref{prop:ODEsSU(2)xU(2)} and \eqref{eq:metric.expansion}, any such $G_2$-instanton $A$ can be written as in \eqref{eq:InvariantConnection.BGGG}  with $f^-=0=g^-$ and with 
$f^+$, $g^+$ solving the ODEs:
\begin{eqnarray}\label{eq:InstantonODE1}
\dot{f}^+ +  
\frac{1}{2}\left(\frac{A_1}{B_2^2}-\frac{A_1}{A_2^2}\right)f^+ & = & - (g^+)^2, \\ \label{eq:InstantonODE2}
\dot{g}^+ +  
\frac{1}{2}\left(\frac{A_2^2+B_1^2+B_2^2}{A_2B_1B_2}-\frac{A_1^2+2A_2^2}{A_1A_2^2}\right)g^+ & = & - f^+ g^+ ,
\end{eqnarray}
subject to $f^+(t)=f_1^+ t+ t^3 u_1(t)$, $g^+(t)=g_1^+ t+ t^3 u_2(t)$,
where $f_1^+,g_1^+ \in \mathbb{R}$ and the $u_i$ are real analytic functions such that
\begin{eqnarray}\label{eq:u1}
u_1(0) &=& -f_1^+\left( \frac{1}{8 b^2} + 2C_2(0)-C_1(0) \right) -\frac{(g^+_1)^2}{2} , \\ 
u_2(0)&=& -\frac{g_1^+}{2}\left(\frac{1}{4 b^2} + 2 C_1(0) + f_1^+ \right).\label{eq:u2}
\end{eqnarray}
\end{proposition}
\begin{proof}
It is convenient to study our initial value problem by writing 
\begin{align*}
f^+(t)&=f_1^+ t+ t^3 u_1(t), & g^+(t)&=g_1^+ t+ t^3 u_2(t),\\
f^-(t)&= 
t^2 v_1(t), & g^{-}(t)&= t^2 v_2(t),
\end{align*}
for some real analytic functions $u_1,u_2,v_1,v_2$, which we can do by Lemma \ref{lem:SmoothlyExtendBGGG}. In this way the ODEs for $G_2$-instantons from Proposition \ref{prop:ODEsSU(2)xU(2)} turn into ODEs for these $4$ functions, which we write as $X(t)=(u_1(t),u_2(t), v_1(t), v_2(t))$.  A lengthy but otherwise straightforward computation yields 
the regular singular initial value problem at $t=0$:
$$\frac{dX}{dt}= \frac{M_{-1}(X)}{t}+M(t,X),$$
where $M(t,X)$ is real analytic in the first coordinate and
\begin{align*}
M_{-1}(X)=\Big( &  -2 u_1 - \left( \frac{1}{4 b^2} + 4C_2(0)-2C_1(0) \right) f_1^+ - (g_1^+)^2 ,\\
 &-2u_2 - \left(\frac{1}{4 b^2} + 2C_1(0)+f_1^+ \right) g_1^+ , -6 v_1, -6 v_2 \Big).
  \end{align*}  
The existence and uniqueness theorem for singular initial value problems  
(\cite{Malgrange1974}, see also Theorem 4.7 in \cite{Foscolo2015} for a clearer statement) applies if and only if $M_{-1}(X(0))=0$ and $dM_{-1}(X(0))$ does not have any positive integer as an eigenvalue.   Since $dM_{-1}(X(0))$ is diagonal with eigenvalues $-2,-2,-6,-6$,
 we only need 
$v_1(0)=0=v_2(0)$ and $u_1(0)$, $u_2(0)$ as in \eqref{eq:u1}-\eqref{eq:u2} to apply the existence and uniqueness theorem: this determines the possible initial values $X(0)$, which are therefore parametrized by  $f_1^+ , g_1^+ \in \mathbb{R}$.  
 We conclude that there is a local two-parameter family of $G_2$-instantons with $SU(2)^2 \times U(1)$-symmetry as claimed.\\ 
Notice that all these $G_2$-instantons have  $v_1(0)=0=v_2(0)$.  
Thus, setting the smaller singular initial value problem above with $f^-$ and $g^-$ both vanishing gives  
the same local existence and uniqueness result, and hence the uniqueness  implies that in fact $f^-(t)$, $g^-(t)$ must vanish identically for any solution extending smoothly over the singular orbit.  The resulting ODEs \eqref{eq:InstantonODE1}-\eqref{eq:InstantonODE2} then follow from Proposition \ref{prop:ODEsSU(2)xU(2)}. 
\end{proof}

\begin{remark}
Recall that the BS, BGGG and Bogoyavlenskaya  
$G_2$-metrics all have $SU(2)^2 \times U(1)$-symmetry and so Proposition \ref{prop:LocalG2Instantons} yields $G_2$-instantons in these cases.
\end{remark}

Our first result shows that the sign of $g_1^+$ determines the sign of $g^+$.

\begin{lemma}\label{lem:g+sign}  Let $(f^+,g^+)$ solve \eqref{eq:InstantonODE1}--\eqref{eq:InstantonODE2}.  
The sign of $g^+$ does not change as long as $f^+$ does not blow up, and if $g^+(t_0)=0$ for some $t_0> 0$ or if $g_1^+=0$ then $g^+\equiv 0$.  
\end{lemma}
\begin{proof}
Suppose, for a contradiction, that the sign of $g^+$ changes.  Then there is $t_0> 0$ such that $g^+(t_0)=0$. 
The ODE \eqref{eq:InstantonODE2} implies that 
$\dot{g}^+(t_0)=0$ and thus $g^+\equiv 0$ (as $g^+$ solves a linear first order ODE), giving our contradiction.  The same argument using \eqref{eq:InstantonODE2} yields the statement. 
\end{proof}

\begin{remark}\label{g+sign.rmk}
The ODEs \eqref{eq:InstantonODE1}--\eqref{eq:InstantonODE2} are invariant under $g^+\mapsto -g^+$. We may therefore   exchange $g^+$ with $-g^+$ and, by virtue of Lemma \ref{lem:g+sign}, assume that  $g_1^+\geq 0$, and thus $g^+\geq 0$,  if we wish.
\end{remark}

\noindent We first focus on the BS $G_2$-manifold $\mathbb{R}^4\times S^3$.  It follows from Proposition \ref{prop:LocalG2Instantons}  that there is at most a 2-parameter family of $SU(2)^2\times U(1)$-invariant $G_2$-instantons defined globally on the BS $G_2$-manifold.  We have a 1-parameter family of such instantons (with more symmetry) from Theorem \ref{thm:Clarke} and a 1-parameter family of abelian examples from Corollary \ref{cor:U1}.  We now show that these examples provide a complete classification.

\begin{theorem}\label{thm:BS.cor}
Let $A$ be a $SU(2)^2\times U(1)$-invariant $G_2$-instanton with gauge group $SU(2)$ on the BS $G_2$-manifold $\mathbb{R}^4\times S^3$ which extends smoothly on $P_1$.  Either $A$ is $SU(2)^3$-invariant, and so is given in Theorem \ref{thm:Clarke}; or it is reducible, in which case it has gauge group $U(1)$ and is given in Corollary \ref{cor:U1}(a) with $x_2=x_3=0$.
\end{theorem}

\begin{proof}
In the BS case, using \eqref{eq:ODEsMetricBryantSalamon}, we see that \eqref{eq:InstantonODE1}-\eqref{eq:InstantonODE2} are now
\begin{align}\label{eq:BS.cor.1}
\dot{f}^+-\frac{\dot{A}_1}{A_1}f^+=-(g^+)^2\quad\text{and}\quad
\dot{g}^+-\frac{\dot{A}_1}{A_1}g^+=-f^+g^+.
\end{align}
Let $F=f^+/A_1$ and $G=g^+/A_1$, and define $s\in [0,\infty)$ by $\frac{ds}{dt}=A_1$.  If we let $f'=\frac{df}{ds}$ then \eqref{eq:BS.cor.1} is equivalent to
\begin{equation}\label{eq:BS.cor.2}
F'=-G^2\quad\text{and}\quad G'=-FG.
\end{equation}
It follows from \eqref{eq:BS.cor.2} that $F^2-G^2=c\in\mathbb{R}$, so we need only consider the ODE
\begin{equation}\label{eq:BS.cor.3}
F'=c-F^2.
\end{equation}
  
\noindent If $c<0$, the solutions to \eqref{eq:BS.cor.3} satisfy $F^2(s)=-c\tan^2(a-\sqrt{-c} s)$, which can then only exist for finite $s$ and thus finite $t$. 

\noindent If $c>0$ there are two types of solutions to \eqref{eq:BS.cor.3}: either  $F^2(s)=c\tanh^2(a+\sqrt{c}s)$ or $F^2=c$.  The first solutions have $F^2-c<0$ which contradicts $F^2-c=G^2$. The second solutions force $G\equiv 0$, which give  abelian instantons as  in Corollary \ref{cor:U1}. 

\noindent If $c=0$, then $F^2=G^2$, which means $F=\pm G$ so $f^+=\pm g^+$.  By Remark \ref{g+sign.rmk}, we may assume that $f^+=g^+$. In this case, $A$ is $SU(2)^3$-invariant and the result then follows from Theorem \ref{thm:Clarke}.
\end{proof}
 
\noindent We now focus attention on the BGGG $G_2$-manifold, though some of our results hold for the 1-parameter family of Bogoyavlenskaya metrics which includes the BGGG metric.
It is 
natural in the study of $G_2$-instantons on non-compact $G_2$-manifolds to
 assume a decay condition on the curvature of the connection at infinity. 
The weakest reasonable assumption we can make is the curvature is bounded.  
  In this setting we can prove both existence and non-existence results.

\noindent We first observe the conditions imposed on $f^+,g^+$ for the Bogoyavlensakaya metrics when the curvature is bounded.

\begin{lemma}\label{lem:CurvatureBound}
Let $A$ be the $G_2$-instanton on one of the Bogoyavlenskaya $G_2$-manifolds induced by the pair $(f^+,g^+)$ as in Proposition \ref{prop:LocalG2Instantons}. Then $|F_A|$ is bounded only if $g^+$ is bounded, and if both $f^+$ and $g^+$ are bounded then $|F_A|$ is bounded.  
\end{lemma}
\begin{proof}
The $G_2$-instanton $A$ induced by the pair $(f^+,g^+)$ has connection form as in \eqref{eq:InvariantConnection.BGGG} with $f^-=g^-=0$.  
Since $A=a(t)$, the curvature $F_A = dt \wedge \dot{a} + F_a$ of $A$ can be computed from Lemma \ref{lem:Curvature}. Notice that $\vert F_A \vert^2 = \vert \dot{a} \vert^2 + \vert F_a \vert^2$. Computing each of these terms separately we have
\begin{align*}
\vert F_a \vert^2 & =  \frac{1}{4} \left( (g^+)^2 - \frac{A_1}{A_2^2} f^+ \right)^2 + \frac{(g^+)^2}{2} \left( f^+-\frac{1}{A_1} \right)^2 
+ \frac{A_1^2 (f^+)^2}{4 B_2^4} + \frac{A_2^2 (g^+)^2}{2 B_1^2 B_2^2},\\
\vert \dot{a} \vert^2 & =  \frac{1}{4} \left( (g^+)^2-\frac{A_1 f^+}{A_2^2}+ \frac{A_1 f^+}{B_2^2} \right)^2 + \frac{(g^+)^2}{2} \left( f^+-\frac{1}{A_1} + \frac{A_2}{B_1 B_2} \right)^2,
\end{align*}
where in the second case we have used the $G_2$-instanton equations \eqref{eq:InstantonODE1}-\eqref{eq:InstantonODE2}. 

\noindent It follows from the work in \cite{Bogo2013} that, up to rescaling, 
\begin{equation}\label{eq:Bogo.asym}
\lim_{t\to\infty}A_1=1,\quad \lim_{t\to\infty}\frac{A_2}{t}=c, \quad \lim_{t\to\infty}\frac{B_1}{t}=\frac{2}{\sqrt{3}}c,\quad\lim_{t\to\infty}\frac{B_2}{t}=c
\end{equation}
for a constant $c>0$.  Hence, from the third term in $|F_a|^2$, we see $t^{-2}f^+$ is bounded as $t\to\infty$.  Thus, from the first term, $g^+$ must be bounded. We also quickly see that if $f^+,g^+$ are both bounded then from the formulae for $|F_a|^2$ and $|\dot{a}|^2$ we see that $|F_A|^2$ is bounded as well.
\end{proof}

\noindent We have a 1-parameter family of reducible invariant $G_2$-instantons on the BGGG $G_2$-manifold which have gauge group $U(1)\subseteq SU(2)$: they are given in Corollary \ref{cor:U1}(b) with $x_2=x_3=0$ and have bounded (in fact, decaying) curvature.  
We start with our non-existence result, which shows that a 2-parameter family of initial conditions leads to local $G_2$-instantons which either do not extend with bounded curvature or can only extend as one of the above abelian instantons.

\begin{theorem}\label{thm:BGGG.NoInst}
Let $A$ be a $SU(2)^2\times U(1)$-invariant $G_2$-instanton with gauge group $SU(2)$ defined in a neighbourhood of $\{0\}\times S^3$ on the BGGG $G_2$-manifold $\mathbb{R}^4\times S^3$  smoothly extending over $P_1$ as given by Proposition \ref{prop:LocalG2Instantons}.  

\noindent If $f_1^+\leq\frac{1}{2}$, or $g_1^+\geq 0$ with $g_1^+\geq f_1^+$, then $A$ extends globally to $\mathbb{R}^4\times S^3$ with bounded curvature if and only if $A$ has gauge group $U(1)$ and is given in Corollary \ref{cor:U1}(b) with $x_2=x_3=0$. 
\end{theorem}

\begin{proof} If $g^+\equiv 0$ then we obtain an abelian instanton as in Corollary \ref{cor:U1}(b) with $x_2=x_3=0$.  Suppose, for a contradiction, that $g^+$ is not identically zero and that $A$ is defined for all $t$ has bounded curvature.  By Lemma \ref{lem:g+sign} and Remark \ref{g+sign.rmk} we may assume without loss of generality that $g_1^+>0$ and thus $g^+>0$ for all $t$.

\noindent Let $F=f^+/A_1$ and $G=g^+/A_1$ and let $s=r-\frac{9}{4} \in [0,\infty)$ be given as in \eqref{eq:r.BGGG}.  If we let $f'=\frac{df}{ds}$ then \eqref{eq:InstantonODE1}--\eqref{eq:InstantonODE2} are equivalent to
\begin{equation}\label{eq:NoInst.1}
F'=-G^2\quad\text{and}\quad G'=(H-F)G,
\end{equation} 
where
\begin{equation}\label{eq:H}
H=\frac{1}{2}\left(\frac{2}{A_1^2}+\frac{1}{B_2^2}-\frac{A_2^2+B_1^2+B_2^2}{A_1A_2B_1B_2}\right)=1-\frac{5(r-\frac{9}{20})^2-\frac{27}{10}}{r(r-3/4)(r+9/4)}.
\end{equation}
Notice that $H$ takes values in $(0,1)$, is increasing, and $\lim_{s\to\infty}H(s)=1$.

\noindent Suppose first that $f_1^+\leq \frac{1}{2}$.  Since $f^+=f_1^+t+O(t^3)$ and $A_1=t/2+O(t^3)$ by Example \ref{ex:BGGG}, we see that
$F(0)=2f_1^+\leq 1$.  Moreover, $F$ is strictly decreasing by \eqref{eq:NoInst.1} as $G>0$, so there is $\epsilon>0$ so that $F(s)\leq 1-\epsilon$ for all $s>0$.  As $H(s)\to 1$, there exists $s_0>0$ so that $H(s)-F(s)>\frac{\epsilon}{2}$ for all $s\geq s_0$.  
We deduce from \eqref{eq:NoInst.1} that $G'>\frac{\epsilon}{2}G$ for all 
$s\geq s_0$ as $G>0$, and hence $G\geq e^{\epsilon s/2}$.  Therefore, $g^+$ grows at least exponentially, so $F_A$ is unbounded by Lemma \ref{lem:CurvatureBound}, giving a contradiction.

\noindent Now suppose $g_1^+\geq f_1^+$.  Then $G(0)-F(0)\geq 0$ and 
 one sees that $G(s)-F(s)>0$ is increasing for small $s>0$ using 
 $g^+=g_1^+t+u_2(0)t^3+O(t^5)$, $f^+=f_1^+t+u_1(0)t^3+O(t^5)$ and the formulae \eqref{eq:u1}--\eqref{eq:u2}, where the values $C_1(0),C_2(0)$ for the BGGG metric are given in Example \ref{ex:BGGG}.  
 
\noindent We see from \eqref{eq:NoInst.1} that
\begin{equation}\label{eq:NoInst.2}
(G-F)'=(H+G-F)G.
\end{equation}
Therefore, when $G-F=0$ we must have $(G-F)'=HG>0$.  As $G-F$ is initially increasing, it therefore cannot have any zeros for $s>0$, which means that $G-F>0$ for all $s>0$.  We deduce that $-F>-G$ and hence, by \eqref{eq:NoInst.1}, $G'>(H-G)G$.

\noindent If $F(s)\leq 1$ for some $s$, then we are in the same situation as the previous case of $f_1^+\leq\frac{1}{2}$, which leads to a contradiction.  If instead $F(s)>1$ for all $s$ then $F$ is bounded below, so as $F$ is strictly decreasing we need from \eqref{eq:NoInst.1} that $\lim_{s\to\infty}G(s)=0$.  Hence, as $H(s)\to 1$, there exists $s_0$ so that $G'>\frac{1}{2}G$ for all $s\geq s_0$, which implies that $g^+$ grows at least exponentially.  This again gives a contradiction by Lemma \ref{lem:CurvatureBound}.
\end{proof}

\begin{remark}
The above proof of non-existence of irreducible instantons for $f_1^+\leq \frac{1}{2}$ immediately extends to the Bogoyavlenskaya metrics by the asymptotics in \eqref{eq:Bogo.asym}.  The proof for $g_1^+\geq 0$ and $g_1^+\geq f_1^+$ would also extend if we knew that $H$ given in \eqref{eq:H} continued to be positive for all $t>0$ for the Bogoyavlenskaya metrics.
\end{remark}

\noindent We now give our existence result, which provides a full 2-parameter family of irreducible $SU(2)^2\times U(1)$-invariant $G_2$-instantons with gauge group $SU(2)$ on the BGGG $G_2$-manifold.

\begin{theorem}\label{thm:BGGG.Inst}
Let $A$ be a $SU(2)^2\times U(1)$-invariant $G_2$-instanton with gauge group $SU(2)$ defined in a neighbourhood of $\{0\}\times S^3$ on the BGGG $G_2$-manifold $\mathbb{R}^4\times S^3$  smoothly extending over $P_1$ as given by Proposition \ref{prop:LocalG2Instantons}.\\  If $f_1^+\geq \frac{1}{2}+g_1^+>\frac{1}{2}$, then $A$ extends globally to $\mathbb{R}^4\times S^3$ with bounded curvature. 
\end{theorem}

\begin{proof}
Recall the notation from the proof of Theorem \ref{thm:BGGG.NoInst}.  The conditions in the statement are equivalent to $F(0)-1\geq G(0)>0$.  Hence, by Lemma \ref{lem:g+sign}, we have that $G(s)>0$ for all $s$.  Now observe from \eqref{eq:NoInst.1} that since the function $H$ in \eqref{eq:H} takes values in $(0,1)$ we have
\begin{equation}\label{eq:Inst.1}
\frac{d}{ds}\big((F-1)^2-G^2\big)=2(1-H)G^2>0.
\end{equation}
As $(F-1)^2-G^2\geq 0$ at $s=0$, we have that $(F-1)^2-G^2>0$ for all $s>0$.  Thus $(F-1)^2>G^2>0$ and since $F(0)>1$ this means that $F(s)>1$ for all $s$.  

\noindent By \eqref{eq:NoInst.1}, $F$ is decreasing and thus $F$ is bounded as it is bounded below (by 1).  We also know that $0<G<F-1$ so $G$ is also bounded.  As $H$ is also bounded, we deduce that a long time solution to the ODEs \eqref{eq:NoInst.1} must exist.  

\noindent Since $F$ is bounded below by $1$, decreasing and exists for all $s$ we must have again from \eqref{eq:NoInst.1} that $G(s)\to 0$ as $s\to\infty$, and that $\lim_{s\to\infty}F(s)$ exists and equals some constant greater or equal to $1$.  Hence, both $f^+$ and $g^+$ are bounded  and so $A$ has bounded curvature by Lemma \ref{lem:CurvatureBound}.
\end{proof}

\begin{remark}\label{rem:Inst_Bogoyavlenskaya}
Via the asymptotics in \eqref{eq:Bogo.asym}, we see that the function $H$ in \eqref{eq:H} for any given Bogoyavlenskaya metric is always bounded above by some $C\geq 1$ (possibly depending on the metric, though one might hope to show that $C=1$).  Thus, the proof of Theorem \ref{thm:BGGG.Inst} extends to prove the existence of $G_2$-instantons with bounded curvature for $f_1^+\geq \frac{C}{2}+g_1^+>0$ in these cases.  
\end{remark}

\noindent  Given a $G_2$-instanton $A$ on the BGGG $\mathbb{R}^4\times S^3$ as in Theorem \ref{thm:BGGG.Inst} we can evaluate the holonomy of $A$ around the circle at infinity, which is a $U(1)$ transformation.  In particular, if we fix $g_1^+>0$, we obtain a map
\begin{equation}\label{eq:hol.infty}
Hol_{\infty} : ( \tfrac{1}{2} +  g_1^+ , + \infty ) \rightarrow U(1) \subset SU(2)
\end{equation}
which is the map that takes $f_1^+$ to this limit holonomy.  It is natural to ask about the image of this map, which we now show is all of $U(1)$.

\begin{corollary}\label{cor:Holonomy}
For any fixed $g_1^+ >0$, the map \eqref{eq:hol.infty} is surjective.
\end{corollary}
\begin{proof}
From \eqref{eq:Inst.1} we conclude that for all $s>0$
$$(F(s)-1)^2 > (F(s)-1)^2-G(s)^2 > (F(0)-1)^2-G(0)^2 = (2f_1^+ -1)^2 -(2 g_1^+)^2 .$$
Since $F(s)>1$ for all $s$ by the proof of Theorem \ref{thm:BGGG.Inst}, we deduce in fact that  
$$F(s) > 1 + \sqrt{(2f_1^+-1)^2 -(2 g_1^+)^2 }$$
for all $s>0$.  Moreover, as $F$ is decreasing by\eqref{eq:NoInst.1}, we have that
\begin{equation}\label{eq:F_infinity}
F_{\infty} (f_1^+) := \lim_{ t \rightarrow + \infty} F(t) \in \left[ 1 + \sqrt{(2f_1^+-1)^2 -(2 g_1^+)^2 } \ , \ 2f_1^+ \right] .
\end{equation}
Hence, for any fixed $g_1^+$, we can vary $f_1^+ > \frac{1}{2}+g_1^+$ to ensure that $F_{\infty}$ is as large as we want. By continuous dependence with respect to initial conditions for ODEs, we have that $F_{\infty}$ varies continuously with $f_1^+$, and so the image of the map $F_{\infty} : ( \frac{1}{2} + g_1^+ , + \infty ) \rightarrow \mathbb{R}$ contains at least the interval $ ( 1+2g_1^+ , + \infty )$.\\
Now let $\gamma_t$ be the circle parametrized by
$$\gamma_t(\theta) = (t , \exp_{(1,1)}(2 \pi \theta T_1^+) )\subseteq \mathbb{R}^+_t \times S^3 \times S^3 ,$$
for $\theta \in [0,1]$. Then, the holonomy of $A=a(t)$ around $\gamma_t$ is 
\begin{eqnarray}\nonumber
Hol(\gamma_t) & = & \exp\left(\int_{\gamma_t} a(t) \right) = \exp\left(\int_{\gamma_t} A_1(t)^2 F(t) T_1\otimes \eta_1^+ \right) \\ \nonumber 
& = & \exp( 2\pi A_1(t)^2 F(t) T_1).
\end{eqnarray}
Taking the limit as $t \rightarrow + \infty$ and recalling \eqref{eq:Bogo.asym}  gives $Hol_{\infty}= \exp(2 \pi F_{\infty}T_1)$.  The surjectivity of $F_{\infty}(f_1^+)$ onto $ ( 1+2g_1^+ , + \infty )$ proves the desired result.
\end{proof}

\begin{remark}\label{rem:Asymptotics_BGGG} 
The proofs of Theorem \ref{thm:BGGG.Inst} and Corollary \ref{cor:Holonomy} show that for the $G_2$-instantons $A$ constructed we have $F\to F_{\infty} \geq 1$ and $G\to 0$ at infinity.  Moreover, if $F_{\infty}>1$ (which occurs if $f_1^+>\frac{1}{2}+g_1^+$) then \eqref{eq:NoInst.1} implies that $G$ tends to 0 at an exponential rate.  Observe that the abelian $G_2$-instantons of \ref{cor:U1}(b) with $x_2=x_3=0$ are given by $F=x_1\in\mathbb{R}$ and $G=0$.  Hence, $A$ is asymptotic to an abelian $G_2$-instanton, with exponential rate of convergence if $F_{\infty}>1$. Moreover, using Lemma \ref{lem:Curvature} and \eqref{eq:Bogo.asym} we may compute the pointwise norm of the curvature $F_A$ of $A$ satisfies
$$\vert F_A \vert \sim 2\sqrt{\frac{A_1^4}{A_2^4}+\frac{A_1^4}{B_2^4} } =O(t^{-2}),$$ 
which proves they have quadratically decaying curvature.\\
By contrast, in Proposition \ref{prop:asym}, we showed that the irreducible $SU(2)^2\times U(1)$-invariant $G_2$-instantons for the BS metric are asymptotic to an irreducible connection and the rate of convergence is $O(t^{-3})$.
\end{remark}

\noindent In summary, on the BGGG $G_2$-manifold $\mathbb{R}^4\times S^3$, we have shown non-existence for irreducible $SU(2)^2\times U(1)$-invariant $G_2$-instantons with gauge group $SU(2)$ and bounded curvature for  $g_1^+>0$ and $f_1^+\leq\frac{1}{2}$ or $g_1^+\geq f_1^+$, and existence for $f_1^+\geq\frac{1}{2}+g_1^+>\frac{1}{2}$.  This currently leaves open the region where $0<f_1^+-\frac{1}{2}<g_1^+<f_1^+$.  Some numerical investigation indicates that some of these initial conditions may lead to globally defined instantons with bounded curvature and some may not.

\subsubsection[\texorpdfstring{Solutions smoothly extending on $P_{\id}$}{Solutions smoothly extending on Pid}]{Solutions smoothly extending on {\boldmath $P_{\id}$}}

We now turn our attention to the more difficult case of solutions to the $SU(2)^2\times U(1)$-invariant $G_2$-instanton equations on 
$\mathbb{R}^4\times S^3$ which smoothly extend on the bundle $P_{\id}$.  
Here the ODE system does not simplify, but we obtain a 1-parameter family of local solutions in a neighbourhood of the singular orbit.  Although the strategy of proof remains the same as in our earlier similar results, the analysis is more involved. In order to ease  computations, we use the Taylor expansion for a smooth $SU(2)^2 \times U(1)$-symmetric $G_2$-holonomy metric in a neighbourhood of a singular orbit $\lbrace 0 \rbrace \times S^3$ at $t=0$,  computed in \eqref{eq:Metric_Taylor1}--\eqref{eq:Metric_Taylor4}, which depends on constants $b,c$.

\begin{proposition}\label{prop:LocalG2Instantons2}
Let $X \subset \mathbb{R}^4\times S^3 $ contain the singular orbit $\{0\}\times S^3 $ of the $SU(2)^2 \times U(1)$ action and be equipped with an $SU(2)^2 \times U(1)$-invariant holonomy $G_2$-metric.  There is a $1$-parameter family of $SU(2)^2 \times U(1)$-invariant $G_2$-instantons $A$ with gauge group $SU(2)$ in a neighbourhood of the singular orbit in $X$ smoothly extending over $P_{\id}$.\\
Moreover, in the notation of Proposition \ref{prop:ODEsSU(2)xU(2)} and \eqref{eq:metric.expansion}, any such $G_2$-instanton $A$ can be written as in \eqref{eq:InvariantConnection.BGGG}  
with $f^{\pm}$, $g^{\pm}$ solving the ODEs \eqref{eq:dotf+}-\eqref{eq:dotg-}
 subject to
\begin{eqnarray} \nonumber
f^{+}(t) & = &  \frac{2}{t} +  \left( \frac{(b_0^-)^2}{4}-\frac{1}{4b^2} -4c \right) t \\ \nonumber
&  + & \!\! \left( \frac{35  \left( b^2(b_0^-)^2- \frac{16}{7} \right) b^2 (b_0^-)^2 + 112  (b^2 c +12) b^2 c +22}{480 b^4} \right) t^3 + O(t^5 ), \\ \nonumber
g^+(t) & = & \frac{2}{t} +  \left( \frac{(b_0^-)^2}{4} + 2c \right) t \\ \nonumber
&  + &  \!\! \left( \frac{35  \left( b^2(b_0^-)^2- \frac{16}{7} \right) b^2 (b_0^-)^2 + 112  (b^2 c +12) b^2 c +22}{480 b^4} \right) t^3 + O(t^5 ), \\ \nonumber
\nonumber
f^-(t) & = & b_0^- + \frac{b_0^-}{4b^2} (b^2 (b_0^-)^2 -2) t^2 + O(t^4), \\ \nonumber
g^{-}(t) & = &  b_0^- + \frac{b_0^-}{4b^2} (b^2 (b_0^-)^2 -2) t^2 + O(t^4), 
\end{eqnarray}
for $b_0^- \in \mathbb{R}$.
\end{proposition}
\begin{proof}
On $P_{\id}$, the singular initial value problem to be solved has 
\begin{gather*}
f^-(t)=b_0^- + t^2 v_1(t), \quad g^{-}(t)=b_0^- + t^2 v_2(t),\\
f^{+}= \frac{2}{t} + (b_2^+-4C_1(0) ) t + t^3 u_1(t), \quad g^+= \frac{2}{t} + (b_2^+-4C_2(0)) t + t^3 u_2(t),
\end{gather*}
for some real analytic $v_1(t),v_2(t),u_1(t),u_2(t)$ by Lemma \ref{lem:SmoothlyExtendBGGG}. Moreover, notice that from \eqref{eq:Metric_Taylor1}--\eqref{eq:Metric_Taylor4} we have $C_1(0)=c$ and $C_2(0)=-\frac{1+8 cb^2}{16 b^2}$ and in the following we will write the coefficients of the metric in terms of $b,c \in \mathbb{R}$. The ODEs in Proposition \ref{prop:ODEsSU(2)xU(2)} then turn into the following ones for $X(t)=(u_1(t),u_2(t),v_1(t),v_2(t))$:
$$\frac{dX}{dt} = \frac{M_{-3}(b_0^- , b_2^+)}{t^3}+ \frac{M_{-1}(X(t))}{t} + f(t,X(t)),$$
where $f(t,X(t))$ is real analytic in both entries and
$$M_{-3}(b_0^- , b_2^+) = \left( (b_0^-)^2-4b_2^+- \frac{1}{b^2}, (b_0^-)^2-4b_2^+- \frac{1}{b^2}, 0 , 0 \right) .$$
For this to have a real analytic solution $X(t)$ we must have $M_{-3}=0$ which requires that $4b_2^+=(b_0^-)^2-\frac{1}{b^2}$. In that case we have
\begin{align*}
 M_{-1}(X(0)) & =   \!\Big(\! -6u_1(0)+ 2b_0^- v_2(0) , -3 u_1(0)-3u_2(0)+ b_0^- v_1(0) + b_0^- v_2(0), \\ \qquad  & -6v_1(0)+4v_2(0),
 2v_1(0)-4v_2(0) \Big) + K(b_0^-),
\end{align*}
where $K(b_0^-)\in \mathbb{R}^4$ is a constant only depending on $b_0^-$ and the metric. For a real analytic solution to exist we need $M_{-1}(X(0))=0$. As this is a linear equation and $dM_{-1}(X(0))$ is always an isomorphism, it can be uniquely solved for any $K(b_0^-)$. The unique solution of $M_{-1}(X(0))=0$ can be written as
\begin{equation}\label{eq:Pid.BGGG.inits}
\begin{gathered}
u_1(0)=u_2(0)= \frac{35  \left( b^2(b_0^-)^2- \frac{16}{7} \right) b^2 (b_0^-)^2 + 112  (b^2 c +12) b^2 c +22}{480 b^4},\\
v_1(0) = v_2(0) = \frac{b_0^-}{4b^2} (b^2 (b_0^-)^2 -2) .
\end{gathered}
\end{equation}
We now use the existence and uniqueness theorem for initial value problems of  
\cite{Malgrange1974}. This guarantees that for each $b_0^- \in \mathbb{R}$ there is a unique solution to the system 
$$\frac{dX}{dt} = \frac{M_{-1}(X(t))}{t} + f(t,X(t)),$$
provided that $M_{-1}(X(0))=0$ and $d M_{-1}(X)$ has no eigenvalues in the positive integers. We showed above that we can always find a unique $X(0)$ such that $M_{-1}(X(0))=0$. Moreover, the eigenvalues of $dM_{-1}$ can be computed to be $-8,-6,-3,-2$. 
Hence, for each $b_0^-\in \mathbb{R}$ there is indeed a unique solution $X(t)$ to the system above. This yields a unique $G_2$-instanton as in the statement determined by $b_0^-$.
\end{proof}

\begin{remark}
Since the BS, BGGG and Bogoyavlenskaya 
$G_2$-metrics all have $SU(2)^2 \times U(1)$-symmetry, Proposition \ref{prop:LocalG2Instantons2} yields $G_2$-instantons in these cases. In particular, in the BS case we have $c=-\frac{1}{24 b^2}$, $b=\frac{1}{\sqrt{3}}$ and these $G_2$-instantons coincide with those given in Proposition \ref{prop:LocalExistenceBS_1}.  
\end{remark}

\noindent In light of the existence result in Theorem \ref{thm:Alim} and the local existence result in Proposition \ref{prop:LocalG2Instantons2}, it is certainly an interesting non-trivial question which members of the 1-parameter family of local $G_2$-instantons from Proposition \ref{prop:LocalG2Instantons2} extend on $P_{\id}$ on the BS, BGGG or Bogoyavlenskaya $\mathbb{R}^4\times S^3$.\\
Another natural problem for further study is to understand the limits of the family of instantons constructed in Theorem \ref{thm:BGGG.Inst}, and their possible relationship to any extensions of the local instantons given in 
Proposition \ref{prop:LocalG2Instantons2}.  We saw in Proposition \ref{prop:asym} that global $G_2$-instantons on the BS $\mathbb{R}^4\times S^3$ have a limit at infinity given by a canonical connection 
on the link $S^3\times S^3$ of the asymptotic cone. For the instantons constructed in Theorem \ref{thm:BGGG.Inst} we know, by Remark \ref{rem:Asymptotics_BGGG}, that these are asymptotic to the abelian $G_2$-instantons with 
a rate depending on the asymptotic connection.  It is also certainly an interesting problem to investigate the behaviour of the family of instantons from Theorem \ref{thm:BGGG.Inst} when one or both of $f_1^+$ and $g_1^+$ go to 
infinity.   We would expect bubbling phenomena as in the BS case in Theorem \ref{thm:Compactness}, with possible relationship to the ASD instantons on Taub--NUT found in \cite{EtesiHausel}.  
The lack of an explicit formula for our instantons makes the bubbling analysis more difficult.\\
One other interesting problem is to investigate the behaviour of   $G_2$-instantons as the underlying metric is deformed. For instance, Remark \ref{rem:Inst_Bogoyavlenskaya} shows how to adapt the proof of existence in 
Theorem \ref{thm:BGGG.Inst} to the Bogoyavlenskaya $G_2$-manifolds, and we would want to analyse these instantons as the size of the circle at infinity gets very large or small. 
When it gets very large we expect them to resemble $G_2$-instantons for the BS metric given in Theorem \ref{thm:BS.cor}.  When it gets very small, there may be a relation with Calabi--Yau monopoles on the deformed conifold 
(as in \cite{Oli16}).

\appendix

\section{\texorpdfstring{\boldmath $SU(2)^2$}{SU(2)xSU(2)}-invariant tensors}\label{app:R4xS3}

In this appendix, we use Eschenburg--Wang's technique \cite{Eschenburg2000} to determine when a metric or connection extends smoothly over a singular orbit $Q=SU(2)\times SU(2)/ \Delta SU(2)\cong S^3$ in $X=\mathbb{R}^4 \times S^3$. The relevant group diagram 
is $I(SU(2)\times SU(2); \lbrace 1 \rbrace ; SU(2))$ and so the principal orbits are topologically $S^3 \times S^3$.  We will often identify $SU(2)$ with the unit quaternions.
\\
The normal bundle $NQ$ to $Q$ 
is $\mathbb{R}^4 \times S^3$ and is homogeneously constructed by $NQ= (SU(2) \times SU(2) )\times_{ SU(2)} \mathbb{H}$, where $SU(2)$ acts on $SU(2)\times SU(2)$ diagonally and on $\mathbb{H}$ by left multiplication. 
Similarly, $TQ=(SU(2)\times SU(2) ) \times_{SU(2)} \im (\mathbb{H})$, where $q\in SU(2)$ acts on $x \in \im(\mathbb{H})$ by $q \cdot x = qx\overline{q}$. We also note that 
$$T(NQ) \cong NQ\oplus\pi^* TQ,$$
where $\pi: NQ \rightarrow Q$ is the projection.

\subsection{Metrics}
 
By the previous discussion, $T(NQ)$ is modelled on $W=\mathbb{H}\oplus\im (\mathbb{H})$, with $a\in SU(2)$ acting by $a \cdot (p,q)= (ap,a q \overline{a})$, for $(p,q) \in W$. Following \cite{Eschenburg2000}, to determine which metrics extend smoothly over 
$Q$ we seek a basis of $S^2(W)$ corresponding to the evaluation at $1 \in \mathbb{H}$ of homogeneous $SU(2)$-equivariant polynomials $\mathbb{H} \rightarrow S^2(W)$ of minimal degree.\\
The equivariance condition implies that any such polynomial is of the form $x \mapsto \phi(x) \in S^2(W) \subset W \otimes W \cong \End(W)$, such that for $(p,q) \in W = \mathbb{H}\oplus \im(\mathbb{H})$  and $x\in SU(2)$ we have
\begin{align}
\phi(x) (p,q) &=\big(\phi_1(x)(p,q),\phi_2(x)(p,q)\big)\nonumber\\
&=\big(x \phi_1 (1) ( \overline{x} p ,\overline{x} q x ),  x \big(\phi_2(1) ( \overline{x} p,\overline{x} q x ) \big)\overline{x}  \big).\label{phi.eq}
\end{align}

\begin{itemize}
\item First we look at maps 
$x \mapsto \psi(x)(\cdot) \in \End(\im \mathbb{H})$  such that $ \psi(x)(q)=  x ( \psi (1) ( \overline{x} q x ) ) \overline{x}$ for $x\in SU(2)$.  The identity map is constant and so homogeneous of degree $0$. We also have the homogeneous degree $4$ polynomials $\psi(x)(q)=-xl\overline{x}qxl\overline{x}$ for $l\in\{i,j,k\}$. 
Given the coordinates $q=q_1i+q_2j+q_3 k \in \im( \mathbb{H})$ we have a canonical identification $\End(\im(\mathbb{H}) \cong \im(\mathbb{H})^* \otimes \im(\mathbb{H})^*$.  Using this identification, we have that the identity and degree 4 polynomials given, when  
evaluated at $x=1$, correspond to 
\begin{equation*}
\begin{aligned}
dq_1 \otimes dq_1 + dq_2 \otimes dq_2 + dq_3 \otimes dq_3,&\\
dq_1 \otimes dq_1 - dq_2 \otimes dq_2 - dq_3 \otimes dq_3, & \\ -dq_1 \otimes dq_1 + dq_2 \otimes dq_2 - dq_3 \otimes dq_3, &\\ -dq_1 \otimes dq_1 - dq_2 \otimes dq_2 + dq_3 \otimes dq_3.&
\end{aligned} 
\end{equation*}
\item Now we consider 
maps $x \mapsto \psi(x)(\cdot) \in \End( \mathbb{H})$ such that $ \psi(x)(p)=  x \psi (1)( \overline{x} p )$ for $x\in SU(2)$. Fixing coordinates $p=p_0 + i p_1 + j p_2 + kp_3 \in \mathbb{H}$, we may identify $\End(\mathbb{H})$ with $\mathbb{H}^* \otimes \mathbb{H}^*$. 
Certainly, the constant maps given by the identity and the complex structures are $\SU(2)$-equivariant. The constant map corresponds to 
\begin{equation*}
dp_0 \otimes dp_0 + dp_1 \otimes dp_1 + dp_2 \otimes dp_2 + dp_3 \otimes dp_3,
\end{equation*} while the complex structures correspond to antisymmetric (anti-self-dual) $2$-tensors. 
We also have homogeneous degree 2 polynomials, where $\psi(x)(p)=\langle p,xl\rangle xl$ for $l\in\{i,j,k\}$. These are 
 $SU(2)$-equivariant and correspond under evaluation at $x=1$ to
\begin{equation*}
 dp_1 \otimes dp_1, \ dp_2 \otimes dp_2, \ dp_3 \otimes dp_3.
 \end{equation*}
\item Finally, it suffices to consider maps $x\mapsto \phi(x)$ as in \eqref{phi.eq} with $\phi_1(x)(p,q)=\phi_1(x)(q)$ and $\phi_2(x)(p,q)=\phi_2(x)(p)$.  We have the $SU(2)$-equivariant linear polynomial $\phi(x)(p,q)=(qx,\frac{1}{2}(p\bar{x}-x\bar{p}))$, which in the coordinates as above corresponds at $x=1$ to
$$\sum_{i=1}^3 dq_i\otimes dp_i+dp_i\otimes dq_i.$$
The equivariant homogeneous degree 3 polynomials $$\phi(x)(p,q)=(\langle q,xl_1\overline{x}\rangle xl_2,\langle p,xl_2\rangle xl_1\overline{x}),$$ for $l_1,l_2\in\{i,j,k\}$, then correspond under evaluation at $x=1$ to $dp_i\otimes dq_j+dq_i\otimes dp_j$ for $i,j\in\{1,2,3\}$.
\end{itemize}

\begin{remark}
As an alternative to the degree 4 polynomials we wrote down in the first bullet above, we could have used 
$\psi(x)(q)=\langle q, xl_1\overline{x} \rangle xl_2\overline{x}$, where $l_1,l_2 \in \lbrace i,j,k \rbrace$.
\end{remark}

\noindent We now have enough information to analyze metrics of the form
\begin{equation}\label{eq:MetricToExtend}
g= dt^2 + \sum_{i=1}^3 (2A_i(t))^2 \eta_{i}^+ \otimes \eta_i^+ + (2B_i(t))^2 \eta_i^- \otimes \eta_i^-,
\end{equation}
where $\eta_i^{\pm}$ define bases for the diagonal and anti-diagonal copies of $\mathfrak{su}(2)$ in $\mathfrak{su}(2)\oplus\mathfrak{su}(2)$ as in $\S$\ref{sec:G2metrics}.  
We embed $\mathbb{R}^4 \times S^3 \hookrightarrow \mathbb{H} \times \mathbb{H}$ and let $SU(2) \times SU(2)$ act via $(a_1,a_2) \cdot (p,q) = (a_1p,a_1 q \overline{a}_2)$. Using this action and the coordinates $p=p_0+ip_1+jp_2+kp_3$ and $q=q_0+iq_1+jq_2+kq_3$, we compute that, at $(t,1)\in\mathbb{R}^4\times S^3$ for $t\in\mathbb{R}$, the dual frames $\{T_i^{\pm}\}$ to the coframes $\{\eta_i^{\pm}\}$ satisfy
$$T_i^+ = t \frac{\partial}{\partial p_i}, \quad T_i^-=t \frac{\partial }{\partial p_i} + 2 \frac{\partial}{\partial q_i}, \quad \frac{\partial }{\partial t}  = \frac{\partial}{\partial p_0},$$
for $i=1,2,3$. At $t=0$ the isotropy is $\Delta SU(2)$ 
and the orbit $Q$ is an $S^3$ whose tangent space at $(0,1)$ is $0\oplus\im (\mathbb{H})$. 
For $t\neq 0$ we have 
\begin{eqnarray}\label{eq:app.coframe}
\eta_i^+=\frac{1}{t}dp_i - \frac{1}{2}dq_i , & \eta_i^- = \frac{1}{2}dq_i, & dt = dp_0.
\end{eqnarray}
 It is now easy 
 to rewrite the metric in \eqref{eq:MetricToExtend} in terms of the equivariant symmetric $2$-tensors we found above. This gives
\begin{align*}
&g  =  dp_0^2 + \sum_{i=1}^3 \left( \frac{2A_i}{t} \right)^2 dp_i \otimes dp_i - \sum_{i=1}^3 \frac{2A_i^2}{t} ( dp_i \otimes dq_i + dq_i \otimes dp_i ) \\
&\qquad   + \sum_{i=1}^3 ( A_i^2 + B_i^2 ) dq_i \otimes dq_i \displaybreak[0]\\ \nonumber
& =  \sum_{i=1}^4 dp_i^2 + C \sum_{i=1}^3 dq_i^2  + \sum_{i=1}^3 \!\bigg( \!\left( \frac{2A_i}{t} \right)^2 -1 \bigg) dp_i^2 + \sum_{i=1}^3 \!\left( A_i^2 + B_i^2- C \right) dq_i^2\\
&  + D\sum_{i=1}^3   ( dp_i \otimes dq_i + dq_i \otimes dp_i )
-\sum_{i=1}^3\left(\frac{2A_i^2}{t}+D\right)( dp_i \otimes dq_i + dq_i \otimes dp_i )
\end{align*}
where $C$ is some smooth even function of $t$ and $D$ is some smooth odd function of $t$. 
 Eschenburg--Wang's technique guarantees that $g$ 
  smoothly extends over $Q$ if and only if, for $i=1,2,3$, $\left( 2A_i/t \right)^2 -1$ is even and $O(t^2)$, $B_i^2+A_i^2-C$ is even and $O(t^4)$, and $\frac{2A_i^2}{t}+D$ is odd and $O(t^3)$. In other words, $A_i(t) = t/2 + O(t^3)$ and $B_i^2(t)=C(t) - t^2/4 + O(t^4)$; in particular notice that up to order $O(t^2)$ the $A_i$ and $B_i$ do not depend on $i=1,2,3$. Moreover, for $g$ to extend to a metric we also require it to be positive definite. This implies that $A_i$, $B_i$ are sign definite for $t>0$ and $A_i(0)=0$, while $B_i(0) \neq 0$. 
 We summarise these conclusions. 

\begin{lemma}\label{lem:ExtendingSmoothlyMetric}
The metric $g$ 
in \eqref{eq:MetricToExtend} extends smoothly (as a metric) over the singular orbit $Q=SU(2)^2/ \Delta SU(2)$ if and only if   
$A_i$, $B_i$ are sign definite for $t>0$ and:
\begin{itemize}
\item  
the $A_i$'s are odd with $\dot{A}_i(0)=1/2$;
\item 
the $B_i$'s are even with $B_1(0)=B_2(0)=B_3(0) \neq 0$ and $\ddot{B}_1(0)=\ddot{B}_2(0)=\ddot{B}_3(0)$.
\end{itemize}
\end{lemma}

\begin{remark}
In fact, for our applications 
there is no restriction in having the metrics above being real analytic instead of smooth. As $G_2$ manifolds are Ricci-flat, the metric is 
real analytic in harmonic coordinates. The function $t$ can be interpreted as the arclength parameter along a geodesic intersecting the principal orbits orthogonally, so it is a real analytic function of the harmonic coordinates, and thus the metric coefficients must be real analytic functions of $t$.
\end{remark}

\noindent Using Lemma \ref{lem:ExtendingSmoothlyMetric} and equations \eqref{eq:dotA1}--\eqref{eq:dotB2} we can compute the first order terms in the Taylor expansion for a metric with holonomy $G_2$ in a neighbourhood of a singular orbit $Q$ at $t=0$ to be
\begin{align}\label{eq:Metric_Taylor1}
A_1(t) & =  \frac{t}{2} + c t^3 + {\frac {  96(22c{b}^{2}+1) cb^2+11}{640 b^{4}}} t^5 + \ldots \\ \label{eq:Metric_Taylor2}
A_2(t) & =  \frac{t}{2} - \frac{1+ 8 c b^2}{16 b^2} t^3 - {\frac { 11- 24( 32c{b}^{2}+1) cb^2}{640 b^{4}}} t^5 + \ldots \displaybreak[0]\\  \label{eq:Metric_Taylor3}
B_1(t) & =  b + \frac{1}{4 b} t^2 - \frac{7+8cb^2}{160 b^3} t^4 + \ldots \\ \label{eq:Metric_Taylor4}
B_2(t) & =  b + \frac{1}{4 b} t^2 - \frac{13-8cb^2}{320 b^3} t^4 + \ldots
\end{align}

\noindent We now confirm that the BS and BGGG 
metrics from $\S$\ref{sec:G2metrics} satisfy the conditions of Lemma \ref{lem:ExtendingSmoothlyMetric}.  We use these formulae on a number of occasions. 

\begin{example}\label{ex:BS}
The BS metric on $\mathbb{R}^4\times S^3$ from $\S$\ref{sss:BS} has $A_1=A_2=A_3$ and $B_1=B_2=B_3$ in \eqref{eq:MetricToExtend}, with expansions
\begin{equation}\nonumber
A_1(t)  =  \frac{t}{2}-\frac{1}{8} t^3 +O(t^5), \quad 
B_1(t)  =  \frac{1}{\sqrt{3}} + \frac{\sqrt{3}}{4} t^2 - \frac{\sqrt{3}}{8} t^4 +O(t^6).
\end{equation}
\end{example}

\begin{example}\label{ex:BGGG}
The  BGGG 
metric on $\mathbb{R}^4\times S^3$ from $\S$\ref{sss:BGGG} has $A_2=A_3$ and $B_2=B_3$ in \eqref{eq:MetricToExtend}, with expansions
\begin{eqnarray}\nonumber
A_1(t) = \frac{t}{2}-\frac{7}{108} t^3 +O(t^5), &  &  \ \ A_2(t) = \frac{t}{2}+\frac{1}{216} t^3 +O(t^5), \\ \nonumber
B_1(t) = \frac{3}{2} + \frac{1}{6} t^2 - \frac{7}{648} t^4 + O(t^6), &  &  \ \ B_2(t) = \frac{3}{2} + \frac{1}{6} t^2 - \frac{17}{1296} t^4 + O(t^6) .
\end{eqnarray}
\end{example}

\subsection{Lie algebra-valued \texorpdfstring{$1$}{1}-forms}

Let $G$ be a compact Lie group with Lie algebra $\mathfrak{g}$.  
We now analyze the conditions to extend $\mathfrak{g}$-valued $1$-forms of the form 
\begin{equation}\label{eq:b2}
b  =  
\sum_{i=1}^3 b_i^+ \otimes \eta_i^+ +  \sum_{i=1}^3 b_i^- \otimes \eta_i^-
\end{equation}
over the singular orbit $Q$.  This a priori depends on how the (trivial) bundle $P= (SU(2) \times SU(2) ) \times G$ extends over $Q$. 
Such extensions are parametrized by (conjugacy classes) of isotropy homomorphisms $ \mu: SU(2) \rightarrow G$. Given $\mu$, we pull $P_{\mu}=SU(2)^2 \times_{(SU(2),\mu)} \mathfrak{g}$ back to $\mathbb{R}^4 \times S^3$, which determines the extension.\\
 Then $SU(2)$ acts on $\mathfrak{g}$ via $\Ad \circ \mu$ and we need a 
 basis for $\Hom(W , \mathfrak{g})$ given by evaluation at $1$ of homogeneous $SU(2)$-equivariant polynomials $\mathbb{H} \rightarrow \Hom(W , \mathfrak{g})$, where $W = \mathbb{H}\oplus\im(\mathbb{H})$. Following \cite{Eschenburg2000}, we seek homogeneous polynomials $x \mapsto \phi(x)$ such that for $(p,q) \in W$ we have, for $x\in SU(2)$, 
$$\phi(x) (p,q)= \Ad \circ \mu \left( x \right) \phi(1) \left(\overline{x}p, \overline{x} q x   \right).$$

\subsubsection[\texorpdfstring{$G=U(1)$}{G=U(1)}]{{\boldmath $G=U(1)$}}

Here, $\mu:SU(2) \rightarrow U(1)$ must be trivial and $\mathfrak{g}=\mathbb{R}$.  
 We also have $\Hom(W, \mathbb{R}) \cong  \mathbb{H}^*\oplus \im(\mathbb{H})^* $ and we are left with analyzing when a $1$-form extends over  $Q$.  We describe $SU(2)$-equivariant homogeneous polynomials in $\mathbb{H}$ with values in $\im(\mathbb{H})^*$ and $\mathbb{H}^*$  independently.

\begin{itemize}
\item First we look for 
homogeneous polynomials $\mathbb{H} \rightarrow \im(\mathbb{H})^*$, given by $x\mapsto \psi(x)$ such that $\psi(x)(q)=\psi(1)(\overline{x}qx)$ for $x\in SU(2)$. These are generated by the degree $2$  polynomials 
$\psi(x)(q)=\langle qx, xl \rangle$, where $l \in \lbrace i,j,k \rbrace$. Under evaluation at $x=1$ these correspond to the $1$-forms $dq_i$ for $i=1,2,3$.

\item Next we look for   
homogeneous polynomials $\mathbb{H} \rightarrow \mathbb{H}^*$ given by $x \mapsto \psi(x)$ such that $\psi(x)(p)=\psi(1)(\overline{x} p )$ for $x\in SU(2)$. These 
are generated by the degree $1$ polynomials $\psi(x)(p)= \langle p, x l \rangle$, where $l \in \lbrace 1,i,j,k \rbrace$, which correspond under evaluation at $x=1$ to the $dp_i$'s for $i=0,1,2,3$. 
\end{itemize}

\noindent We now consider 
 extending the $1$-form
\begin{equation}\label{eq:b}
b  =  
\sum_{i=1}^3 b_i^+  \eta_i^+ +  \sum_{i=1}^3 b_i^-  \eta_i^- 
 =  
 \sum_{i=1}^3 \frac{b_i^+}{t}  dp_i +  \sum_{i=1}^3 \frac{b_i^- - b_i^+}{2} dq_i,
\end{equation}
where we used \eqref{eq:app.coframe}. Using the homogeneous polynomials in $\mathbb{H}$ computed above and Eschenburg--Wang's technique, we immediately deduce the following.

\begin{lemma}\label{lem:Extend1formOnBS}
The 1-form $b$ 
 as given in \eqref{eq:b} extends over the singular orbit $Q=SU(2)^2/\Delta SU(2)$ if and only if 
 the $b_{i}^{\pm}$'s are even and $b_i^{\pm}(0)=0$  for $i=1,2,3$.
\end{lemma}

\subsubsection[\texorpdfstring{$G=SU(2)$}{G=SU(2)}]{{\boldmath $G=SU(2)$}}

In this case, using our earlier notation,  $\mu:SU(2) \rightarrow SU(2)$ must be either the identity $\mu=\id$ (up to conjugacy), or the trivial homomorphism $\mu=1$. Then, $\mathfrak{g}=\mathfrak{su}(2) \cong \im(\mathbb{H})$ and $\Ad \circ \mu$ is either the adjoint action $\Ad$ or trivial, respectively. We shall denote the respective bundles by $P_{\id}=(SU(2) \times SU(2))\times_{(\Delta SU(2), \id)} SU(2)$ and $P_{1}=(SU(2) \times SU(2))\times_{(\Delta SU(2), 1)} SU(2)$. The main result of this section considers the problem of extending $\mathfrak{su}(2)$-valued $1$-forms as in \eqref{eq:b2}. 

\begin{lemma}\label{lem:ExtendConnectionOnBS}
Let $b$ be an $\mathfrak{su}(2)$-valued $1$-form as in \eqref{eq:b2}. 
Write $b_i^{\pm}= \sum_{j=1}^3 b_{ij}^{\pm} T_j$, where $\lbrace T_i \rbrace_{i=1}^3$ is a standard basis for $\mathfrak{su}(2)$. 
Then the $1$-form $b$ extends over the singular orbit $Q=SU(2)^2/\Delta SU(2)$ on the bundle $P_{\mu}$ if:
\begin{itemize}
\item $\mu=\id$ and for $i=1,2,3$, $b_{ii}^{\pm}$ are even and there are $c_0^{-}, c_2^{\pm} \in \mathbb{R}$ such that 
$$b_{ii}^+ = c_2^+ t^2 + O(t^4), \quad b^-_{ii}= c_0^- + c_2^- t^2  + O(t^4);$$
and for $i \neq j$, $b_{ij}^{\pm}=O(t^4)$ are even;

\item  $\mu=1$ and the $b_{ij}^{\pm}$'s are even with $b_{ij}^{\pm}(0)=0$.
\end{itemize}
\end{lemma}

\noindent The rest of this Appendix is concerned with the proof of Lemma \ref{lem:ExtendConnectionOnBS}.

\subsubsection*{Case {\boldmath $\mu=\id$}}

Here, we may write $\Ad(x) q= x q \overline{x}$ and $\Hom(W, \im(\mathbb{H})) \cong  ( \im(\mathbb{H})\otimes\mathbb{H}^* )\oplus ( \im(H)\otimes\im(\mathbb{H})^* ) $.
As before, we shall analyze $SU(2)$-equivariant  homogeneous polynomials in $\mathbb{H}$ with values in each of the components independently.

\begin{itemize}
\item We begin by looking for  
homogeneous polynomials $\mathbb{H} \rightarrow \im(\mathbb{H})\otimes\mathbb{H}^* $ given by $x \mapsto \psi(x)$ such that $\psi(x)(q)=x (\psi(1)(\overline{x} q x )) \overline{x}$ for $x\in SU(2)$. We have the constant polynomial corresponding to the identity, which is 
$$T_1\otimes dq_1  + T_2\otimes dq_2 + T_3\otimes dq_3 .$$ 
We also see that the degree $4$ polynomials $\psi(x)(q)=\langle q, xl_1\overline{x} \rangle xl_2\overline{x}$, where $l_1,l_2 \in \lbrace i,j,k \rbrace$, generate the space of  $T_j \otimes dq_i $ for $i,j=1,2,3$ when evaluated at $x=1$.

\item Next we look for homogeneous polynomials $\mathbb{H} \rightarrow \im(\mathbb{H})\otimes\mathbb{H}^* $ given by $x \mapsto \psi(x)$ such  that $\psi(x)(p)=x (\psi(1)(\overline{x} p )) \overline{x}$ for $x\in SU(2)$. The degree $1$ polynomials $\psi(x)(p)= pl \overline{x} + \langle p, x l \rangle $, where $l \in \lbrace 1,i,j,k \rbrace$, correspond under evaluation at $x=1$ to the maps
\begin{align*}\nonumber
T_1 \otimes dp_1  + T_2 \otimes dp_2  + T_3 \otimes dp_3 , &  & T_1 \otimes dp_0  - T_3 \otimes dp_2  + T_2 \otimes dp_3 , \\ \nonumber
T_2 \otimes dp_0 + T_3 \otimes dp_1  - T_1 \otimes dp_3, &  & T_3 \otimes dp_0  - T_2 \otimes dp_1  + T_1 \otimes dp_2.
\end{align*}
We also have $\SU(2)$-equivariant maps $\psi(x)(q)=\langle xl_1 ,  p \rangle xl_2\overline{x}$, for $l_1\in\{1,i,j,k\}$, $l_2\in\{i,j,k\}$ which are homogeneous of degree $3$. Taking $x=1$, these generate $T_j\otimes dp_i $, for $i=0,1,2,3$ and $j=1,2,3$.
\end{itemize}

\noindent Recall that our goal is to consider the problem of extending the $\mathfrak{su}(2)$-valued  $1$-form
\begin{equation}\nonumber
b  =  
\sum_{i=1}^3 b_i^+ \otimes \eta_i^+ +  \sum_{i=1}^3 b_i^- \otimes \eta_i^- 
 =  
 \sum_{i=1}^3 \frac{b_i^+}{t} \otimes dp_i +  \sum_{i=1}^3 \frac{b_i^--b_i^+}{2} \otimes dq_i,
\end{equation}
where we used \eqref{eq:app.coframe}. 
Since $b_i^{\pm}\in\mathfrak{su}(2)$, we can write $b_i^{\pm}= \sum_{j=1}^3 b_{ij}^{\pm} T_j$ and 
\begin{align}\nonumber
b  = &   \frac{b_{11}^+}{t} \sum_{i=1}^3 T_i \otimes dp_i  + \sum_{i=1}^3 \frac{b^+_{ii}- b_{11}^+}{t} T_i \otimes dp_i + \sum_{i \neq j} \frac{b^+_{ij}}{t} T_i \otimes dp_j \\ \nonumber
&   + \frac{b_{11}^-}{2} \sum_{i=1}^3 T_i \otimes dq_i + \sum_{i=1}^3 \frac{b_{ii}^- - b_{ii}^+ - b_{11}^-}{2} T_i \otimes dq_i + \sum_{i \neq j} \frac{b_{ij}^- - b_{ij}^+}{2} T_i \otimes dq_j.
\end{align}
Given the homogeneous polynomials in $\mathbb{H}$ computed above, we conclude that $b$ extends smoothly over $Q$ on $P_{\id}$ if and only if: $ \frac{b_{11}^+}{t}$ is odd; $\frac{b_{11}^-}{2}$ is even; $\frac{b^+_{ii}- b_{11}^+}{t}=O(t^3)$ and, for $i \neq j$, $\frac{b^+_{ij}}{t}=O(t^3)$ are odd; $\frac{b_{ii}^- - b_{ii}^+ - b_{11}^-}{2}=O(t^4)$ and, for $i \neq j$, $\frac{b_{ij}^- - b_{ij}^+}{2}=O(t^4)$ are even. Hence,  the $b_{ij}^+$  are all even and $b_{11}^+ = O(t^2)$, $b_{ii}^+ = b_{11}^+ +O(t^4)$ (so the $O(t^2)$ terms in all $b_{ii}^+$ coincide) and for $i \neq j$ we have $b_{ij}^+=O(t^4)$. Thus, the $b_{ij}^-$ must all be even, $b_{ii}^-=b_{11}^- + b_{ii}^+ + O(t^4)$ (so, up to order $O(t^4)$ the $b_{ii}^-$ do not depend on $i$) and for $i \neq j$ we have $b_{ij}^- = b_{ij}^+ +O(t^4)$. This proves the first part of Lemma \ref{lem:ExtendConnectionOnBS}.

\subsubsection*{Case {\boldmath $\mu=1$}}
 
 Here, $\Ad \circ \mu (x) q= q$, so we require homogeneous $SU(2)$-equivariant polynomials $\mathbb{H}\to\Hom(W,\mathbb{R}^3)$ where the action of $SU(2)$ on $\mathbb{R}^3$ is trivial.  This is essentially the same as the situation where the gauge group $G=U(1)$.  Therefore,  
as in that setting, we have degree 2 polynomials corresponding to $T_j\otimes dq_i$ for $i,j=1,2,3$ and degree 1 polynomials corresponding to $T_j\otimes dp_i$ for $i=0,1,2,3$ and $j=1,2,3$.

\noindent We can now consider the problem here of extending the $\mathfrak{su}(2)$-valued $1$-form $b$ in \eqref{eq:b2} over $Q$. As before, we write $b_i^{\pm}= \sum_{j=1}^3 b_{ij}^{\pm} T_j$ and deduce from 
\eqref{eq:app.coframe} that
\begin{align}\nonumber
b & =   \sum_{i,j=1}^3 \left( \frac{b_{ij}^+}{t}  T_j \otimes dp_i  +  \frac{b_{ij}^- - b_{ij}^+}{2} T_j \otimes dq_i \right).
\end{align}
Hence $b$ extends smoothly over $Q$ on $P_1$ if and only if: $ \frac{b_{ij}^+}{t}$ are odd, while the $\frac{b_{ij}^- -b_{ij}^+}{2}$ are even and must vanish at $t=0$. In other words, for all $i,j \in \lbrace 1,2,3 \rbrace$, $b_{ij}^{\pm} = O(t^2)$ and is even. This completes the proof of Lemma \ref{lem:ExtendConnectionOnBS}.

\bibliography{refsa}

\end{document}